\documentclass{emsprocart}
\usepackage{mathrsfs,tikz,relsize,hyperref}
\usepackage{amssymb,amsmath,latexsym}
\usepackage{tikz,relsize}
\usepackage{amssymb}
\usepackage{amscd}
\usepackage[all,cmtip]{xy}
\usepackage{rotating}
\usepackage{hyperref,mathrsfs}
\usepackage{makeidx}

\def\P{\mathbb{P}}
\def\I{\mathbf{I}}

\newcommand{\wis}[1]{{\text{\em \usefont{OT1}{cmtt}{m}{n} #1}}}

\def\PG{\mathbf{PG}}

\newcommand{\rosetlt}{\begin{turn}{-45}$\leadsto$\end{turn}}
\newcommand{\rosotlt}{\begin{turn}{-90}$\leadsto$\end{turn}}
\newcommand{\roswtlt}{\begin{turn}{-135}$\leadsto$\end{turn}}

\newcommand{\Spec}{\wis{Spec}}
\newcommand{\Proj}{\wis{Proj}}

\newcommand{\hL}{\mathbf{L}}
\newcommand{\hT}{\mathbf{T}}
\newcommand{\A}{\mathbb{A}}
\newcommand{\Z}{\mathbb{Z}}
\newcommand{\Fun}{\mathbb{F}_1}

\newcommand{\fK}{\mathbb{K}}
\newcommand{\mD}{\mathscr{D}}

\newcommand{\GL}{\mathbf{GL}}
\newcommand{\Aut}{\mathrm{Aut}}

\newcommand{\K}{\mathbb{K}}

\newcommand{\mC}{\mathscr{C}}

\newcommand{\mF}{\mathscr{F}}
\newcommand{\hF}{\mathbf{F}}

\newcommand{\mP}{\mathscr{P}}
\newcommand{\mB}{\mathscr{B}}
\newcommand{\mL}{\mathscr{L}}

\newcommand{\mO}{\mathscr{O}}
\newcommand{\mA}{\mathscr{A}}
\newcommand{\mG}{\mathscr{G}}

\newcommand{\F}{\mathbb{F}}

\newcommand{\fa}{\mathfrak{a}}
\newcommand{\fp}{\mathfrak{p}}
\newcommand{\id}{\mathbf{1}}

\newcommand{\hM}{\mathbf{M}}
\newcommand{\hN}{\mathbf{N}}

\def\doubleprod#1#2{\ooalign{$#1\prod$\cr$#1\coprod$\cr}}
\DeclareMathOperator*{\Rprod}{\mathpalette\doubleprod\relax}

\contact[koen.thas@gmail.com]{Koen Thas, Department of Mathematics, Ghent University, Krijgslaan 281 | S25, Ghent, Belgium}

\newtheorem{theorem}{Theorem}[section]
\newtheorem{corollary}[theorem]{Corollary}
\newtheorem{lemma}[theorem]{Lemma}
\newtheorem{proposition}[theorem]{Proposition}
\newtheorem{conjecture}[theorem]{Conjecture}

\theoremstyle{definition}

\newtheorem{remark}[theorem]{Remark}

\makeindex
\title[The combinatorial-motivic nature of $\mathbb{F}_1$-schemes]{The combinatorial-motivic nature of $\mathbb{F}_1$-schemes}

\author[Koen Thas]{Koen Thas}

\begin{document}
\setcounter{page}{87}

\begin{abstract}
We review Deitmar's theory of monoidal schemes to start with, and have a detailed look at the standard examples. It is explained how one can combinatorially study such schemes through a generalization of graph theory.
In a more general setting we then introduce the author's version of $\Fun$-schemes (called {\em $\Upsilon$-schemes} here), after which we study Grothendieck's motives in some detail in order to pass to ``absolute motives''. Throughout several considerations about absolute zeta functions are written.
In a final part of the chapter, we describe the approach of Connes and Consani to understand the ad\`{e}le class space through hyperring extension theory, in which a marvelous connection is revealed with certain group actions on projective spaces, and brand new results in the latter context are described.
Many questions are posed, conjectures are stated and speculations are made.
\end{abstract}



\maketitle

\tableofcontents

\newpage
\section{Introduction}

In the category of commutative unital rings, the tensor product $\Z \otimes_{\Z} \Z$ is isomorphic to $\Z$, so in the opposite category of affine Grothendieck schemes, 
\begin{equation}
\Spec(\Z) \times \Spec(\Z) \cong \Spec(\Z).
\end{equation}

It has been speculated that there might be larger sites than that of $\Z$-schemes, in which $\Spec(\Z)$ is not a final object anymore, such that this tensor product (over a deeper base) could become a fundamental object which behaves as a surface over the new base.

\medskip
\subsection{Numbers and polynomials}

There are many striking and deep analogies between numbers (integers) and polynomials over finite fields in one variable. 
One of the main questions we want to address in this chapter is, after passing to geometry, as to whether there exists a large site containing Grothendieck schemes, in which one can define ``absolute Descartes powers''
\begin{equation}
\Spec(\Z) \times_{\Upsilon} \Spec(\Z) \times_{\Upsilon} \cdots \times_{\Upsilon} \Spec(\Z),
\end{equation}
which would give a geometric interpretation through a generalization of the aforementioned analogy between these Cartesian powers and polynomials in multiple variables over finite fields.
We want to give a meaning to this expression in such a way that $\Spec(\Z)$ behaves as a curve over some deeper base ``$\Upsilon$'' than $\Z$

\begin{equation}
\Upsilon \longrightarrow \Z
\end{equation}
so that
$\Spec(\Z) \times_{\Upsilon} \Spec(\Z)$ becomes a surface, etc. (In much the same way as a commutative ring is an algebra over $\Z$, the latter must become an algebra over $\Upsilon$.)\\

To be more precise (or rather, more concrete), let $\mC$ be a nonsingular absolutely irreducible algebraic curve over the finite field $\F_q$; its zeta function is
\begin{equation}
\zeta_{\mC}(s) = \prod_{\fp}\frac{1}{1 - N(\fp)^{-s}},
\end{equation}
where $\fp$ runs through the closed points of $\mC$ and $N(\cdot)$ is the norm map. Fix an algebraic closure $\overline{\F_q}$ of $\F_q$ and let $m \ne 0$ be a positive integer; we have the following Lefschetz formula for the number $\vert \mC(\F_{q^m}) \vert$ of rational points over $\F_{q^m}$:

\begin{equation}
\vert \mC(\F_{q^m}) \vert = \sum_{\omega = 0}^2(-1)^{\omega}\mathrm{Tr}(\mathrm{Fr}^m \Big| H^\omega(\mC)) = 1 - \sum_{j = 0}^{2g}\lambda_j^m + q^f,
\end{equation}
where $\mathrm{Fr}$ is the Frobenius endomorphism acting on the \'{e}tale $\ell$-adic cohomology of $\mC$, the $\lambda_j$s are the eigenvalues of this action, and $g$ is the genus of the curve. It is not hard to show that we then have a ``weight decomposition''

\begin{equation}
\zeta_{\mC}(s) = \prod_{\omega = 0}^2\zeta_{h^{\omega}(\mC)}(s)^{(-1)^{\omega - 1}} 
		= \frac{\prod_{j = 1}^{2g}(1 - \lambda_jq^{-s})}{(1 - q^{-s})(1 - q^{1 - s})} \nonumber \\
\end{equation}

\begin{eqnarray} 	
	 = \frac{\mbox{\textsc{Det}}\Bigl((s\cdot\id - q^{-s}\cdot\mathrm{Fr})\Bigl| H^1(\mC)\Bigr)}{\mbox{\textsc{Det}}\Bigl((s\cdot\id - q^{-1}\cdot\mathrm{Fr})\Bigl| H^0(\mC)\Bigr.\Bigr)\mbox{\textsc{Det}}\Bigl((s\cdot\id - q^{-s}\cdot\mathrm{Fr})\Bigl| H^2(\mC)\Bigr.\Bigr)}.
	\end{eqnarray}
Here the $\omega$-weight component is the zeta function of the pure weight $\omega$ motive $h^{\omega}(\mC)$ of $\mC$.

Recalling the analogy between integers and polynomials in one variable over finite fields, 
Deninger gave a description of conditions on a conjectural category of motives that would admit a translation of Weil's proof of the Riemann Hypothesis for function fields of projective curves over finite fields $\F_q$ to the hypothetical curve $\overline{\Spec(\Z)}$. In particular, he showed that the following formula would hold:

\begin{equation}
\zeta_{\overline{\Spec(\Z)}}(s) = 2^{-1/2}\pi^{-s/2}\Gamma(\frac s2)\zeta(s)  
		=  \frac{\Rprod_\rho\frac{s - \rho}{2\pi}}{\frac{s}{2\pi}\frac{s - 1}{2\pi}} \overset{?}{=} \nonumber \\
\end{equation}
	\begin{eqnarray} 	
	 \frac{\mbox{\textsc{Det}}\Bigl(\frac 1{2\pi}(s\cdot\id - \Theta)\Bigl| H^1(\overline{\Spec(\Z)},*_{\mathrm{abs}})\Bigr.\Bigr)}{\mbox{\textsc{Det}}\Bigl(\frac 1{2\pi}(s\cdot\id -\Theta)\Bigl| H^0(\overline{\Spec(\Z)},*_{\mathrm{abs}})\Bigr.\Bigr)\mbox{\textsc{Det}}\Bigl(\frac 1{2\pi}(s\cdot\id - \Theta)\Bigl| H^2(\overline{\Spec(\Z)},*_{\mathrm{abs}})\Bigr.\Bigr)}, 
	\end{eqnarray}
where $\Rprod$ is the infinite {\em regularized product}, similarly
$\mbox{\textsc{Det}}$ denotes the {\em regularized determinant} (a determinant-like function of operators on infinite dimensional vector spaces), $\Theta$ is an ``absolute'' Frobenius endomorphism, and the $H^i(\overline{\Spec(\Z)},*_{\mathrm{abs}})$ are certain proposed cohomology groups. The $\rho$s run through the set of critical zeroes of the classical Riemann zeta. (In the formula displayed above, $\Spec(\Z)$ is compactified to $\overline{\Spec(\Z)}$ in order to see it as a projective curve.)
This description combines with Kurokawa's work on multiple zeta functions (\cite{Kurokawa1992}) from 1992 to  the hope that there are motives $h^0$ (``the absolute point''), $h^1$ and $h^2$ (``the absolute Lefschetz motive'') with zeta functions
	\begin{equation}
	\label{eqzeta}
		\zeta_{h^w}(s) \ = \ \mbox{\textsc{Det}}\Bigl(\frac 1{2\pi}(s\cdot\id-\Theta)\Bigl| H^w(\overline{\Spec(\Z)},*_{\mathrm{abs}})\Bigr.\Bigr) 
	\end{equation}
for $w=0,1,2$. Deninger computed that 
\begin{equation}
\zeta_{h^0}(s)=s/2\pi\ \ \mbox{and}\ \  \zeta_{h^2}(s)=(s-1)/2\pi. 
\end{equation}

Manin proposed in \cite{Manin} to interpret  $h^0$ as $\Spec(\Fun)$ and  $h^2$ as the affine line over $\Fun$. The search for a proof of the Riemann Hypothesis became a main motivation to look for a geometric theory over $\Fun$. 

And on the other hand, in this larger Grothendieck site of Deninger, the Cartesian product
\begin{equation}
\Spec(\Z) \times_{\Upsilon} \Spec(\Z) \times_{\Upsilon} \cdots \times_{\Upsilon} \Spec(\Z)
\end{equation}
now should make sense.\\

\medskip
\subsection{In search of a symbol}

So we are in search of the symbol $\Upsilon$, which we imagine to be, as Manin suggests in \cite{Manin}, the field with one element. In fact, we are in search of {\em Algebraic Geometry over $\Fun$} (``absolute'' Algebraic Geometry), having already developed several interesting theories over $\Fun$, as well as an initial realization of $\Fun$ itself as the multiplicative monoid $(\{ 0,1 \},\cdot)$ (cf. the first chapter of this monograph). 

In this chapter we will meet Deitmar's fundamental theory of monoidal schemes \cite{Deitmarschemes2,Deitmarschemes1,Deitmartoric}, inspired by work of Kato \cite{Kato}, in which the main algebraic objects | ``$\Fun$-rings'' so to speak | are unital monoids (usually foreseen with an absorbing element $0$). This theory works remarkably well, and several central objects such as projective spaces confirm Tits's earlier predictions of how such structures should look like when interpreted over $\Fun$. Also, zeta functions of Deitmar schemes of finite type appear to be very promising in the context of Deninger's proposed formula.

Still, it is clear that rather than being the definite scheme theory over $\Fun$, Deitmar schemes are a starting point. Several other absolute scheme theories have been defined, and in one way or the other, there always seems some functor present which descends the schemes to Deitmar schemes. One of these scheme theories is the one initiated by the author \cite{NotesI} (called {\em $\Upsilon$-schemes}), and which associates with any Deitmar scheme $S$ a category of $\Z$-schemes which descend to $S$. Some features about $\Upsilon$-schemes will be handled in the present chapter.

We will also indicate how one can naturally construct Deitmar schemes from a generalization of graphs called ``loose graphs'', and read several properties of 
Deitmar schemes from these loose graphs. For a number of important examples of Deitmar schemes, the associated loose graphs are exactly what Tits had in mind. (For instance, the loose graph of a Deitmar projective space scheme is the complete graph, and the loose graph of a Deitmar affine space scheme is a single point with a number of vertices through it which correspond to the directions of the space.)

\medskip
\subsection{Realization through zeta functions}

Although at present we do not have the right tools at hand to see $\Spec(\Z)$ as an $\Fun$-curve and $\Spec(\Z) \times_{\Fun} \Spec(\Z)$ as a surface (let alone having an intersection theory at hand on the latter object), it is at least possible to define zeta functions of 
such objects. In this chapter we will consider in some detail results of Kurokawa and Deitmar on zeta funtions in the $\Fun$-context. It is interesting to see that 
the examples of the absolute point and flag meet the Deninger-Manin predictions! 

Also, we will elaborate on Manin's vision on absolute zeta functions and absolute motives to make the story more complete. To make these considerations self-contained, we introduce the theory of motives in some detail.

\medskip
\subsection{Back to buildings}

In a final part of the chapter, we will describe observations of Connes and Consani which link $\Fun$-theory to certain action of groups on 
combinatorial geometries through the theory of hyperfield extensions (of the so-called ``Krasner hyperfield''). 
The idea is that a hyperfield extension $E$ of the Krasner hyperfield $\mathbf{K}$ yields a sharply transitive action of a certain group $G(E,\mathbf{K})$ on the 
point set of some projective space (which could be axiomatic). And the converse is also true: from such an action on a space one can construct a hyperfield
extension of $\mathbf{K}$. 

Many deep and intruiging questions arise,
some of which go back to as far as the 1920s.

\newpage
\section{Deitmar schemes}
\label{Dschem}

Several interesting attempts have been made to define schemes ``defined over $\mathbb{F}_1$'', and often the approaches only differ in
subtle variations. 
We will describe some of these viewpoints in this text. We start with the most basic one, which is the ``monoidal scheme theory'' of Anton Deitmar \cite{Deitmarschemes1}. Deitmar's theory is very important for global $\mathbb{F}_1$-theory and plays a crucial role in the present chapter, we will explain it in some detail (and add additional comments and observations that will serve us later on).

\medskip
\subsection{Rings over $\mathbb{F}_1$}

A {\em monoid}\index{monoid} is a set $A$ with a binary operation $\cdot: A \times A \longrightarrow A$ which is associative, and has a unit element ($\id$). 
Homomorphisms of monoids preserve units, and for a monoid $A$, $A^\times$\index{$A^\times$} will denote the group of invertible elements (so that if $A$ is a  group, $A^\times = A$).

\begin{theorem}[First Isomorphism Theorem for monoids]
\label{FITM}
Let $\Phi: M \longrightarrow N$ be a homomorphism of monoids. Then
\begin{equation}
M/\mathrm{ker}(\Phi) \cong \Phi(M).
\end{equation}
Here, by $M/\mathrm{ker}(\Phi)$ we mean the monoid  naturally induced on $M$ by the equivalence relation $\mathrm{ker}(\Phi) = \{ (m,m') \in M \times M \vert \Phi(m) = \Phi(m')  \}$.
\end{theorem}

In \cite{Deitmarschemes2}, Deitmar defines the category of rings over $\mathbb{F}_1$\index{$\Fun$-ring} to be the category of monoids (as thus ignoring additive structure). 

Given an $\mathbb{F}_1$-ring $A$, {\em Deitmar base extension to $\mathbb{Z}$}\index{Deitmar base extension} is defined by

\begin{equation}
A \otimes \mathbb{Z} = A \otimes_{\mathbb{F}_1} \mathbb{Z} = \mathbb{Z}[A].
\end{equation}

(Here, $\mathbb{Z}[A]$ is a ``monoidal ring'' | it is naturally defined similarly to a group ring.)
Denote the functor of base extension by $\mathscr{F}(\cdot,\otimes_{\mathbb{F}_1}\mathbb{Z})$\index{$\mathscr{F}(\cdot,\otimes_{\mathbb{F}_1}\mathbb{Z})$}.

Conversely, we have a forgetful functor $\mathscr{F}$ which maps any (commutative) ring (with unit) to its (commutative) multiplicative monoid.

\begin{theorem}[Deitmar \cite{Deitmarschemes2}]
The functor  $\mathscr{F}(\cdot,\otimes_{\mathbb{F}_1}\mathbb{Z})$ is left adjoint to $\mathscr{F}$, that is, for every ring 
$R$ and every $\mathbb{F}_1$-ring $A$ we have that
\begin{equation}
\mathrm{Hom}_{\mathrm{Rings}}(A \otimes_{\mathbb{F}_1}\mathbb{Z},R) \cong \mathrm{Hom}_{\mathbb{F}_1}(A,\mathscr{F}(R)).
\end{equation}
\end{theorem}

\medskip
\subsection{Algebraic extensions}

Let $A$ be a submonoid of the monoid $B$. An element $b \in B$ is {\em algebraic} over $A$ if there exists $n \in \mathbb{N}$ for which $b^n \in A$. The extension
$B/A$ is {\em algebraic} if every $b \in B$ is algebraic over $A$. If $B/A$ is algebraic, then $\mathbb{Z}[B]/\mathbb{Z}[A]$ is an algebraic ring extension, but 
the converse is not necessarily true. 
An algebraic extension $B/A$ is {\em strictly algebraic} if for every $a \in A$ the equation $x^n = a$ has at most $n$ solutions in $B$.

A monoid $A$ is {\em algebraically closed} if every equation of the form $x^n = a$ with $a \in A$ and $n \in \mathbb{N}$ has a solution in $A$. Every monoid can be 
embedded into an algebraically closed monoid, and if $A$ is a group, then there exists a ``smallest'' such embedding which is called the {\em algebraic closure} of $A$.\\

\medskip
\subsection{Important example}

The algebraic closure $\overline{\mathbb{F}_1}$\index{$\overline{\mathbb{F}_1}$} of $\mathbb{F}_1$ is the group $\mu_{\infty}$ of all complex roots of unity; it is the torsion group of the circle group $\mathbb{T} = \{e^{i\theta} \vert \theta \in [ 0,2\pi )\}, \times \cong \mathbb{R} \oplus \mathbb{Q}/\mathbb{Z}$, so  
it is isomorphic to $\mathbb{Q}/\mathbb{Z}$. Note that the multiplicative group $\overline{\mathbb{F}_p}^\times$ of the algebraic closure $\overline{\mathbb{F}_p}$ of 
the prime field $\mathbb{F}_p$ is isomorphic to the group of all complex roots of unity of order prime to $p$, so that the definition of $\overline{\mathbb{F}_1}$ is in accordance with the finite field case.

\medskip
\subsection{Localization}

Let $S$ be a submonoid of the monoid $A$. We define the monoid $S^{-1}A$\index{$S^{-1}A$}, the {\em localization}\index{localization} of $A$ {\em by} $S$, to be $A \times S/\sim$, where the equivalence relation ``$\sim$'' is given by

\begin{equation}
(a,s) \sim (a',s') \ \ \mathrm{if\ and\ only\ if}\ \ s''s'a = s''sa'\ \ \mathrm{for\ some}\ \ s'' \in S.
\end{equation}

Multiplication in $S^{-1}A$ is componentwise, and one suggestively writes $\frac{a}{s}$ for the element in $S^{-1}A$ corresponding to $(a,s)$ ($\frac{a}{s}\frac{a'}{s'} = \frac{aa'}{ss'}$).

\medskip
\subsection{Ideal and spectrum}

If $C$ and $D$ are subsets of the monoid $A$, $CD$\index{$CD$} denotes the set of products $cd$, with $c \in C$ and $d \in D$. 

In this paragraph, and in fact throughout, a ring is always commutative with unit and any monoid is also supposed to be abelian. 

An {\em ideal}\index{ideal} $\fa$ of a monoid $M$ is a
subset such that $M\fa \subseteq \fa$. For any ideal $\fa$ in $M$, $\mathbb{Z}[\fa]$ is an ideal in $\mathbb{Z}[M]$. Note that if $A$ and $B$
are monoids and $\alpha: A \longrightarrow B$ is a morphism, then $\alpha^{-1}(\fa)$ is an ideal in $A$ if $\fa$ is an ideal in $B$. 
If $S$ is a subset of the monoid $A$, $\langle S \rangle$\index{$\langle S\rangle$} denotes the ideal {\em generated by $S$} (i.e., it is the smallest ideal containing $S$).

An ideal $\fp$  is called a {\em prime ideal}\index{prime!ideal}  if $S_{\fp} := M \setminus \fp$\index{$S_{\fp}$} is a monoid (that is, if $uv \in \fp$, then $u \in \fp$ or $v \in \fp$).
For any prime ideal $\fp$ in $M$, denote by 
\begin{equation}
M_{\fp} = S_{\fp}^{-1}M 
\end{equation}
the {\em localization}\index{$M_{\fp}$}\index{localization} of $M$ {\em at} $\fp$. 

\begin{proposition}[\cite{Deitmarschemes1}]
The natural map 
\begin{equation}
M \longrightarrow M_{\fp},\ m \longrightarrow \frac{m}{1}
\end{equation} 
with $\fp = M \setminus M^{\times}$ is an isomorphism.
\end{proposition}

Let $M$ be a monoid. The {\em spectrum}\index{spectrum} $\Spec(M)$\index{$\Spec(M)$} of $M$ is the set of prime ideals endowed with the obvious Zariski topology. 
(Note that the spectrum cannot be empty since $M \setminus M^{\times}$ is a prime ideal.)
The closed subsets are the empty set and all sets of the form

\begin{equation}
V(\fa) := \{ \fp \in \Spec(M) \vert \fa \subseteq \fp\},
\end{equation}\index{$V(\fa)$}
where $\fa$ is any ideal.
The point $\eta = \emptyset$ is contained in every nonempty open set and the point $M \setminus M^{\times}$ is closed and contained in every nonempty closed set.
Note also that for every $m \in M$ the set $V(m) :=  \{ \fp \in \Spec(M) \vert m \in \fp\}$ is closed (as $V(m) = V(Mm)$).

\begin{proposition}
$M \setminus M^{\times}$ is the unique maximal  ideal for any monoid $M$, so any such $M$ is a local $\mathbb{F}_1$-ring.
\end{proposition}

\medskip
\subsection{Structure sheaf}

Let $A$ be a ring over $\mathbb{F}_1$. For any open set $U \subseteq \Spec(A)$, one defines $\mO_{\Spec(A)}(U) = \mO(U)$ to be the set of functions
(called {\em sections}\index{section})

\begin{equation}
s: U \longrightarrow \coprod_{\fp \in U}A_{\fp}
\end{equation}
for which $s(\fp) \in A_{\fp}$ for each $\fp \in U$, and such that there exists a neighborhood $V$ of $\fp$ in $U$, and elements $a, b \in A$, for which
$b \not\in \mathfrak{q}$ for every $\mathfrak{q} \in V$, and $s(\mathfrak{q}) = \frac{a}{b}$ in $A_{\mathfrak{q}}$. The map 

\begin{equation}
\mO_{\Spec(A)}: \Spec(A) \longrightarrow \mbox{monoids}: U \longrightarrow \mO(U)
\end{equation}
is the {\em structure sheaf}\index{structure sheaf} of $\Spec(A)$.

\begin{proposition}[\cite{Deitmarschemes1}]
\label{globsec}
\begin{itemize}
\item[{\rm (i)}]
For each $\fp \in \Spec(A)$, the stalk $\mO_{\fp}$ of the structure sheaf is isomorphic to the localization of $A$ at $\fp$.
\item[{\rm (ii)}]
for global sections, we have $\Gamma(\Spec(A),\mO) := \mO(\Spec(A)) \cong A$.
\end{itemize}
\end{proposition}

\medskip
\subsection{Monoidal spaces}

A {\em monoidal space}\index{monoidal space} is a topological space $X$ together with a sheaf of monoids $\mO_X$. 
Call a morphism of monoids $\beta: A \longrightarrow B$ {\em local}\index{local morphism} if $\beta^{-1}(B^{\times}) = A^{\times}$. In particular, monoidal epimorphisms and  isomorphisms and automorphisms are always local.
A {\em morphism}\index{morphism!of monoidal spaces} between monoidal spaces $(X,\mO_X)$ and $(Y,\mO_Y)$ is defined naturally: it is a pair $(f,f^\#)$ with $f: X \longrightarrow Y$ a 
continuous function, and 

\begin{equation}
f^\#: \mO_Y \longrightarrow f_*\mO_X
\end{equation}
a morphism between sheaves of monoids on $Y$. (Here, $f_*\mO_X$\index{$f_*\mO_X$}, the {\em direct image sheaf}\index{direct image sheaf} on $Y$ induced by $f$, is defined by $f_*\mO_X(U) := \mO_X(f^{-1}(U))$ for all open $U \subseteq Y$.)
The morphism is {\em local}\index{local morphism} if each of the induced morphisms $f^\#_x: \mO_{Y,f(x)} \longrightarrow \mO_{X,x}$ is local.

\begin{proposition}[\cite{Deitmarschemes1}]
\label{propmor}
\begin{itemize}
\item[{\rm (i)}]
For an $\mathbb{F}_1$-ring $A$, we have that $(\Spec(A),\mO_A)$ is a monoidal space.
\item[{\rm (ii)}]
If $\alpha: A \longrightarrow B$ is a morphism of monoids, then $\alpha$ induces a morphism of monoidal spaces
\begin{equation}
(f,f^{\#}): \Spec(B) \longrightarrow \Spec(A),
\end{equation}
yielding a functorial bijection
\begin{equation}
\mathrm{Hom}(A,B) \cong \mathrm{Hom}_{\mathrm{loc}}(\Spec(B),\Spec(A)),
\end{equation}
where on the right hand side we only consider local morphisms (hence the notation).
\end{itemize}
\end{proposition}

We will need the following converse of (ii):

\begin{proposition}
\label{isomloc}
Any local morphism of monoidal spaces $(f,f^\#)$: $\Spec(B) \longrightarrow \Spec(A)$ is induced by a monoidal
morphism $\alpha = \alpha_{(f,f^\#)}$ as in {\em Proposition \ref{propmor}(ii)}.
\end{proposition}
\begin{proof}
Let $(f,f^\#)$ be as in the statement of the theorem; then taking global sections, $f^\#$ induces a morphism $\phi: 
\Gamma(\Spec(A),\mO) \longrightarrow \Gamma(\Spec(B),\mO)$, which by Proposition \ref{globsec} is a morphism $\phi: A \longrightarrow B$. 
For any $\fp \in \Spec(B)$, we have a local morphism $f^\#_{\fp}: A_{f(\fp)} \longrightarrow B_{\fp}$ such that the following diagram commutes:

\begin{equation}
\begin{array}{ccc}
A &\xrightarrow{\phi} &B\\
&&\\
\downarrow & &\downarrow \\
&&\\
A_{f(\fp)} & \xrightarrow{f^{\#}_{\fp}} &B_{\fp}
\end{array}
\end{equation}

As $f^{\#}$ is a local homomorphism, we have that $\phi^{-1}(\fp) = f(\fp)$, so that $f$ coincides with the map $\Spec(B) \longrightarrow \Spec(A)$ induced by $\phi$. It follows readily that the monoid homomorphism  $\phi = \alpha_{(f,f^\#)}$ induces $(f,f^\#)$.
\end{proof}

Since any automorphism is local, we have the following implication.

\begin{corollary}
\label{corisomloc}
If $(f,f^\#) \in \Aut(\Spec(A))$ is such that $f = \id$ implies that $\alpha_{(f,f^\#)} = \id$, then $f^\#$ also is trivial.
\end{corollary}

So if the topology of $\Spec(A)$ is sufficiently fine, the only element in its automorphism group $\Aut(\Spec(A))$\index{$\Aut(\Spec(A))$} with trivial component $f$, is the trivial one.

\medskip
\subsection{Deitmar's $\mathbb{F}_1$-schemes}

As in the theory of rings, we have defined a structure sheaf $\mO_X$ on the topological space $X = \Spec(M)$. One then defines a {\em scheme}
over $\mathbb{F}_1$ to be a topological space together with a sheaf of monoids, locally isomorphic to spectra of monoids in the above sense. The details are below.

An {\em affine scheme}\index{affine!$\Fun$-scheme} over $\mathbb{F}_1$ is a monoidal space which is isomorphic to $\Spec(A)$ for some monoid $A$. For the rest of the chapter, we will call such schemes {\em affine Deitmar schemes}\index{affine!Deitmar scheme} or also {\em $\mD$-schemes}\index{affine!$\mD$-scheme}
or {\em $\mD_0$-schemes}\index{affine!$\mD_{0}$-scheme}. (The ``$\mD$'' obviously stands for ``Deitmar''; sometimes we add the sub-index $0$ to stress that monoids have a zero in this context.)
A monoidal space $X$ is a {\em scheme}\index{$\Fun$-scheme} over $\mathbb{F}_1$ if for every point $x \in X$ there is an open neighborhood $U \subseteq X$ such that 
$(U,\mO_{X \vert U})$ is an affine scheme over $\mathbb{F}_1$. As in the affine case, we also speak of {\em $\mD$-schemes}\index{$\mD$-scheme} and {\em $\mD_0$-schemes}\index{$\mD_0$-scheme}.
A {\em morphism}\index{morphism!of $\mD_{(0)}$-schemes} of $\mD_{(0)}$-schemes is a local morphism of monoidal spaces. 

Recall that a point $\eta$ of a topological space is a {\em generic point}\index{generic point} if it is contained in every nonempty open set.

\begin{proposition}[\cite{Deitmarschemes1}]
Any connected $\mD_0$-scheme has a unique generic point $\emptyset$, and morphisms between connected schemes map generic points to generic points.  
\end{proposition}

As a corollary we have:

\begin{proposition}[\cite{Deitmarschemes1}]
For an arbitrary $\mD_0$-scheme $X$, $\mathrm{Hom}(\Spec(\mathbb{F}_1),X)$ can be identified with the set of connected components of $X$.
\end{proposition}

\medskip
\subsection{Schemes of finite type}

One obtains a functor

\begin{equation}
X \longrightarrow X_{\mathbb{Z}}
\end{equation}
from $\mD_0$-schemes to $\mathbb{Z}$-schemes, thus extending the base change functor $\mF(\cdot,\otimes_{\mathbb{F}_1}\mathbb{Z})$\index{$\mF(\cdot,\otimes_{\mathbb{F}_1}\mathbb{Z})$},
in the following way. 
One writes a scheme $X$ over $\mathbb{F}_1$ as a union of affine $\mD_0$-schemes, $X = \cup_i \Spec(A_i)$, and then map it to
$\cup_i\Spec(A_i \otimes_{\mathbb{F}_1} \mathbb{Z})$ (glued via the gluing maps from $X$).

We say that the $\mD_0$-scheme $X$ is {\em of finite type}\index{finite!type $\mD_0$-scheme} if it has a finite covering by affine schemes $U_i = \Spec(A_i)$ such that the
$A_i$ are finitely generated. 

\begin{proposition}[\cite{Deitmarschemes1}]
$X$ is of finite type over $\mathbb{F}_1$ if and only if $X_{\mathbb{Z}}$ is a $\mathbb{Z}$-scheme of finite type.
\end{proposition}

Conversely, as mentioned in the previous section, one has a functor from monoids to rings, and it is left adjoint to the 
forgetful functor that sends a ring $R$ to the multiplicative monoid $(R, \times)$. 
A scheme $X$ over $\mathbb{Z}$ can be written as a union of affine schemes 

\begin{equation}
X = \cup_{i}\Spec(A_i)
\end{equation}
for some set of rings $\{A_i\}$. Then map $X$ to $\cup_i\Spec(A_i,\times)$ (using the gluing maps from $X$) to obtain a functor from schemes over $\mathbb{Z}$ to schemes over $\mathbb{F}_1$ which extends the aforementioned forgetful functor.

\medskip
\subsection{Structure of $\mD_0$-schemes (of finite type)}

Let $A$ be a monoid (abelian, with unit), and $\mathrm{Quot}(A)$\index{$\mathrm{Quot}(A)$} its {\em quotient group}\index{quotient group} or {\em Grothendieck group}\index{Grothendieck!group} | namely the localization of $A$ by $A$.
So $\mathrm{Quot}(A)$ coincides with the stalk of $A_\eta$ of $\Spec(A)$ at the generic point $\eta = \emptyset$.
Every morphism from $A$ to a group factorizes uniquely over the natural morphism

\begin{equation}
A \longrightarrow \mathrm{Quot}(A).
\end{equation}

If $A$ is a finitely generated monoid, the point set of its spectrum is a finite set. So the underlying $\mD_0$-scheme $X$ (which is of finite type)
is a finite set. 

A monoid $A$ is {\em integral}\index{integral!monoid} if it has the cancellation property, that is, if $ab = ac$ implies $b = c$. (This is equivalent to requiring that $A$ injects into
$\mathrm{Quot}(A)$.)

\medskip
\subsection{Modules}

A {\em module}\index{module} of a monoid $A$ is a set $S$ together with a map (action)

\begin{equation}
A \times S \longrightarrow S: (a,s) \longrightarrow as
\end{equation}
such that $\id s = s$ for all $s$ and $(ab)s = a(bs)$. A point of a module is {\em stationary}\index{stationary point} if it is invariant under the aforementioned map for any $a \in A$.
A {\em pointed module}\index{pointed module} is a pair $(S,s_0)$ consisting of an $A$-module $S$ and a stationary point $s_0 \in S$.

The {\em tensor product}\index{tensor product} $M \otimes N$\index{$M \otimes N$} of two $A$-modules $M, N$ is 
\begin{equation}
M \otimes N = M \otimes_A N = M \times N/\sim,
\end{equation}
where ``$\sim$'' is the equivalence relation generated by $(am,n) \sim (m,an)$, for every $a \in A$, $m \in M$, $n \in N$.
The class of $(m,n)$ is written as $m \otimes n$. It then is clear that $M \otimes N$ becomes an $A$-module via
\begin{equation}
a(m \otimes n) = (am)\otimes n = m \otimes (an).
\end{equation}

\subsection*{Example}

The $A$-module $A \otimes M$ is isomorphic to the $A$-module $M$ through the map $A \otimes M \longrightarrow M: a \otimes m \longrightarrow am$.\\

If $(M,m_0)$ and $(N,n_0)$ are pointed $A$-modules, the ``pointed tensor product'' is $(M \otimes N,m_0 \otimes n_0)$.\\

A pointed module is {\em flat}\index{flat!pointed module} if and only if for every injection $M \hookrightarrow N$ of pointed modules the induced map 
\begin{equation}
F \otimes M \longrightarrow F \otimes N
\end{equation}
is also an injection.
A morphism of monoids $\phi: A \longrightarrow B$ is {\em flat}\index{flat!morphism} if $B$ is flat as an $A$-module. A morphism of $\mD_0$-schemes $f: X \longrightarrow Y$ is {\em flat}\index{flat!morphism}
if for every $x \in X$ the morphism of monoids $f^{\#}: \mO_{Y,f(x)} \longrightarrow \mO_{X,x}$ is flat.

\begin{proposition}[\cite{Deitmartoric}]
\begin{itemize}
\item
A morphism of monoids $\phi: A \longrightarrow B$ is flat if and only if the induced morphism of $\mD_0$-schemes $\Spec(B) \longrightarrow \Spec(A)$ is flat.
\item
The composition of flat morphisms is flat.
\item
Base change of flat morphisms by an arbitrary morphism is flat.
\end{itemize}
\end{proposition}

\medskip
\subsection{Cohomology}

In \cite{Deitmartoric}, the author states that  (sheaf) cohomology over $\mathbb{F}_1$ is not defined. He provides an example of a topological space consisting of three points, endowed with sheaves of abelian groups $\mF$ and $\mG$, and a ``flip'' map, such that if the sheaves were defined over $\mathbb{F}_1$ (which they are), the flip also should be as such, which is not the case.
In that same example, there even are different injective resolutions which produce different cohomology groups.

This indicates that at least  on the level of cohomology, one needs a different approach.
We refer to Anton Deitmar's chapter in the present volume for (much) more on this matter.

\medskip
\subsection{Fibre products}

Let $S$ be a scheme over $\mathbb{F}_1$. Then the $\mD_0$-scheme $X$ is a {\em scheme over $S$}\index{$\mD_0$-scheme!over $S$} if there exists a morphism $X \longrightarrow S$. 

\begin{proposition}
Let $X$ and $Y$ be $\mD_0$-schemes over the $\mD_0$-scheme $S$. Then there exists a scheme $X \times_S Y$ over $S$, the {\em fibre product}\index{fibre product} of $X \longrightarrow S$ and $Y \longrightarrow S$,
unique up to $S$-isomorphism, and 
morphisms from $X \times_S Y$ to $X$ and $Y$ such that the diagram below is commutative: 

\begin{equation}
\begin{array}{ccc}
X \times_S Y &\longrightarrow &X\\
&&\\
\downarrow &\searrow &\downarrow \\
&&\\
Y & \longrightarrow &S
\end{array}
\end{equation}
and such that these morphisms induce a bijection 
\begin{equation}
\mathrm{Hom}_S(Z,X) \times \mathrm{Hom}_S(Z,Y) \longrightarrow \mathrm{Hom}_S(Z,X \times_S Y)
\end{equation}
for every scheme $Z$ over $S$. Moreover, the fibre product is compatible with extension to $\mathbb{Z}$ and the usual fibre product for schemes:

\begin{equation}
(X \times_S Y) \otimes_{\mathbb{F}_1} \mathbb{Z} \cong (X \otimes_{\mathbb{F}_1} \mathbb{Z}) \times_{S \otimes_{\mathbb{F}_1} \mathbb{Z}} (Y \otimes_{\mathbb{F}_1} \mathbb{Z}).
\end{equation}
\end{proposition}

\medskip
\subsection{\'{E}tale morphisms}

Let $A$ be a monoid, and $m_A = A \setminus A^{\times}$ its maximal ideal. Then a homomorphism $\phi: A \longrightarrow B$ is {\em local}\index{local!morphism} if and only if $m_A^{\phi} \subseteq m_B$. Such a local homomorphism is {\em unramified}\index{unramified} if both of the following conditions are satisfied
\begin{itemize}
\item[(i)]
$m_A^{\phi}B = m_B$, and
\item[(ii)]
$\phi$ injects $A^\times$ into $B^\times$, and $B/A^{\phi}$ is a finite strictly algebraic extension. 
\end{itemize}

If $\phi$ is unramified, then so are all localizations $\phi_{\fp}: A_{\fp^{\phi^{-1}}} \longrightarrow B_{\fp}$ for $\fp \in \Spec(B)$.

A morphism $f: X \longrightarrow Y$ of $\mD_0$-schemes is called {\em unramified}\index{unramified} if for every $x \in X$ the local morphism 
\begin{equation}
f^{\#}: \mO_{Y,f(x)} \longrightarrow \mO_{X,x} 
\end{equation}
is unramified. 

A morphism $f: X \longrightarrow Y$ of $\mD_0$-schemes is {\em locally} of {\em finite type}\index{morphism!locally of finite type} if every point $y \in Y$ has an open affine neighborhood $V =  \Spec(A)$ such that $f^{-1}(V)$ is a union of open affine $\Spec(B_i)$s with $B_i$ finitely generated as a monoid over $A$.
The morphism is then of {\em finite type}\index{morphism!of finite type} if for every point $y \in Y$ the number of $B_i$s can be chosen to be finite.
A morphism $f: X \longrightarrow Y$ of finite type is {\em \'{e}tale}\index{etale@\'{e}tale!morphism} if $f$ is flat and unramified. It is an {\em \'{e}tale covering}\index{etale@\'{e}tale!covering} if it is also finite.

\begin{theorem}[\cite{Deitmartoric}]
The \'{e}tale coverings of $\Spec(\mathbb{F}_1)$ are the morphisms of the form $\Spec(A) \longrightarrow \Spec(\mathbb{F}_1)$, where $A$ is a  finite cyclic group. 
\end{theorem}

\begin{theorem}[\cite{Deitmartoric}]
The scheme $\Spec(\overline{\mathbb{F}_1})$ has no nontrivial \'{e}tale coverings.
\end{theorem}

A connected scheme over $\mathbb{F}_1$ which has only the trivial \'{e}tale covering is said to be  {\em simply connected}\index{simply connected}.

\medskip
\subsection{Toric varieties}
\label{Dtoric}

A {\em toric variety}\index{toric variety} is an irreducible variety $V$ over the field of complex numbers $\mathbb{C}$ together with an algebraic action of the $r$-dimensional torus
$\GL_1^r$, such that $V$ contains an open orbit. Every toric variety is the lift $X_{\mathbb{C}}$ of a $\mD_0$-scheme $X$. 

The next theorem, which is due to Deitmar, obtains the converse; essentially it shows that integral $\mD_0$-schemes of finite type are the same as toric varieties.

\begin{theorem}[\cite{Deitmartoric}]
Let $X$ be a connected integral $\mD_0$-scheme of finite type. Then every irreducible component of $X_\mathbb{C}$ is a toric variety. The components of 
$X_{\mathbb{C}}$ are mutually isomorphic as toric varieties.
\end{theorem}

This means that Deitmar's base extension to $\mathbb{Z}$ is basically too tight: the addition which appears after applying the functor $\mF(\cdot,\otimes_{\mathbb{F}_1}\mathbb{Z})$ leads us to toric varieties, and we want to have a much bigger set of schemes at hand after base extension. 
Later in this chapter, we will define the category of {\em $\Upsilon$-schemes}\index{$\Upsilon$-scheme} to fulfil this desire; we also refer 
to Lorscheid's chapter for his theory of ``blueprints''\index{blueprint}.



\newpage
\section{Fundamental examples | affine and projective spaces}

 \subsection{$\Spec(\mathbb{F}_1)$ | The absolute point}
 
  The spectrum of $\Spec(\mathbb{F}_1)$ consists of precisely one point, namely, the 
unique prime ideal $\{0\}$, which coincides with the unique closed point. The stalk at $\{0\}$ is equal to $\mathbb{F}_1$. 
The base 
$\mD_0$-extension to $\mathbb{Z}$ is $\Spec(\mathbb{Z})$. The $\mD_0$-scheme $\Spec(\mathbb{F}_1)$ is a terminal object in the 
category of $\mD_0$-schemes.

\medskip
\subsection{Polynomial rings}

Define
\begin{equation}
\mathbb{F}_1[X_1,\ldots,X_n] = \{0\} \cup \{ X_1^{u_1}\ldots X_n^{u_n} \vert u_j \in \mathbb{N}\}, 
\end{equation}\index{$\mathbb{F}_1[X_1,\ldots,X_n]$}
that is, the union of $\{0\}$ and the (abelian) monoid generated by the $X_j$. In other words, 
\begin{equation}
\mathbb{F}_1[X_1,\ldots,X_n]  = \mathbb{F}_1[\langle X_1,\ldots,X_n\rangle_{\mathrm{ab}}],
\end{equation}
where $\langle X_1,\ldots,X_n\rangle_{\mathrm{ab}}$ is the free abelian group generated by the letters $X_1,\ldots,X_n$.

\medskip
\subsection{Groups}

Let $G$ be a group and put
$A = \mathbb{F}_1[G] := \{0\} \cup G$.
Then $\Spec(A)$ 
consists of the unique prime ideal $\{0\}$ of $A$ and the stalk at $\{0\}$ is $A$.  
Base $\mD_0$-extension to $\mathbb{Z}$ is 
\begin{equation}
\Spec(A)_\mathbb{Z} = \Spec(\mathbb{Z}[G]). 
\end{equation}

In particular, if $G$ is a free abelian group on $n$ generators $X_1,\ldots,X_n$, then 
\begin{equation}
A = \mathbb{F}_1[X_1^{\pm 1},\ldots,X_n^{\pm 1}],\ \ A_\mathbb{Z} = \mathbb{Z}[X_1^{\pm 1},\ldots,X_n^{\pm 1}],  
\end{equation}
and thus $\Spec(A)_\mathbb{Z} \cong \mathbb{G}^n_m$. 
So in accordance, in this case we will denote $\Spec(A)$ by $\mathbb{G}^n_{m\vert\mathbb{F}_1}$\index{$\mathbb{G}^n_{m\vert\mathbb{F}_1}$}. (If $n = 1$, we omit the supscript.) \\

\medskip
\subsection{Affine space}

Let $A=\Fun[X_1,\dotsc,X_n]$; then $A_\Z = \Z[X_1,\dotsc,X_n]$ and thus $\Spec(A)_\Z \equiv \A^n$. Denote $\Spec(\mathrm{\Fun[X_1,\dotsc,X_n]})$ by $\A_{\Fun}^n$ and call it the \emph{$n$-dimensional affine space over $\Fun$}\index{affine!space over $\Fun$}. The $\ne (0)$ prime ideals of $A$ are of the form 

\begin{equation}
\fp_I=\bigcup_{i\in I}X_iA, 
\end{equation}
where $I$ is a subset of $\{1,\dotsc,n\}$ and $X_iA=\{X_ia\mid a\in A\}$. The stalk of the structure sheaf at $\fp_I$ is the localization of $A$ at the multiplicative set $S$ that contains all products of elements $X_j$ where $j\notin I$.\\

\medskip
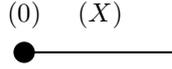
\subsection{The absolute flag}

An example which deserves separate mention (certainly in a context of combinatorial geometry) is the {\em absolute flag}\index{absolute!flag} | the $\mD_0$-scheme $\Spec(\mathbb{F}_1[X])$; it consists of one closed point and one (different) generic point.

\bigskip
\begin{center}
\item
\begin{tikzpicture}[style=thick, scale=2]
\foreach \x in {1}{
\fill (\x,0) circle (2pt);}

\draw (1,0) -- (2,0);
\draw (1.5,.25) node {$(X)$};
\draw (1,.25) node {$(0)$};

\end{tikzpicture}
\item
The absolute flag
\end{center}
\hspace{0.5cm} 

At the $\mD_0$-level, all affine varieties will consist of a number of absolute flags (while the much more complex projective varieties will be built up out of 
projective lines over $\mathbb{F}_1$ | cf. the next paragraph, and the next section).\\

\medskip
\subsection{Proj-schemes}

In \cite{NotesI} we introduced the Proj-scheme construction for Deitmar schemes. We quickly repeat this procedure.

 \medskip
 \subsubsection{Monoid quotients}
 
 Let $M$ be a commutative unital monoid (with $0$), and $I$ an ideal of $M$. We define the monoidal quotient $M/I$ to be the set $\{ [m] \in M \vert m \in M \}/([m] = [0] \mbox{if}\ m \in I)$.
 (When $R$ is a commutative ring and $J$ an ideal, then the ring quotient $R/J$ induces the monoidal quotient on $R, \times$.)
 
 \subsubsection{The $\Proj$-construction}
 
 Consider the $\mathbb{F}_1$-ring $\mathbb{F}_1[X_0,X_1,\ldots,X_m]$, where $m \in \mathbb{N}$. Since any polynomial is 
 homogeneous in this ring, we have a natural grading
 
 \begin{equation}
 \mathbb{F}_1[X_0,\ldots,X_m] = \bigoplus_{i \geq 0}R_i,
 \end{equation}
 where $R_i$ consists of the elements of $\mathbb{F}_1[X_0,X_1,\ldots,X_m]$ of total degree $i$, for $i \in \mathbb{N}$. 
 The {\em irrelevant ideal}\index{irrelevant ideal}\index{$\mathrm{Irr}$} is 
 
 \begin{equation}
\mathrm{Irr} = \{0\} \cup \bigoplus_{i \geq 1}R_i.
 \end{equation}
 
 Now $\Proj(\mathbb{F}_1[X_0,\ldots,X_m]) =: \Proj(\mathbb{F}_1[\mathbf{X}])$\index{$\Proj(\mathbb{F}_1[\mathbf{X}])$} consists, as a set, of the prime ideals of  $\mathbb{F}_1[X_0,X_1,\ldots,X_m]$
 which do not contain $\mathrm{Irr}$ (so only $\mathrm{Irr}$ is left out of the complete set of prime ideals). 
 The closed sets of the (Zariski) topology on this set are defined as usual: for any ideal $I$ of  $\mathbb{F}_1[X_0,X_1,\ldots,X_m]$, we define
 
 \begin{equation}
 V(I) := \{ \fp \vert \fp\in \Proj(\mathbb{F}_1[\mathbf{X})],\ \ I \subseteq \fp \},
 \end{equation}
 where $V(I) = \emptyset$ if $I = \mathrm{Irr}$ and $V(\{0\}) = \Proj(\mathbb{F}_1(\mathbf{X}))$,
 the open sets then being of the form
 
  \begin{equation}
 D(I) := \{ \fp \vert \fp\in \Proj(\mathbb{F}_1[\mathbf{X})],\ \ I \not\subseteq \fp \}.
 \end{equation}
 
 It is obvious that $\Proj(\mathbb{F}_1[\mathbf{X}])$ is a $\mD_0$-scheme.
 (The structure sheaf is described below in a more general setting.)
  Each ideal $(X_i)$ defines an open set $D((X_i))$ such that the restriction of the scheme to this set is isomorphic to $\Spec(\mathbb{F}_1[\mathbf{X}_{(i)}])$, where $\mathbf{X}_{(i)}$ is $X_0,X_1,\ldots,X_m$ with $X_i$ left out.\\

 Suppose $M$ is a commutative unital monoid (with $0$) with a  grading
 
 \begin{equation}
M = \bigoplus_{i \geq 0}M_i,
 \end{equation}
 where the $M_i$ are the sets with elements of total degree $i$, for $i \in \mathbb{N}$, and let, as above, the
{\em irrelevant ideal}\index{irrelevant ideal} be $\mathrm{Irr} = \{0\} \cup \bigoplus_{i \geq 1}M_i$\index{$\mathrm{Irr}$}. Define the topology $\Proj(M)$\index{$\Proj(M)$} as before (noting that homogeneous (prime) ideals are the same as ordinary monoidal (prime) ideals here). For an open 
$U$, define $\mO_M(U)$ as consisting of all functions
 \begin{equation}
f: U \longrightarrow \coprod_{\fp \in U}M_{(\fp)},
\end{equation}
where $M_{(\fp)}$ is the subset of $M_{\fp}$ of fractions of elements with the same degree, 
for which $f(\fp) \in M_{(\fp)}$ for each $\fp \in U$, and such that there exists a neighborhood $V$ of $\fp$ in $U$, and  elements $u, v \in M$, for which
$v \not\in \mathfrak{q}$ for every $\mathfrak{q} \in V$, and $f(\mathfrak{q}) = \frac{u}{v}$ in $M_{(\mathfrak{q})}$.

In this way we obtain a sheaf of $\mathbb{F}_1$-rings on $\Proj(M)$ making it a $\mD_0$-scheme.

\medskip
\subsection{Notes on dimension}

 We will need the definition of inifinite dimensional projective spaces (over $\mathbb{F}_1$) for later purposes. From the incidence geometrical point of view, these can be seen as complete graphs on a set $\Omega$, with $\vert \Omega\vert$ the required dimension, endowed with the natural  induced subspace structure. Here $\Omega$ need not be countable. We want to formally see this in schematic language, in the spirit of the previous paragraphs. The definition boils down to an Ind-scheme construction.\\
 
So let $\Omega$ be any infinite set of cardinality $\omega$; we define {\em projective space of dimension $\omega$}\index{projective space} to be 

\begin{equation}
\Proj(\mathbb{F}_1[X_i]_{i \in \Omega}).
\end{equation} 

Again,  we have a natural grading
 
 \begin{equation}
 \mathbb{F}_1[X_i]_{i \in \Omega} = \bigoplus_{i \geq 0}R_i,
 \end{equation}
 where $R_i$ consists of the elements of $\mathbb{F}_1[X_i]_{i \in \Omega}$ of total degree $i$, for $i \in \mathbb{N}$,
 and the irrelevant ideal is by definition
 
 \begin{equation}
\mathrm{Irr} = \bigoplus_{i \geq 1}R_i.
 \end{equation}
 
As a set, $\Proj(\mathbb{F}_1[X_i]_{i \in \Omega})$ consists of the prime ideals of  $\mathbb{F}_1[X_i]_{i \in \Omega}$
 which do not contain $\mathrm{Irr}$, and the closed/open sets are analogous to those described in the finite dimensional case.
 In $\mathbb{F}_1[X_i]_{i \in \Omega} = R$, prime ideals are ideals $\fp$ for which $R \setminus \fp$ are multiplicative, so they are of the form
 
 \begin{equation}
 \fp = \cup_{j \in J}RX_j
\end{equation}
with $J \subseteq \Omega$. In $\Proj(\mathbb{F}_1[X_i]_{i \in \Omega})$ we do not allow the prime ideal $R \setminus R^{\times} = \mathrm{Irr}$.

 The closed points of $\Proj(\mathbb{F}_1[\mathbf{X}])$ correspond to the next-to-maximal ideals which are of the form 
 \begin{equation}
 \cup_{i \in J}RX_i =: R(j),
 \end{equation}
 where $\vert I \setminus J \vert = 1$.
 More generally, linear subspaces can again be seen as follows. Let $S$ be any set of closed points, corresponding to the elements of the set $K \subseteq \Omega$. The ideal 
 
 \begin{equation}
 \cup_{i \not\in K}MX_i = D
 \end{equation}
 corresponds to a closed set $V(D)$ of the topology on which the coordinate ring is 
 
 \begin{equation}
 \mathbb{F}_1[X_i]_{i \in \Omega}/D,
 \end{equation}
 which can be identified with $\mathbb{F}_1[X_i]_{i \in I \setminus K}$. Now the induced scheme  
is the projective space over $\mathbb{F}_1$ of dimension $\vert K\vert$. 

\begin{remark}[On dimension]{\rm
Essentially, there is no problem with having $\vert I \setminus K\vert$ infinite: in infinite dimensional combinatorial projective space (over any field), the linear spaces of infinite dimension and infinite co-dimension can only be defined by an infinite number of linear equations over that field. From the point of Algebraic Geometry however, one only considers 
a finite number of linear equations in this context (due to the fact many basic properties are lost if one allows the more general approach). Still, in the $\F_1$ context,  one wants that the dimension of the space coincides with the number of closed points, so in the Algebraic Geometry way of defining closed linear subspaces, one could have an obstruction in the usual Krull dimension definition (as the supremum over finite chains of ideals) if the number of closed points  of the space is ``too large''. (No such problems arise when the number of closed points is countable.)
We propose to use the more general definition, and live with the exotic phenomena which will occur by the 
existence of linear subspaces of infinite dimension and infinite co-dimension.  }\end{remark}

  Note that we have assumed in the above approach of Proj-schemes, that polynomials have finite degree. On the other hand, it also makes sense to consider the notion of
 $\Proj(\mathbb{F}_1[X_i]_{i \in \P})$ where polynomials of infinite degree in $\mathbb{F}_1[X_i]_{i \in \P}$ are allowed; the situation corresponds to 
 infinite dimensional vector spaces, say of the form $\fK^{\Omega}$ with $\fK$ a field and $\Omega$ some infinite set, where vectors not necessarily have 
 a finite number of nonzero entries. One could define elements of $\mathbb{F}_1[X_i]_{i \in \P}$ as elements of the infinite cartesian product
 
 \begin{equation}
 \prod_{\nu \in \P}\chi^{\nu},
 \end{equation}
 where each $\chi^{\nu}$ is a copy of $\cup_{\epsilon \in \P}X_{\epsilon}$, and where we agree that two elements $(Y_{\ell})_{\P}$ and $(Y'_{\ell})_{\P}$ are the same if there exists some permutation $\sigma$ of $\P$ such that
 
 \begin{equation}
 (Y_{\sigma(\ell)})_{\ell \in \P} =  (Y'_{\ell})_{\ell \in \P}.
 \end{equation}
 
 (Formally, one denotes such an element by $\prod_{\mu \in \P}X^{e(\mu)}_{\mu}$, where $e(\mu)$ is the cardinality of the number of times $X_{\mu}$ occurs.)
 By this definition, the degree of a polynomial is then at most the cardinality of $\P$. 
 
 Of course the form of the prime ideals changes in this approach. For let $\fp$ be a prime ideal, and 
 take $\mu \in \P$. Then either $X_{\mu} \in \fp$ and so $RX_{\mu} \subseteq \fp$, or the countable cyclic  monoid $\langle X_{\mu}\rangle$ (which only contains elements with finite exponents) is in $R \setminus \fp$.
 Whence a set of the form
 
 \begin{equation}
 R \setminus \langle X_{\mu}\rangle
 \end{equation}
 is a prime ideal, while it contains polynomials $X_\mu^\omega$,  where $\omega \not\in \mathbb{N}$. 
 We have an injection
 
 \begin{equation}
 \psi: 2^{\P} \longrightarrow \Spec(R): J \longrightarrow R \setminus \langle X_j \vert j \in J\rangle,
 \end{equation}
 where for once $\Spec(R)$ denotes the prime spectrum of $R$. It easily follows that we have chains of prime ideals
 
 \begin{equation}
 (\fp_{\ell})_{\ell \in \Omega}\ \ \mathrm{such\ that}\ \ \mu < \mu' \Longrightarrow \fp_{\mu} \subset \fp_{\mu'}\ \ \mathrm{for}\ \ \mu, \mu' \in \Omega, 
 \end{equation}
 where $\mathrm{card}(\Omega) = \mathrm{card}(2^{\P})$.\\

 \begin{remark}[Krull dimension]{\rm
 We do not work with this definition of Proj-schemes since the Krull dimension is not equal to the cardinality of $\P$ anymore, and this is a feature we really want to enjoy in $\mathbb{F}_1$-theory (having Tits's definition of spherical $\mathbb{F}_1$-buildings in mind, cf. the first chapter).  In some sense, this notion of projective space (vector space) does not reflect the motivic nature of its $\mathbb{F}_1$/Weyl-theory.}
 \end{remark}

\newpage
\section{Loose graphs and $\mD_0$-schemes}

 Define a {\em loose graph}\index{loose!graph} (``L-graph'') to be a rank $2$ incidence geometry $(V,E,\I)$ with the additional property that each line is incident with at most two distinct points. 
 (In other words, it relaxes the definition of graphs, in that an edge can now also have one, or even no, point(s). Also, since we introduce loose graphs as incidence geometries, we do not allow loops, and the geometry is undirected.) In this section, we usually suppose that loose graphs are always connected unless otherwise mentioned | so isolated points/vertices or lines/edges do not exist.\\

 \subsection{Embedding theorem}
 
 Let $\Gamma = (V,E,\I)$ be a loose graph. We define a projective space $\P(\Gamma)$ over $\mathbb{F}_1$ as follows. 
 Let $E' \subseteq E$ be the set of ``loose edges''\index{loose!edge} | edges with only a single point. On each of these edges, we add a new point, as such creating a point set $V'$
 which is in bijective correspondence with $E'$.
Now  $\P(\Gamma)$ is the complete graph
 on the vertex set $V \cup V'$. 
 As such, we have an embedding of geometries
 
 \begin{equation}
  \psi: \Gamma \hookrightarrow \P(\Gamma) = \P,
 \end{equation}
 where we see $\P$ as the combinatorial projective space  over $\mathbb{F}_1$ of dimension $\vert V \vert + \vert V'\vert - 1$.
If $\Gamma$ is a graph, then $E' = 0$ and the dimension of $\P$ is $\vert V\vert - 1$.
 
 \begin{theorem}
  The following properties clearly hold.
 \begin{itemize}
 \item[DIM]
 $\P$ has minimal dimenson $\vert V \vert + \vert V' \vert - 1$ with respect to the embedding property (that is, there is no projective space over $\mathbb{F}_1$ of smaller dimension in which $\Gamma$ embeds).
 \item[AUT]
 Each automorphism of $\Gamma$ is faithfully induced by an automorphism of $\P$. (So that in particular, only the identity automorphism of $\P$ can fix any element of $E \cup V$.)
 \end{itemize}
 \end{theorem}

 \subsection{Example: Projective completion}
 Note that if one starts with a combinatorial affine space $\A$ over $\mathbb{F}_1$, considered as a loose graph,
 $\P(\A)$ is precisely the projective completion of $\A$.\\

\subsection{Patching and the functor $\Theta$}
\label{inf}

Now let $\Gamma = (V,E,\I)$ be  a not necessarily finite graph. We will give a "dual patching" argument as follows.

Consider $\P = \P(\Gamma)$, and note that since $\Gamma$ is a graph, $\P \setminus \Gamma$ | when $\P$ is considered as a graph | is just a 
set $S$ of edges. Let $\mu$ be arbitrary in $S$, and let $z$ be one of the two (closed) points on $\mu$.
Suppose that in the projective space $\P = \Proj(\mathbb{F}_1[X_i]_{i \in V})$, $z$ is 
defined by the ideal generated by the polynomials

\begin{equation}
X_i,\ \ i \in V, i \ne j = j(z).
\end{equation}

Let $\P(z)$ be the complement in $\P$ of $z$; it is a hyperplane defined by $X_j = 0$ (and it forms a complete graph on all the points but $z$).
Denote the corresponding closed subset of $\Proj(\mathbb{F}_1[X_i]_{i \in V})$ by $C(z)$.
Let $z' \ne z$ be the other point of the edge $\mu$ corresponding to the index $j' = j(z') \in V$. Define the subset $\P(z') = \P \setminus \{z'\}$ of $V$, and denote the corresponding closed subset by $C(z')$. 
Finally, define

\begin{equation}
C({\mu}) = C(z) \cup C(z').
\end{equation}

It is also closed in $\Proj(\mathbb{F}_1[X_i]_{i \in V})$, and the corresponding closed subscheme is the projective space $\P$ ``without the edge $\mu$''; the coordinate ring is $\mathbb{F}_1{[X_i]}_{i \in V}/I_{\mu}$ (where $(X_jX_l) =: I_{\mu}$) and its scheme is the Proj-scheme defined by this ring.
Now introduce the closed subset

\begin{equation}
C(\Gamma) =  \cap_{\mu \in S}C(\mu).
\end{equation}

Then $C(\Gamma)$ defines a closed subscheme $S(\Gamma)$ which corresponds to the graph $\Gamma$.  We have 

\begin{equation}
S(\Gamma) = \Proj(\mathbb{F}_1{[X_i]}_{i \in V}/\cup_{\mu \in S}I_{\mu}).
\end{equation}

In this presentation, an edge corresponds to a relation, and we construct a coordinate ring for $\Theta(\Gamma) = S(\Gamma)$ by deleting all relations of the ambient space $\P(\Gamma)$ which  are defined by edges in the complement of $\Gamma$.
We call a $\mD_0$-scheme $S(\Gamma)$ constructed from a graph $\Gamma$ a {\em G-scheme}\index{G-scheme}. 

A similar construction can be done for loose graphs, cf. \S\S \ref{logr}.\\

\subsection{Automorphism groups}

The following theorem, using the notation of the introductory paragraph of this section, is easy to obtain.

\begin{theorem}
\label{WLGaut}
For any element $\Gamma \in \mathbb{G}$, we have that
 \begin{equation}
 \Aut(\Gamma) \cong \Aut(S(\Gamma)).
 \end{equation}
 \end{theorem}

\subsection{Extension to loose graphs}
\label{logr}

Let $\Gamma = (V,E)$ be a connected loose graph. We distinguish three types:

\begin{itemize}
\item[type I]
graphs;
\item[type II]
complements of graphs $\Delta \subseteq C$ (where $C$ is some complete graph in which $\Delta$ is embedded); 
\item[type III]
loose graphs not of type I nor II.
\end{itemize}

If $\Gamma$ is of type I, we have seen how to associate a closed $\mD_0$-subscheme $S(\Gamma)$ of $\P(\Gamma)$ to $\Gamma$. If $\Gamma$ is of type II, then we define 
the $\mD_0$-scheme $S(\Gamma)$ naturally on the open set of $\P(\Gamma)$ which is the complement of the (closed) point set of the graph $\Gamma^c$ (the complement of $\Gamma$ in $\P(\Gamma)$). If $\Gamma$ is of type III, $S(\Gamma)$ is the $\mD_0$-scheme defined by the intersection of the closed subscheme defined on its graph theoretical completion $\overline{\Gamma} \ne \Gamma$, and the open set which is the complement of the complete graph defined on the vertices of $\overline{\Gamma} \setminus \Gamma$. As such we have:

\begin{proposition}
Each loose graph $\Gamma$ defines a $\mD_0$-scheme $S(\Gamma)$.
\end{proposition}

Denote the category of loose (undirected, loopless)
graphs and natural morphisms by $\mathbb{LG}$\index{$\mathbb{LG}$}.
The following theorem is obtained in a similar way as Theorem \ref{WLGaut}.

\begin{theorem}
\label{WLGaut2}
For any element $\Gamma \in \mathbb{LG}$, we have that
 \begin{equation}
 \Aut(\Gamma) \cong \Aut(S(\Gamma)).
 \end{equation}
 \end{theorem}

\medskip
\subsection{Connectedness}

Elements of the category of loose schemes have many important properties which can easily be read off from the corresponding  loose graph | recall for instance 
Theorem \ref{WLGaut}. Another one is:

\begin{theorem}
A  loose scheme $S(\Gamma)$ is connected if and only if the  loose graph $\Gamma$ is connected.
\end{theorem}

Proofs and more details can be found in \cite{NotesI}.

\bigskip
\begin{remark}[Weighted incidence geometries]
{\rm
One could go a step further and associate a $\mD_0$-scheme to a {\em weighted incidence geometry}\index{weighted incidence geometry} (that is, an incidence geometry coming with a weight function on the point set)  in a similar way as one does for  loose graphs (of course, one should do this by by-passing the embedding theorem). As such, all $\mD_0$-schemes could be constructed
from a combinatorial geometry, and they could be studied through these geometries. 
}
\end{remark}

\newpage
\section{Another approach | $\Upsilon$-schemes}

In this section we present yet another approach to $\Fun$-schemes, taken from the sketch in \cite{NotesI}.

\subsection{$\F_1$-Descent}

Let $R$ be any commutative ring with unit, and let $G = \{g_i \vert i \in I\}$ be a minimal generating set. Define a surjective 
homomorphism
\begin{equation}
\Phi: \mathbb{Z}{[X_i]}_I \longrightarrow R: X_j \longrightarrow g_j\ \ \forall j \in I, 
\end{equation}
so that $R \cong \mathbb{Z}{[X_i]}_I/J$ with $J$ the kernel of $\Phi$. 
For an element $P$ of $J$, write $P(1)$ for the set of ``$\mathbb{F}_1$-polynomials'' defined by $P$; if 
\begin{equation}
P = \sum_{i = 0}^kk_iX_0^{n_{i0}}\ldots X_m^{n_{im}}, 
\end{equation}
then 
\begin{equation}
P(1) := \{ X_0^{n_{i0}}\ldots X_m^{n_{im}}  \vert i = 0,\ldots,k\}.
\end{equation}

Then $\Spec(\mathbb{F}_1[X_0,\ldots,X_m]/\langle P(1) \vert P \in J \rangle) =: \underline{Y}$ is an {\em $\mathbb{F}_1$-descent}\index{$\Fun$-descent} of the affine scheme $\Spec(R) =: Y$, and we write

\begin{equation}
Y \leadsto \underline{Y} \ \ \mbox{or}\ \ Y  \overset{\delta}{\leadsto} \underline{Y}.
\end{equation}

We also use the same notation for the opposite operation between rings
(by mapping $\mathbb{Z}{[X_i]}_I$ to 
$\mathbb{F}_1{[X_i]}_I$ and $J$ to $\{ P(1) \vert P \in J\}$. For the sake of convenience, define $J(1) := \{P(1) \vert P \in J\}$\index{$J(1)$}.
Associate to a commutative unital ring $R$ the set $\Omega(R) := \{ (G,J) \vert \langle G \rangle \overset{\mbox{min}}{=} R, R \cong \Z[X_i]_{i \in G}/J\}$ (= 
the category of minimal generating sets of $R$, together with explicit kernels of the natural morphism $\phi: \Z[X_i]_{i \in G} \mapsto R: X_g \mapsto g$).

\begin{remark}
{\rm Note that the operation ``$\leadsto$'' is very different than the forgetful functor which Deitmar applies to descend from $\mathbf{Gro}$ to $\mD_0$.}
\end{remark}

The category $\underline{\mathbb{Y}}$ is defined by having as objects all $\mathbb{F}_1$-descents, and the morphisms consist of $\mD_0$-scheme morphisms $f: A \longrightarrow B$ making the diagram below commute

\begin{equation}
\begin{array}{ccc}
Y  &\rightarrow  &A \\
&&\\
 &\rosetlt &\downarrow{f} \\
&&\\
& &B
\end{array}
\end{equation}

Note that the diagrams

\begin{equation}
\Spec(k[X_0,\ldots,X_n]) \rightarrow \Spec(\mathbb{F}_1[X_0,\ldots,X_n]), \ \ \mbox{and}
\end{equation}

\begin{equation}
\Proj(k[X_0,\ldots,X_n]) \rightarrow \Proj(\mathbb{F}_1[X_0,\ldots,X_n])
\end{equation}
are unique and in accordance with the theory seen so far. In fact, using the analogy between the integers and polynomial rings, we als could write

\begin{equation}
\mathbb{Z} \rightarrow \mathbb{F}_1.
\end{equation}

It is clear why these diagrams are unique | the more relations are needed to describe a commutative ring inside the polynomial ring corresponding to a given generating set (that is, the more ``complex'' the kernel ideal), the more $\mathbb{Z}$-extensions exist and the ``less natural'' the presentation.



\medskip
\subsection{Geometric interpretation}

Note that the choice of representing a commutative ring by a generating set is not canonical (with respect to $\mathbb{F}_1$-descent) {\em whatsoever}:
consider for instance the two rings 
\begin{equation}
\mathbb{Z}[X_1,\ldots,X_n]/(X_1)\ \ \mbox{and}\ \ \mathbb{Z}[X_1,\ldots,X_n]/(X_1 - X_2)
\end{equation}
with $n \in \mathbb{N} \setminus \{0,1\}$;
they are isomorphic (by $X_1 \longrightarrow X_1 - X_2$ and $X_i \longrightarrow X_i$ for $i \ne 1$) and give the same $\mathbb{Z}$-schemes, 
but have nonisomorphic $\mathbb{F}_1$-descents (affine spaces $\Spec(\mathbb{F}_1[X_3,\ldots,X_n])$ and $\Spec(\mathbb{F}_1[X_2,\ldots,X_n])$) in the theory sketched above. This  is simply because from below one cannot see addition. Although Krull dimension is not preserved in this viewpoint, this 
is not so important since we are looking at a category which represents all $\mathbb{F}_1$-pieces of a variety relative to its generating sets.
In fact, what we do when descending to $\mathbb{F}_1$, for instance for a projective $\mathbb{Z}$-variety, is intersecting it with the canonical base, as points 
over $\mathbb{F}_1$ can have at most one nonzero coordinate. So isomorphic varieties over $\mathbb{Z}$ can be nonisomorphic over $\mathbb{F}_1$ (which
is also true when considering varieties over fields in comparison to $\mathbb{Z}$). (The intersection of $\mathbb{Z}$-varieties with the canonical base
could be compared with considering the fixed points of an ``absolute Frobenius map''.)

\medskip
\subsection{The category $\overline{\mathbb{X}}$}
\label{UX}

For each $X \in \mD_0$ we also consider the category $\overline{\mathbb{X}} = \mathbf{C}_X$ with objects

\begin{equation}
\overline{X} \in \mathbf{Gro}\ \ \mbox{for}\ \mbox{which}\ \ \overline{X} \leadsto X,
\end{equation}
and where the morphisms are those scheme theoretical morphisms commuting with the $\delta$-map | in other words, 
for $Y, Z \in \overline{\mathbb{X}}$, we have that
$\mathrm{Hom}_{\overline{\mathbb{X}}}(Y,Z)$ consists of $\mathbb{Z}$-scheme morphisms $\alpha: Y \longrightarrow Z$ for which the 
following diagram commutes:

\begin{equation}
\begin{array}{ccc}
Y  &\rightarrow  &Z \\
&&\\
\rosotlt &\roswtlt & \\
&&\\
X& &
\end{array}
\end{equation}


\medskip
\subsection{Ideals and prime ideals}
 
Let $R \cong \mathbb{Z}{[X_i]}_I/U$ be a commutative ring, and let $A = \delta(\mathbb{Z}{[X_i]}_I/U)$, $A = \mathbb{F}_1{[X_i]}_I/V$. 
Then if $e \cup V$ is an element of $\mathbb{F}_1{[X_i]}_I/V$, we have that the inverse image through $\delta$ is the element
\begin{equation}
\mathbb{Z}[e \cup V] + U \in \mathbb{Z}{[X_i]}_I/U.
\end{equation}

Note that ideals in $R$ are mapped by $\delta$ to ideals of $A$, and that vice versa, inverse images of ideals in $A$ are again ideals in $R$.

Let $P$ be a prime ideal in $A$; then 
\begin{equation}
\delta^{-1}(P) = \mathbb{Z}[P \cup V]/U. 
\end{equation}

Let $f + U, g + U \not\in \delta^{-1}(P)$; then some terms of $f$ and $g$, say for example
$f_* (+ U)$ and $g_* (+ U)$ are not in   $\mathbb{Z}[P \cup V]/U$. So $\delta(f_* + U)$ and $\delta(g_* + U)$ are not in $P$, implying that 
$\mathbb{Z}f_*g_* \not\subseteq \delta^{-1}(P)$. Whence $fg + U$ is not in $\delta^{-1}(P)$, and so
the latter is a prime ideal in $R$.

The converse is not true; consider for instance $R \cong \mathbb{Z}[X,Y]$ and the prime ideal $(X^2 + Y)$. After applying $\delta$, we get 
$\mathbb{F}_1[X,Y]$ and the ideal $(X^2,Y)$ which is not prime anymore. (Another interesting example is $\mathbb{Z}[X]$ with $(X^2 + 1)$.)
 
In the particular case of $P$ being the maximal ideal $(\mathbb{F}_1{[X_i]}_I/V) \setminus \{1 \cup V\}$, we have that $\delta^{-1}(P) = 
 (\mathbb{Z}{[X_i]}_I \setminus \mathbb{Z}^\times)/U$. Over a field $k$, this becomes $(k{[X_i]}_I \setminus k^\times)/U$. 
 Summarizing, we obtain the following proposition.
 
 \begin{proposition}
 \begin{itemize}
 \item[{\rm (i)}]
 We have that $\delta$ and $\delta^{-1}$ preserve ideals, $\delta^{-1}$ sends prime ideals of $\mathbb{F}_1$-rings to prime ideals of 
 commutative rings, but the converse is not true. 
 \item[{\rm (ii)}]
 As a corollary, we have that $\delta$ induces a continuous function between the Zariski 
 topologies of commutative rings and their $\mathbb{F}_1$-descents.
 \item[{\rm (iii)}]
 For all fields, the unique closed point of an affine $\mD_0$-scheme $S$ corresponds to a closed point of any element of $\delta^{-1}(S)$.
 \end{itemize}
 \end{proposition}

\medskip
\begin{remark}[Comparison with Deitmar]
{\rm
Let $A$ and $R$ be represented as earlier. Then $\delta^{-1}(V) = \mathbb{Z}[V] + U$, so that $U \subseteq \delta^{-1}(V)$. In Deitmar's first version 
of scalar extension, $U := \mathbb{Z}[V]$ by definition, so is 
\begin{equation}
\Spec(R) = \Spec(A)_{\mathbb{Z}} 
\end{equation}
($\Spec(A)_{\mathbb{Z}} = \Spec(\mathbb{Z}[A]/\mathbb{Z}[V])$).
This endowes $\mathbb{Z}$-lifts of 
$\mD_0$-schemes with a toric structure (all elements of $V$ are lifted to the same monics, so no extra relations are introduced while lifting $V$ to an ideal in $\mathbb{Z}{[X_i]}_I$), while in the version described above, the structure of $U$ endowes the scheme with a nontrivial gluing, governed
by the relations generated by $U$.
}
\end{remark}

\medskip
\subsection{Base extension to $\mathbb{Z}$, $\mathbb{F}_1$-schemes revisited}

For an affine $\mD_0$-scheme $S$, $X = \mathbf{C}_S$ can be seen as the category of $\mathbb{Z}$-schemes which arise by imposing all possible
additions on the $\mathbb{F}_1$-ring of $S$. 

An element $(\overline{X},X)$ in $\mathbf{Gro} \times \mD_0$ for which $\overline{X} \in \overline{\mathbb{X}}$ could be called an {\em $\mathbb{F}_1$-scheme}\index{$\Fun$-scheme};
\begin{equation}
\overline{X} = (\overline{X},X) \otimes_{\mathbb{F}_1} \mathbb{Z}
\end{equation} 
is its base change to $\mathbb{Z}$.

\begin{remark}{\rm
Note that Deitmar's $X_\mathbb{Z}$ is in $\overline{\mathbb{X}}$.}
\end{remark}

\medskip
Observe the following.

\begin{proposition}
\label{mingenmon}
Let $M = \mathbb{F}_1{[X_i]}_I$, and let $J$ be an ideal in $M$. Then there is a unique set $S \subseteq J$ such that $J = \cup_{s \in S}Ms$, and such that
any other set $S'$ with this property contains $S$ (in other words, $S = \cap_{\langle S' \rangle = J}S')$. We write $S = S(J)$ to denote this set.
\end{proposition}
 
\begin{proof}
Let $S$ be a minimal generating set of $J$ (that is, no subset of $S$ generates $J$), and let $S'$ be another generating set. Let $s \in S$; then there is an $m \in M$
and $s' \in S'$ such that $s = ms'$. On the other hand, there is an $s'' \in S$ and $m' \in M$ such that $s' = m's''$, or $s = mm's''$. By minimality this is only possible
when $m' = m'' = 1$.
\end{proof}

 Let $A = \Fun[X_i]_{I}/V$ and $R = \Z[X_i]_I/U$ be as before. 
Let  $S = S(V)$. Then $U$ is generated by polynomials $\{ f_j \}_J$ such that $\cup_Jf_j(1) = S$.  So $\mathbf{C}_V$ can be canonically described
{\em relative to $S$}. So given an affine $\mD_0$-scheme $A = \Spec(A)$, the objects in $\mathbf{C}_S$ are defined as the $\Spec(R)$ for which there is a 
representation $R \cong  \Z[X_i]_I/U$ such that $A \cong \Fun[X_i]_{I}/V$ with $U(1) = V$.\\

 \medskip
\subsection{General $\mathbb{F}_1$-descent, etc}

Let $X$ be a $\mathbb{Z}$-scheme, and let $X = \cup_{i \in I}X_i$ be an open cover of $X$, with $X_i = \Spec(A_i)$ and $A_i$ a commutative ring.
Then we say that $\underline{X}$ is a {\em base descent}\index{base descent} to $\mathbb{F}_1$ of $X$, and write 
\begin{equation}
X  \leadsto \underline{X}
\end{equation}
as before, if $\underline{X} \cong \cup_{i \in I}\underline{\Spec(A_i)}$, where 
\begin{equation}
\Spec(A_j)  \rightarrow \underline{\Spec(A_j)}
\end{equation}
for all $j \in J$, and where the $\delta$  (that is, the representations of the rings $A_i$) are chosen such that the 
gluing is well defined.\\

Base ascent, the categories $\overline{\mathbb{X}}$  and $\underline{\mathbb{X}}$, etc. are defined similarly.\\

\medskip
\subsection{$\Upsilon$-Schemes}

It might be handy to consider a class of objects which has a richer structure than the ones we encountered till now in the present
section. Instead of merely working with objects $(\overline{Y},Y,\leadsto)$, we also could consider triples 
\begin{equation}
(\mathbf{C}_Y,Y,\leadsto)
\end{equation}
in which the first argument consists of {\em all} $\mathbb{Z}$-schemes which descend to $Y$, and study all these objects at once. 
We call all these instances of $\Fun$-schemes ``$\Upsilon$-schemes''\index{$\Upsilon$-scheme}.

Two natural questions arise:
\begin{itemize}
\item[(a)]
What is the zeta function of an object $(\mathbf{C}_X,X,\leadsto)$?
\item[(b)]
Is there a zeta function associated to $\Upsilon$-schemes $(\mathbf{C}_X,X,\leadsto)$?
\end{itemize}

In the next section, we will make 
a conjecture on the zeta function of these $\Upsilon$-schemes which gives crucial information about a possible answer to (a).\\

\newpage
\section{Zeta functions and absolute zeta functions}

Various $\Fun$-objects are defined through their ($\Fun$-) zeta functions. We review some parts of this theory in
the present section.\\

\subsection{Arithmetic zeta functions}

Let $X$ be a scheme of finite type over $\mathbb{Z}$ | a {\em $\mathbb{Z}$-variety}\index{$\Z$-variety}. (So $X$ has a finite covering of affine $\mathbb{Z}$-schemes $\Spec(A_i)$ with the $A_i$
finitely generated over $\mathbb{Z}$.) 
Recall that if $\widetilde{X}$ is an $\F$-scheme, $\F$ a field, a point $x \in \widetilde{X}$ is {\em $\F$-rational}\index{$\F$-rational} if 
the natural morphism 
\begin{equation}
\F \hookrightarrow k(x)
\end{equation}
is an isomorphism.
A morphism
\begin{equation}
\Spec(\F) \longrightarrow \widetilde{X}
\end{equation}
is completely determined by the choice of a point $x \in \widetilde{X}$ and a field extension $\F/k(x)$ (so once one has such a field extension, the morphism
is constructed by sending the unique point of $\Spec(\F)$ to $x$). Whence the set of $\F$-rational points of $\widetilde{X}$ can be identified with 
\begin{equation}
\mathrm{Hom}(\Spec(\F),\widetilde{X}).
\end{equation}
(If $\widetilde{X} \cong \Spec(A)$ is affine, $A$ being a commutative ring, one also has the identification with $\mathrm{Hom}(A,\F)$.)

\begin{proposition}[Closed and rational points]
\label{proparze}
\begin{itemize}
\item[{\rm (1)}]
A point $x$ of $X$ is closed if and only if its residue field $k(x)$ is finite. (Note that $\vert k(x) \vert = \mathrm{dim}(\overline{\{x\}})$ as a closed subscheme.)
\item[{\rm (2)}]
Let $k = \overline{k}$ be algebraically closed, and let $\widetilde{X} \longrightarrow \Spec(k)$ be a $k$-scheme which is locally of finite type. Then a point $x$ is closed if and only if  it is $k$-rational.
\item[{\rm (3)}]
More generally, let $\mathbb{F}$ be any field. Then a point $x$ of the $\F$-scheme $\widetilde{X} \longrightarrow \Spec(\F)$, which is again assumed to be locally of finite type, is closed if and only if the field extension $k(x)/\F$ is finite. A closed point is $\F$-rational if and only if $k(x) = \F$.
\end{itemize}
\end{proposition}


Assume again that $X$ is an arithmetic scheme. 
Let $\overline{X}$ be the ``atomization''\index{atomization} of $X$; it is the set of closed points, equipped  with the discrete topology and the sheaf of fields $\{k(x) \vert x\}$. For $x \in \overline{X}$, let $N(x)$ be the cardinality of the finite field $k(x)$, that is, the {\em norm}\index{norm} of $x$.
Define the {\em arithmetic zeta function}\index{arithmetic zeta function} $\zeta_X(s)$ as 
\begin{equation}
\zeta_X(s) := \prod_{x \in \overline{X}}\frac{1}{1 - N(x)^{-s}}. 
\end{equation}

\begin{lemma}[Reduction Lemma]
\label{redlem}
\begin{itemize}
\item[{\rm (1)}]
If $X$ is a (possibly infinite) disjoint union of subschemes $X_i$, we have
\begin{equation}
\zeta_X(s) = \prod_i\zeta_{X_i}(s).
\end{equation}
(It is enough that the atomization of $X$ is the disjoint union of the atomizations of the $X_i$, since $\zeta_X(\cdot)$ only depends on $\overline{X}$.)
\item[{\rm (2)}]
{\em Application of (1)}: If $f: X \longrightarrow Y$ is a scheme morphism, and if $X_y := f^{-1}(y)$ for $y \in \overline{Y}$, one has 
\begin{equation}
\zeta_X(s) = \prod_{y \in \overline{Y}}\zeta_{X_y}(s).
\end{equation}
(The $X_y$ are schemes over the finite fields $k(y)$.)
\end{itemize}
\end{lemma}

\medskip
\subsection{Four standard examples | affine and projective space, Dedekind and classical Riemann}

\# Let $X = \Spec(A)$, where $A$ is the ring of integers of an number field $\K$; then $\zeta_X(s)$ is the {\em Dedekind zeta function}\index{Dedekind zeta function} of $\K$.\\

\# Put $X = \Spec(\mathbb{Z})$; then $\zeta_X(s)$ becomes the classical {\em Riemann zeta function}\index{Riemann zeta function}.\\

\# With $\mathbb{A}^n(X)$ being the affine $n$-space over a scheme $X$, $n \in \mathbb{N}$, one has 
\begin{equation}
\zeta_{\mathbb{A}^n(X)} = \zeta_X(s - n).
\end{equation}

\# And with $\mathbb{P}^n(X)$ being the projective $n$-space over a scheme $X$, $n \in \mathbb{N}$, one has 
\begin{equation}
\zeta_{\mathbb{P}^n(X)} = \prod_{j = 0}^n\zeta_X(s - j).
\end{equation}

The latter can be obtained inductively (applying Lemma \ref{redlem}) by using the expression for the zeta function of affine spaces.

\medskip
\subsection{Finite fields}

Let $X$ be a scheme of finite type over $\F = \F_q$; if $x \in \overline{X}$, the residue field $k(x)$ is a finite extension of $\F_q$. With $\mathrm{deg}(x)$ the degree of 
the extension, we have 
\begin{equation}
N(x) = q^{\mathrm{deg}(x)}.
\end{equation}

Then $\zeta_X(s) = Z(X,q^{-s})$, where $Z(X,T)$ is the power series defined by the product
\begin{equation}
Z(X,T) = \prod_{x \in \overline{X}}\frac{1}{1 - T^{\mathrm{deg}(x)}}.
\end{equation}

Denote by $\F_d$ the extension $\F_{q^d}/\F_q$,  and let $X_d = X(\F_d)$ be the $\F_d$-rational points of $X$. We have 
\begin{equation}
X(\overline{\F}) = \bigcup_{d \in \mathbb{N}^\times}X_d,
\end{equation}
$\overline{\F}$ denoting the algebraic closure of $\F$. \\

Now considere the Frobenius map
\begin{equation}
\mathrm{Fr}: X(\overline{\F}) \longrightarrow X(\overline{\F}): x \longrightarrow \mathrm{Fr}(x) = x^q, 
\end{equation}
where $x = (x_1,\ldots,x_n)$ and $x^q = (x_1^q,\ldots,x_n^q)$. Then we can alternatively describe $X_d$ by
\begin{equation}
X_d = \{ x \in X(\overline{\F}) \vert \mathrm{Fr}^d(x) = x\},
\end{equation}
as the elements of $\overline{\F} = \bigcup_{i \in \mathbb{N}^\times}\F_i$ in $\F_d$ are characterized by the fact that they are fixed by $\mathrm{Fr}^d$.

\begin{remark}
{\rm
Note that for an arithmetic scheme $X$, we have that 
\begin{equation}
\zeta_X(s) = \prod_{p\ \ \mbox{prime}}\zeta_{X\vert \F_p}(s),
\end{equation}
where $\zeta_{X\vert \F_p}(s) = Z(X,p^{-s})$.
}
\end{remark}

\medskip
\subsection{Lefschetz fixed points formula and $\ell$-adic cohomology}

Let $\F_q$ be a finite field and $\overline{\F_q}$ an algebraic closure. Let $X$ be an $\F_q$-variety (a scheme of finite type over $\F_q$), and let $\overline{X}$ be the scheme 
obtained by base extension $\F_q \longrightarrow \overline{\F_q}$.  Finally, let $\ell$ be a prime number different from the characteristic of $\F_q$.
For every $i \geq 0$ there is an \'{e}tale {\em $\ell$-adic cohomology group}\index{$\ell$-adic cohomology} $H^i_{\mathrm{et}}(\overline{X},\mathbb{Q}_{\ell})$\index{$H^i_{\mathrm{et}}(\overline{X},\mathbb{Q}_{\ell})$}; it is a finite dimensional $\mathbb{Q}_\ell$-vector space, 
and it vanishes if $i > 2\mathrm{dim}(X)$. Note that 
\begin{equation}
H^i(\overline{X},\mathbb{Z}_{\ell}) := \lim_{\leftarrow}H^i(\overline{X},\mathbb{Z}/{\ell^n\mathbb{Z}})
\end{equation}
and
\begin{equation}
H^i_{\mathrm{et}}(\overline{X},\mathbb{Q}_{\ell}) := H^i(\overline{X},\mathbb{Z}_{\ell}) \otimes_{\mathbb{Z}_\ell}\mathbb{Q}_\ell = H^i(\overline{X},\mathbb{Z}_{\ell})[1/\ell].\\
\end{equation}

Let $\mathrm{Fr}: X \longrightarrow X$ be as above (it is the identity on the underlying topological space, and acts on the structure sheaf $\mO_X$ by $f \longrightarrow f^q$).
It can be shown that the morphism $\mathrm{Fr}: X \longrightarrow X$ acts by functoriality on the spaces $H^i_{\mathrm{et}}(\overline{X},\mathbb{Q}_{\ell})$; denote the trace of this 
endomorphism by $\mathrm{Tr}_i(\mathrm{Fr})$, and put
\begin{equation}
\mathrm{Tr}(\mathrm{Fr}) := \sum_i(-1)^i\mathrm{Tr}_i(\mathrm{Fr}).
\end{equation}
This quantity is the {\em Lefschetz number}\index{Lefschetz!number} of $\mathrm{Fr}$ and is independent of the choice of $\ell$, as the following result by Grothendieck shows:

\begin{theorem}[Lefschetz Formula\index{Lefschetz!formula} \cite{SGA4/5,SGA5}]
$\mathrm{Tr}(\mathrm{Fr}) = \vert X(\F_q)\vert$.\\
\end{theorem}

So for any field extension $\F_{q^m} \vert \F_q$ ($m \in \mathbb{N}^\times$)  the Lefschetz formula reads
\begin{equation}
\vert X(\F_{q^m})\vert = \mathrm{Tr}(\mathrm{Fr}^m) = \sum_i(-1)^i\mathrm{Tr}(\mathrm{Fr}^m \vert H^i_{\mathrm{et}}(\overline{X},\mathbb{Q}_{\ell})).
\end{equation}

For the case that $X$ is a curve of genus $g$ (which will be of special interest later on), one derives the following formula for the zeta function in terms of the action of the Frobenius operator:

$$ \zeta_{X}(s) 
		= \frac{\prod_{j = 1}^{2g}(1 - \lambda_jq^{-s})}{(1 - q^{-s})(1 - q^{1 - s})} \nonumber \\
$$

\begin{eqnarray} 
\label{Lefschet}
			 = \frac{\mbox{\textsc{Det}}\Bigl((s\cdot\id - q^{-s}\cdot\mathrm{Fr})\Bigl| H^1_{\mathrm{et}}(\overline{X},\mathbb{Q}_{\ell})\Bigr)}{\mbox{\textsc{Det}}\Bigl((s\cdot\id - q^{-1}\cdot\mathrm{Fr})\Bigl| H^0_{\mathrm{et}}(\overline{X},\mathbb{Q}_{\ell})\Bigr.\Bigr)\mbox{\textsc{Det}}\Bigl((s\cdot\id - q^{-s}\cdot\mathrm{Fr})\Bigl| H^2_{\mathrm{et}}(\overline{X},\mathbb{Q}_{\ell})\Bigr.\Bigr)},
	\end{eqnarray}
where the $\lambda_j$s are the eigenvalues of the Frobenius acting on \'{e}tale cohomology. (``$\textsc{Det}(\cdot)$'' is the usual determinant.)\\

\medskip
\subsection{Absolute Frobenius endomorphisms}

Consider the algebraic closure $\overline{\mathbb{F}_1}$ of $\F_1$. We define the {\em absolute Frobenius endomorphism}\index{absolute!Frobenius automorphism} of degree $n \in \mathbb{N}$, denoted $\mathrm{Fr}^n_1$\index{ $\mathrm{Fr}^n_1$} to be the map
\begin{equation}
\mathrm{F}_1^n: \overline{\F_1} \longrightarrow \overline{\F_1}: x \longrightarrow x^n.
\end{equation}

Elements of $\F_1^{d} \cong \mu_d \cup \{0\} \leq \overline{\F_1}$ are characterized by the fact that they are the solutions of 
\begin{equation}
x^{\mathrm{Fr}_1^{d + 1}} = x,
\end{equation}
which is analogous to the fact that elements of finite fields $\F_{q^d} \leq \overline{\F_q}$ are singled out as fixed points of $\mathrm{Fr}^d$.

\medskip
\subsection{Projective spaces over extensions}

Let $m \in \mathbb{N}^\times \cup \{ \infty\}$, $n \in \mathbb{N}$, and
let $A$ be the $\mathbb{F}_1$-ring $\mathbb{F}_1^m[X_0,X_1,\ldots,X_n]$. Here, we put $\mathbb{F}_1^{\infty} = \overline{\F_1}$.
Rather than looking at $\Spec(A)$, we want to consider 
the $\Proj$-scheme $\P(n,m) := \Proj(\Spec(A))$\index{$\P(n,m)$}. The following is simple, yet it illustrates a different behavior of $\mathbb{F}_1$-schemes
than schemes over ``real fields''.

\begin{proposition}
Topologically, the structure of $\P(n,m)$ is independent of the choice of $m$. 
\end{proposition}
\begin{proof}
It suffices to observe that for any $\nu \in \mu_m$ and any $i \in \{0,1,\ldots,n\}$, we have
\begin{equation}
(X_i) = (\nu X_i).
\end{equation}
\end{proof}

If $X$ is a $\K$-scheme, where $\K$ is a field, then the closed points of $X$ represent orbits of the Galois group $\mathrm{Gal}(\overline{\K}/\K)$, and all $\overline{\K}$-rational points of $X \times_{\K}\overline{\K}$ are contained in the union of these orbits. So there is a natural map
\begin{equation}
\alpha: X(\overline{\K}) \longrightarrow X,
\end{equation}
sending closed points of $X(\overline{\K})$ to closed points of $X$, which is neither injective nor surjective.

Going back to the spaces $\P(n,m)$ of above, we see the closed points of $\P(n,m)$ (for any $m$) as orbits of $\mathrm{Gal}(\overline{\F_1}/\F_1^m)$.
A stalk at an arbitrary closed point of $\P(n,m)$ is isomorphic to
\begin{equation}
\{ 0\} \cup (\mu_m \times \mathrm{F}^{\mathrm{ab}}(X_0,X_1,\ldots,X_n)),
\end{equation}
where $\mathrm{F}^{\mathrm{ab}}(X_0,X_1,\ldots,X_n)$\index{ $\mathrm{F}^{\mathrm{ab}}(X_0,X_1,\ldots,X_n)$} is the free abelian group generated by $X_0,X_1,\ldots,X_n$. So on the algebraic level we can see
the extension of the ground field | we consider the stalk as consisting, besides $0$, of $m$ distinct copies of $\mathrm{F}^{\mathrm{ab}}(X_0,X_1,\ldots,X_n)$ equipped with 
a sharply transitive $\mu_m$-action, which is in accordance with the classical picture.\footnote{The $m$ ``points'' are invisible points which do not occur in $\mathbb{Z}$-schemes.}

The stalks at any given closed point of $\P(n,\infty)$ are given by 
\begin{equation}
\{ 0\} \cup (\overline{\mathbb{F}_1}^\times \times \mathrm{F}^{\mathrm{ab}}(X_0,X_1,\ldots,X_n)),
\end{equation}
and the absolute Frobenius map $\mathrm{Fr}_1^{d + 1}$ with $d \in \mathbb{N}^\times$ (which acts on the scheme by acting trivially on the topology and as $\mathrm{Fr}_1^{d + 1}$ on the  structure scheaf) singles out the (``ordinary'' and invisible) $\mathbb{F}_1^d$-points.

\medskip
\subsection{Deitmar zeta functions}

In \cite{Soule}, C. Soul\'{e}, inspired by Manin's paper \cite{Manin}, associated a zeta function to any sufficiently regular counting-type function $N(q)$ by considering the limit

\begin{equation}
\zeta_N(s) := \lim_{q \to 1}Z(q,q^{-s})(q - 1)^{N(1)}, \ \ \ \ s \in \mathbb{R}.
\end{equation}

\noindent
(See also the next paragraph for more on this definition.)
Here $Z(q,q^{-s})$ is the evaluation at $T = q^{-s}$ of the Hasse-Weil zeta function

\begin{equation}
Z(q,T) = \mathrm{exp}(\sum_{r \geq 1}N(q^r)\frac{T^r}{r}).
\end{equation}

\noindent
One computes that if $N(X) = a_0 + a_1X + \cdots  + a_nX^n$, then

\begin{equation}
\zeta_{X\vert \mathbb{F}_1}(s) = \prod_{i = 0}^n\frac{1}{(s - i)^{a_i}},
\end{equation}
which is in accordance with the aforementioned example for projective $\mathbb{F}_1$-spaces.\\

 In \cite{Deitmarschemes1}, the following theorem is obtained.

\begin{proposition}
Let $X$ be a $\mD_0$-scheme and $X_{\mathbb{Z}} = X \otimes_{\mathbb{F}_1}\mathbb{Z}$ be the Deitmar base extension to the integers. Then there exists a natural number $e$ and a polynomial $N(T)$ with integer coefficients such that for
every prime power $q$ one has 
\begin{equation}
(q - 1,e) = 1\ \  \Longrightarrow\ \  \# X_{\mathbb{Z}}(\mathbb{F}_q) = N(q).
\end{equation}
The polynomial $N$ is uniquely determined, and independent of the choice of $e$.
\end{proposition}

Deitmar calls this the ``zeta polynomial''\index{zeta!polynomial} of $X$. 

If $N_X(T) = a_0 + a_1T + \ldots + a_nT^n$ is the zeta polynomial of an arbitrary $\mD_0$-scheme $X$, we can thus define the zeta function of $X$ as

\begin{equation}
\zeta_{X \vert \mathbb{F}_1}(s) = \frac{1}{s^{a_0}(s - 1)^{a_1} \cdots (s - n)^{a_n}}. 
\end{equation}

The {\em Euler characteristic}\index{Euler characteristic} of $X$ is then defined as 

\begin{equation}
\chi(X) := N_X(1) = a_0 + \cdots + a_n.
\end{equation}

\medskip
\subsection{Kurokawa zeta functions}

In \cite{Kurozeta}, Kurokawa says a scheme $X$ is of {\em $\Fun$-type}\index{scheme of $\Fun$-type} if its arithmetic zeta function $\zeta_X(s)$ can be expressed in the form
\begin{equation}
\zeta_X(s) = \prod_{k = 0}^n\zeta(s - k)^{a_k}
\end{equation}
with the $a_k$s in $\Z$. A very interesting result in \cite{Kurozeta} reads as follows:

\begin{theorem}
Let $X$ be a $\Z$-scheme. The following are equivalent.
\begin{itemize}
\item[{\rm (i)}]
\begin{equation}
\zeta_X(s) = \prod_{k = 0}^n\zeta(s - k)^{a_k}
\end{equation}
with the $a_k$s in $\Z$.
\item[{\rm (ii)}]
For all primes $p$ we have
\begin{equation}
\zeta_{X\vert \F_p}(s) = \prod_{k = 0}^n(1 - p^{k - s})^{-a_k}
\end{equation}
with the $a_k$s in $\Z$.
\item[{\rm (iii)}]
There exists a polynomial $P_X(Y) = \sum_{i = 0}^na_kY^k$ such that
\begin{equation}
\#X(\F_{p^m}) = N_X(p^m) 
\end{equation}
for all finite fields $\F_{p^m}$.
\end{itemize}
\end{theorem}

Kurokawa defines the {\em $\Fun$-zeta function}\index{$\Fun$-zeta function} of a $\Z$-scheme $X$ which is defined over $\Fun$ as 
\begin{equation}
\zeta_{X\vert \Fun}(s) :=  \prod_{k = 0}^n(s - k)^{-a_k}
\end{equation}
with the $a_k$s as above. Define, again as above, the {\em Euler characteristic}\index{Euler characteristic}
\begin{equation}
\#X(\Fun) := \sum_{k = 0}^na_k.
\end{equation}

The connection between $\Fun$-zeta functions and arithmetic zeta functions is explained in the following theorem, taken from \cite{Kurozeta}.

\begin{theorem}
Let $X$ be a $\Z$-scheme which is defined over $\Fun$. Then
\begin{equation}
\zeta_{X\vert \Fun}(s) =  \lim_{p \longrightarrow 1}\zeta_{X\vert \F_p}(s)(p - 1)^{\# X(\Fun)}.
\end{equation}
Here, $p$ is seen as a complex variable (so that the left hand term is the leading coefficient of the Laurent expansion of $\zeta_{X \vert \Fun}(s)$ around $p = 1$).
\end{theorem}
We will give a brief sketch of the proof here.

\begin{proof}
We have 
\begin{equation}
\zeta_{X\vert \Fun}(s)(p - 1)^{\# X(\Fun)} = \prod_{k = 0}^n\big(\frac{1 - p^{k - s}}{p - 1}\big)^{-a_k}
\end{equation}
so that for the limit we get
\begin{eqnarray}
\lim_{p \longrightarrow 1}\zeta_{X\vert \Fun}(s)(p - 1)^{\# X(\Fun)} &= &\prod_{k = 0}^n\big(\lim_{p \longrightarrow 1}\frac{1 - p^{k - s}}{p - 1}\big)^{-a_k}\nonumber \\
&= &\prod_{k = 0}^n(s - k)^{-a_k}\nonumber \\ 
&= &\zeta_{X\vert \Fun}(s).
\end{eqnarray}
\end{proof}

For affine and projective spaces, we obtain the following zeta functions (over $\Z$, $\F_p$ and $\Fun$, with $n \in \mathbb{N}^\times$):
\begin{eqnarray}
\zeta_{\A^n\vert \Z}(s) &= &\zeta(s - n);\nonumber \\
\zeta_{\A^n\vert \F_p}(s) &= &\frac{1}{1 - p^{n - s}};\nonumber \\
\zeta_{\A^n\vert \Fun}(s) &= &\frac{1}{s - n},
\end{eqnarray}
and
\begin{eqnarray}
\zeta_{\P^n\vert \Z}(s) &= &\zeta(s)\zeta(s - 1)\cdots\zeta(s - n);\nonumber \\
\zeta_{\P^n\vert \F_p}(s) &= &\frac{1}{(1 - p^{-s})(1 - p^{1 - s})\cdots(1 - p^{n - s})};\nonumber \\
\zeta_{\P^n\vert \Fun}(s) &= &\frac{1}{s(s - 1)\cdots(s - n)}.
\end{eqnarray}

\medskip
\subsection{Appendix: Toric varieties and zeta functions}

Suppose $N$ is a {\em lattice}\index{lattice} | a group isomorphic to $\mathbb{Z}^n$ for some natural number $n$. 
A {\em fan}\index{fan} $\Delta$ in $N$ is a finite collection of proper convex rational polyhedral cones in the real vector space $N_{\mathbb{R}} = N \otimes \mathbb{R}$
such that every face of a cone in $\Delta$ is in $\Delta$, and the intersection of two cones in $\Delta$ is a face of each of these cones. Recall that a 
{\em convex cone}\index{convex cone} is a convex subset $\sigma$ of $N_{\mathbb{R}}$ with $\mathbb{R}^+\sigma = \sigma$. Such a cone is {\em polyhedral}\index{polyhedral} if it is finitely generated, 
and {\em proper}\index{proper cone} if it does not contain a nonzero subvector space of $N_{\mathbb{R}}$.

Suppose $\Delta$ is a fan in the lattice $N$, and let $N^D := \mathrm{Hom}(N,\mathbb{Z})$ be the dual lattice of $N$. For a cone $c \in \Delta$, the {\em dual cone}\index{dual cone}
$c^D$ is the cone in the dual space $N^D_{\mathbb{R}}$ consisting of all $\alpha \in N^D_{\mathbb{R}}$ for which $\alpha(c) \geq 0$.
As such one has defined a monoid $A_c = c^D \cap N^D$. Now put $U_c = \Spec(\mathbb{C}[A_c])$; if $\tau$ is a face of $c$, then $A_{\tau} \supseteq A_c$, and 
the latter inclusion gives rise to an open embedding 

\begin{equation}
U_{\tau} \hookrightarrow U_c.
 \end{equation}

Along these embeddings the affine varieties $A_c$ can be glued to obtain a variety $X_{\Delta}$ over $\mathbb{C}$, which has been given an $\mathbb{F}_1$-structure
\cite{Deitmartoric}. Then $X_{\Delta}$ is a toric variety, and the torus is $U_0 \cong \GL_1^n$. Every toric variety can be obtained in this way.

The next proposition, which is used in the proof of the theorem following it, is of independent interest.

\begin{proposition}[\cite{Deitmartoric}]
Let $B$ be a submonoid of the monoid $A$ of finite index. Then the map
\begin{equation}
\varphi: \Spec(A) \longrightarrow \Spec(B)
\end{equation}
defined by $\varphi(\fp) = \fp \cap B$ is a bijection.\\
\end{proposition}

In \cite{Deitmartoric} Deitmar then obtains the next theorem, which supports Manin's predictions (in Deitmar's theory).

\begin{theorem}[\cite{Deitmartoric}]
Let $\Delta$ be a fan in a lattice of dimension $n$. For $j \in \{0,1,\ldots,n\}$, let $f_j$ be the number of cones in $\Delta$ of dimension $j$. Set
\begin{equation}
c_j = \sum_{k = j}^nf_{n - k}(-1)^{k + j}\left(\begin{array}{c}k\\j\end{array}\right).
\end{equation}
Let $X$ be the corresponding toric variety. Then the $\mathbb{F}_1$-zeta function of $X$ equals
\begin{equation}
\zeta_X(s) = s^{c_0}(s - 1)^{c_1}\cdots(s - n)^{c_n}.
\end{equation}
\end{theorem}

\medskip
\subsection{Zeta functions of categories}
\label{Kurocat}

In several versions of $\mathbb{F}_1$-geometry (such as in $\Upsilon$-schemes), it appears that many $\Z$-schemes can descend to one and the same 
$\mD_0$-scheme. As we want to see these data as one object, it is desirable that one can attach a zeta function to such an object. In \cite{Kurokawacat}, Kurokawa introduces such an approach.\\

Let $\mathbf{C}$ be a category with a zero object (that is, an object which is both initial and terminal). An object $X$ of $\mathbf{C}$ is {\em simple}\index{simple object} if for every object $Y$, $\mathrm{Hom}(X,Y)$ only consists of monomorphisms and zero-morphisms. The {\em norm}\index{norm}
of an object $Z$ is defined as 
\begin{equation}
N(Z) = \vert \mathrm{End}(Z,Z) \vert = \vert \mathrm{Hom}(Z,Z)\vert.
\end{equation}

An object is {\em finite}\index{finite!object} if its norm is. We denote the category of isomorphism classes of finite simple objects of $\mathbf{C}$ by
$\mP(\mathbf{C})$\index{$\mP(\mathbf{C})$}. The {\em zeta function}\index{zeta!function!of a category} of $\mathbf{C}$ is
\begin{equation}
\zeta(\mathbf{C},s) = \prod_{P \in \mP(\mathbf{C})}\frac{1}{(1 - N(P)^{-s})}.
\end{equation}
Note that if two categories are equivalent, then their zeta functions are the same. 

Following Kurokawa \cite{Kurokawacat}, we indicate some important examples.\\

\subsubsection{Abelian groups}

Let $\mathbf{C} = \mathbf{Ab}$\index{$\mathbf{A}$}, the category of abelian groups (or $\Z$-modules). Then obviously $\mP(\mathbf{C})$ coincides with the set of 
cyclic groups of prime order, so that $\zeta(\mathbf{Ab},s)$ is nothing else than the classical Riemann zeta $\zeta(s)$:
\begin{equation}
\zeta(\mathbf{Ab},s) = \zeta(s).
\end{equation}

\subsubsection{Groups}

If $\mathbf{C}$ is the category of groups, then $\mP(\mathbf{C})$ runs through the finite simple groups. Kurokawa shows that the associated 
zeta function is meromorphic when $\mathrm{Re}(s) > 3$.\\

\subsubsection{Commutative rings}

Let $R$ be a finitely generated commutative ring, and let $\mathbf{C}$ be the category $\mathbf{Mod}(R)$\index{$\mathbf{Mod}(R)$} of $R$-modules. Then
\begin{equation}
\zeta(\mathbf{Mod}(R),s) = \prod_{\frak{m}}\frac{1}{1 - N(\frak{m})^{-s}},
\end{equation}
where $\frak{m}$ runs over all maximal ideals of $R$ and where $N(\frak{m}) = \vert  (R/\frak{m}) \vert$. (A Kurokawa-simple $R$-module is also
simple in the usual sense | there are no proper nonzero $R$-submodules |  and so such a module is isomorphic to $R/\frak{m}$ where $\frak{m}$ is maximal. And one can identify $\mathrm{Hom}(R/\frak{m},R/\frak{m})$ with $R/\frak{m}$.)\\

\medskip
\subsection{$\Upsilon$-Schemes and a conjecture on zeta functions}

 In the incidence geometry over $\mathbb{F}_1$, we have seen (cf. the first chapter of this volume) that if $W \in \mA$, any element of $\underline{\mA}^{-1}(W)$ has many isomorphic copies of $W$, and $W$ is the most general incidence geometry which satisfies the axioms of the class $\underline{\mA}^{-1}(W)$, and which is contained in any element of the latter class.
 In fact, our proposal of base extension (for $\Upsilon$-schemes) is quite in agreement with this idea.
 Perhaps this could also be an approach to define zeta functions of $\mathbb{F}_1$-schemes. Although we state the precise formulation as conjectures, the reader may regard them as ideas, rather then precise predictions. \\
 
\begin{conjecture}[Zeta 1]
The (inverse) zeta function of 
 a $\mD_0$-scheme $S$ should be a divisor of the inverse of any Soul\'{e} zeta function of a $\mathbb{Z}$-scheme $\overline{S}$ which descends to $S$,
  if the latter is a scheme which has a counting polynomial $X$ such that 
 \begin{equation}
 \#(S \otimes_{\mathbb{Z}}\mathbb{F}_q) = X(q)
 \end{equation}
  for any prime power $q$. 
  \end{conjecture}
  
  \begin{conjecture}[Zeta 2]
  \label{Zeta2}
  The zeta function of an $\Upsilon$-scheme is the greatest common divisor of these polynomials.
  \end{conjecture}
  
  In any case, we predict that there is an ``absolute zeta polynomial'' which is independent of the choice of $\overline{S}$ in $\mathbf{C}_S$, which only depends on $S$, and which has similar properties. A Kurokawa-type approach such as in \S \ref{Kurocat} seems promising to attack these conjectures.\\
  
  Consider  the loose graph $\Gamma$ determined by an ordinary quadrangle; then $S(\Gamma)$ is isomorphic to the scheme 
 \begin{equation}
 \Proj(\mathbb{F}_1[X_0,X_1,X_2,X_3]/(X_0X_1,X_2X_3)), 
 \end{equation}
 and a general element of $\overline{\mathbb{S}(\Gamma)}$ has the form 
 \begin{equation}
 \Proj(\mathbb{Z}[X_0,X_1,X_2,X_3]/(\kappa X_0X_1 + \kappa' X_2X_3)), 
 \end{equation}
 for $\kappa,\kappa' \in \mathbb{Z}^\times$, so any such scheme has a counting polynomial, and obviously the function $\chi(s) = s(s - 1)(s - 2)$ is 
 a divisor of all associated inversed Soul\'{e}-zeta functions. The ``degenerate cases'' are schemes of type 
 
  \begin{equation}
 \Proj(\mathbb{Z}[X_0,X_1,X_2,X_3]/(\kappa X_0X_1,\kappa' X_2X_3), 
 \end{equation}
 which are isomorphic to $S(\Gamma)_{\mathbb{Z}}$ in Deitmar's language.

 \medskip
 \subsubsection{Special case: L-schemes}
 
 Let $\Gamma$ be a loose graph, and $S(\Gamma)$ be the associated $\mD_0$-scheme. Can the zeta function of $S(\Gamma)$, in the vein of the previous paragraph, be read from
 $\Gamma$?

\newpage
\section{Motives, absolute motives and regularized determinants}

{\footnotesize
\hspace*{\fill}Parmi toutes les chose math\'{e}matiques que j'avais eu\\
\hspace*{\fill}le privil\`{e}ge de d\'{e}couvrir et d'amener au jour,\\
\hspace*{\fill}cette r\'{e}alit\'{e} des motifs m'appara\^{\i}t encore comme la plus fascinante,\\
\hspace*{\fill}la plus charg\'{e}e de myst\`{e}re | au coeur m\^{e}me \\
\hspace*{\fill}de l'identit\'{e} profonde entre la ``g\'{e}om\'{e}trie'' et l' ``arithm\'{e}tique''.\\
\hspace*{\fill}Et le ``yoga des motifs'' auquel m'a conduit cette r\'{e}alit\'{e}\\ 
\hspace*{\fill}longtemps ignor\'{e}e est peut-\^{e}tre \\
\hspace*{\fill}le plus puissant instrument de d\'{e}courverte que j'aie d\'{e}gag\'{e}\\
\hspace*{\fill}dans cette premi\`{e}re p\'{e}riode de ma vie de math\'{e}maticien\\}
{\footnotesize \hspace*{\fill}A. Grothendieck, {\em R\'{e}coltes et Semailles}}\\

\bigskip
It is hard to see Absolute Arithmetic not deeply connected to Grothendieck's theory of {\em motives}\index{motive}, which is a universal cohomology 
theory $h$ for ``good'' cohomology theories of (say) varieties over fields.\\

The functor $h$ must satisfy:
\begin{itemize}
\item[(M$_1$)]
the K\"{u}nneth formula
\begin{equation}
h(V \times W) = h(V) \otimes h(W);
\end{equation}
\item[(M$_2$)]
translate disjoint unions into direct sums;
\item[(M$_3$)]
certain additional axioms to obtain the Lefschetz formula.
\end{itemize}

Let $k$ be a field.
The universal cohomology theory $h$ should take values in a category of motives $\mathbf{Mot}(k)$\index{$\mathbf{Mot}(k)$} which should look like the category of finite-dimensional $\mathbb{Q}$-vector spaces | more precisely:
\begin{itemize}
\item[(M$_1'$)]
homomorphism groups should be $\mathbb{Q}$-vector spaces;
\item[(M$_2'$)]
$\mathbf{Mot}(k)$ should be an abelian category (or even better a ``tannakian" category, cf. \cite{DelMil}).
\end{itemize}

Every Weil (= ``good'') cohomology theory $\mathbf{H}$ with coefficients in some field $\F$ (such as the \'{e}tale $\ell$-adic one of before) should fit (uniquely) into the following diagram
\begin{equation}
\mathbf{H}: \mathbf{Var}(\F_q) \overset{h}{\longrightarrow} \mathbf{Mot}(\F_q) \overset{\omega_{\mathbf{H}}}{\longrightarrow} \{ \mathrm{graded}\ \F-\mathrm{vector}\ \mathrm{spaces} \}.
\end{equation}
(We denote the category of nonsingular projective varieties over $\F_q$ by $\mathbf{Var}(\F_q)$\index{$\mathbf{Var}(\F_q)$}.)
Here, $\omega_{\mathbf{H}}$ is a functor which comes with the cohomology theory $\mathbf{H}$, such that
\begin{equation}
\omega_{\mathbf{H}}(h(X)) = \mathbf{H}^*(X) = \bigoplus_{i = 0}^{2\mathrm{dim}(X)}H^i(X).
\end{equation}

\medskip
\subsection{Algebraic cycles and the category $\mathbf{Mot}_{\mathrm{num}}$}

An {\em algebraic cycle}\index{algebraic cycle} of a scheme $S$ is an element in the free abelian group $C(S)$ generated by the closed irreducible reduced subschemes of $S$.
The generators, so the closed irreducible reduced subschemes, are called the {\em prime cycles}\index{prime!cycle} of $S$. We assume in this section that the scheme has finite dimension $n$. We then have a natural grading of $C(S)$ by the codimension, by putting $C^d(S)$\index{$C^d(S)$} equal to the free abelian group generated by the closed irreducible reduced subschemes $Z$ of codimension $d = n - \mathrm{dim}(Z)$.  By mapping a closed irreducible reduced subscheme to its generic point, and vice versa, by mapping a point of $S$ to its closure, we see that prime cycles can be identified with points of $S$, so that $C(S)$ is the free abelian group generated by the points of $S$.

Two cycles $\gamma$ and $\gamma'$ (on $X$) are {\em rationally equivalent}\index{rationally equivalent} if there is an algebraic cycle on $X \times \P^1$ having $\gamma - \gamma'$ as its fibre over one point of $\P^1$ and $0$ as its fibre over a second point; any two algebraic cycles are rationally equivalent to algebraic cycles that intersect properly, and so the ``intersection product'' is well defined.
The latter (on passing to quotients by rational equivalence $\sim$) yields a bi-additive map
\begin{equation}
({\mathrm{C}^n(X,Y)}/\sim) \times ({\mathrm{C}^m(X,Y)}/\sim) \longrightarrow {\mathrm{C}^{n + m}(X,Y)}/\sim.
\end{equation}

The group of {\em correspondences}\index{correspondence} of degree $n$ from a $k$-variety $X$ to a $k$-variety $Y$ is now defined as 
\begin{equation}
{\mathrm{Corr}^n(X,Y)} := \mathrm{C}^{n + \mathrm{dim}(X)}(X \times Y).
\end{equation}
\index{${\mathrm{Corr}^n(X,Y)}$}

We further define
\begin{equation}
\widetilde{\mathrm{Corr}^n(X,Y)} := \mathrm{Corr}^n(X,Y)/\sim\ \ \mbox{and}\ \ 
\widetilde{\mathrm{Corr}^n(X,Y)_{\mathbb{Q}}} := \widetilde{\mathrm{Corr}^nX,Y)} \otimes_{\mathbb{Z}} \mathbb{Q}.
\end{equation}
\index{$\widetilde{\mathrm{Corr}^n(X,Y)}$}\index{$\widetilde{\mathrm{Corr}^n(X,Y)_{\mathbb{Q}}}$}

 Another type of equivalence relation we want to consider is {\em numerical equivalence}\index{numerically equivalent}: two cycles $\gamma$ and $\gamma'$ are said to be {\em numerically equivalent} if for any other algebraic cycle $\gamma''$,
\begin{equation}
\mathrm{deg}(\gamma\cdot\gamma'') = \mathrm{deg}(\gamma'\cdot\gamma''),
\end{equation}
where ``$\mathrm{deg}$'' denotes the degree map.

\medskip
\subsection{Two standard conjectures}

\begin{conjecture}[(D)]
\label{stanD}
If an algebraic cycle is numerically equivalent to zero, then its cohomology class is zero.
\end{conjecture}

Assuming Conjecture \ref{stanD}, every Weil cohomology theory $\mathbf{H}$ {\em does} factor uniquely through $X \longrightarrow h(X)$. The next step is wanting more: one desires a decomposition of $h(X)$ that underlies the decomposition of $\mathbf{H}^*(X) = \bigoplus_{i = 0}^{2\mathrm{dim}(X)}H^i(X)$.

\begin{conjecture}[(C)]
\label{stanC}
In $\mathrm{End}(h(X)) = C^{\mathrm{dim}(X)}_{\mathrm{num}}(X\times X)$, the diagonal $\Delta_X$ has a canonical decomposition into a  sum 
of mutually orthogonal idempotents:
\begin{equation}
\Delta_X = \pi_0 + \cdots + \pi_{2\mathrm{dim}(X)}.
\end{equation}
\end{conjecture}
There would follow an expression
\begin{equation}
h(X) = h^0(X) + \cdots + h^{2\mathrm{dim}(X)}(X),
\end{equation}
where $h^j(X) = h(X,\pi_j,0)$. (The latter is to be defined in the next paragraph.)

\begin{remark}
Conjecture \ref{stanC} is known to be true for nonsingular projective varieties over finite fields (where certain polynomials related to the Frobenius map can be used to decompose the motive), and to abelian varieties in characteristic $0$ (where the power maps $m: X \longrightarrow X$, $m \in \Z$ are used to decompose). We will come back to this in more detail later in this section.
\end{remark}

Assume Conjecture \ref{stanC}. We say that the motive $h^j(X)$ has {\em weight}\index{weight} $j$, and that $h(X,\pi_j,m)$ has {\em weight}\index{weight} $i - 2m$.

\medskip
\subsection{Defining motives}

Usually one defines motives in three simple steps (the fact that they {\em do} define motives is less simple). In this section, let $k$ be a fixed field.\\

\medskip
\# \textsc{Step 1 - The category and morphisms}\quad
Define $\underline{\mathbf{Mot}(k)}$\index{$\underline{\mathbf{Mot}(k)}$} to be a category whose objects are one object $h(X)$ per (nonsingular, projective) $k$-variety $X$, and where, for 
$k$-varieties $X$ and $Y$, we define
\begin{equation}
\mathrm{Hom}(h(X),h(Y)) := \widetilde{\mathrm{Corr}^0(X,Y)_{\mathbb{Q}}}.
\end{equation}
Due to the existence of idempotents, this definition is not sufficient for our goals (the category contains morphisms that do not correspond to the direct 
decomposition of an object, so it is far from abelian). One way to fix this is the following.
An additive category is {\em pseudo-abelian}\index{pseudo-abelian!category} if we have the next property:
\begin{itemize}
\item[PA]
For any idempotent (=projector) $p \in \mathrm{Hom}(X,X)$, where $X$ is any object, there exists a kernel $\mathrm{ker}(p)$, and the canonical homomorphism
\begin{equation}
\mathrm{ker}(p) \oplus \mathrm{ker}(\id_X - p) \longrightarrow X 
\end{equation}
is an isomorphism.
\end{itemize}
Given an additive category $\mathbf{C}$, its {\em pseudo-abelian completion}\index{pseudo-abelian!completion} is the category $\widetilde{\mathbf{C}}$
where objects are pairs $(X,p)$ (in the obvious notation); homomorphisms are defined by 
\begin{equation}
\mathrm{Hom}((X,p),(Y,r)) := \{ f \in \mathrm{Hom}_{\mathbf{C}}(X,Y) \vert f \circ p = r \circ f \}/Z,
\end{equation}
where $Z$ is the subgroup consisting of elements $\mu$ for which $\mu \circ p = r \circ \mu = 0$. So
\begin{equation}
\mathrm{Hom}((X,p),(Y,r)) = \{ r \circ f \circ p \vert f \in \mathrm{Hom}_{\mathbf{C}}(X,Y) \}.
\end{equation}

\medskip
\# \textsc{Step 2 - Adding idempotents}\quad
Define $\overline{\mathbf{Mot}(k)}$\index{$\overline{\mathbf{Mot}(k)}$} to be the pseudo-abelian completion of $\underline{\mathbf{Mot}(k)}$.

As such we have introduced the category of {\em effective motives}\index{effective motive} (where we sometimes specify the sub-index ``rat'' or ``num'' depending on whether the equivalence relation ``$\sim$'' is rational equivalence, or numerical equivalence). 

The contravariant functor $h$ associates with $X$ the element $(X,\id)$; a general element $(X,e)$ then corresponds to the motive $e(h(X))$ that is chopped off from $h(X)$ by the projector $e$. For any Weil cohomology theory $\mathbf{H}$, $e$ then singles out a sub vector space $e(\mathbf{H}(X))$ of $\mathbf{H}^*(X)$.

Define $\mathbb{L} := h^2(\P^1)$\index{$\mathbb{L}$}; it is the so-called {\em Lefschetz motive}\index{Lefschetz!motive}.

\medskip
\# \textsc{Step 3 - Adding duals}\quad
We define $\mathbf{Mot}(k)$\index{$\mathbf{Mot}(k)$} by inverting $\mathbb{L}$; the objects are $h(X,e,m)$ with $X$ a $k$-variety, $e$ an idempotent in the ring $\widetilde{\mathrm{Corr}^0(X,X)_{\mathbb{Q}}}$ and $m \in \mathbb{Z}$, and morphisms are defined by 
\begin{equation}
\mathrm{Hom}(h(X,e,m),h(Y,f,n)) = f \circ \widetilde{\mathrm{Corr}^{n - m}(X,Y)_{\mathbb{Q}}} \circ e.
\end{equation}

This is what one calls the {\em category of motives};
if $\sim$ is rational equivalence, one also speaks of {\em Chow motives}\index{Chow motive}, while if $\sim$ is numerical equivalence, we call them
{\em Grothendieck motives}\index{Grothendieck!motive}. Other equivalences are also in use, and we keep applying the term ``motive'' to each of these theories.

\medskip
\subsection{Grothendieck ring, the Lefschetz motive and virtual Tate motives}

Let $k$ be a field.
Consider the category $\widehat{\mathbf{Var}(k)}$\index{$\widehat{\mathbf{Var}(k)}$} of algebraic $k$-varieties, and let $K_0(\widehat{\mathbf{Var}(k)})$\index{$K_0(\widehat{\mathbf{Var}(k)})$} be its {\em Grothendieck ring}.\index{Grothendieck!ring} This is the free abelian group with generators $[X]$, $[X]$ being the isomorphism class of the object $X$ of $\widehat{\mathbf{Var}(k)}$, moding out the subgroup generated by the relations of the form 
\begin{equation}
[X] =  [Y] + [X \setminus Y], 
\end{equation}
where $Y$ is a closed subvariety of $X$. The group is made into a ring  by defining the product of classes $[X]$ and $[Y]$ as 
\begin{equation}
[X]\cdot [Y] := [X \times_{\Spec(k)} Y],
\end{equation}
extended by linearity. One often calles the generators in $K_0(\widehat{\mathbf{Var}(k)})$ ``virtual motives''\index{virtual!motive}. We define in a similar manner the Grothendieck ring $K_0({\mathbf{Mot}(k)})$\index{$K_0({\mathbf{Mot}(k)})$} of the numerical pure $k$-motives.\\

Denoting the class $[\A^0(k)]$ by $\mathbf{1}$, as an example we have 
\begin{equation}
[\P^n(k)] = \mathbf{1} + [\A^1(k)] + [\A^2(k)] + \cdots + [\A^n(k)],
\end{equation}
for $n \in \mathbb{N}$, since $[\P^n(k)] =  [\P^{n - 1}(k)] + [\A^n(k)]$.\\

Let $R$ be a commutative ring.
An {\em additive invariant}\index{additive invariant} $\chi: \widehat{\mathbf{Var}(k)} \longrightarrow R$ is a map with the following properties (where $X, Y$ are objects in $\widehat{\mathbf{Var}(k)}$):
\begin{itemize}
\item[\#]\textsc{Isomorphism Invariance}
$\chi(X) = \chi(Y)$ if $X \cong Y$
\item[\#]\textsc{Multiplicativity}\quad
$\chi(X \times Y) = \chi(X)\chi(Y)$;
\item[\#]\textsc{Inclusion-Exclusion}
$\chi(X) = \chi(Y) + \chi(X \setminus Y)$ for $Y$ closed in $X$.
\end{itemize}
A standard example is the {\em topological Euler characteristic}. It is clear that giving an additive invariant $\chi$ is the same as giving a ring morphism
\begin{equation}
\chi: K_0(\widehat{\mathbf{Var}(k)}) \longrightarrow R.
\end{equation}

\begin{remark}{\rm
Note that there is an additive invariant 
\begin{equation}
\chi_{\mathrm{mot}}: \widehat{\mathbf{Var}(k)} \longrightarrow K_0({\mathbf{Mot}(k)})
\end{equation}
which assigns to a $k$-variety $X$ the object $[(h(X),\mathrm{id},0)]$ (and is generated by these assignments).}
\end{remark}

Let $\mathbb{L}$ be the class $[\A^1(k)]$ in $K_0(\widehat{\mathbf{Var}(k)})$;\footnote{It is suggestively denoted in the same way as the Lefschetz motive, since in $K_0(\widehat{\mathbf{Var}(k)})$ we have the identity $[\P^1(k)] = \mathbf{1} + [\A^1(k)]$.}  
it is called the virtual {\em Lefschetz motive}\index{virtual!Lefschetz motive}.
The subring $\mathbb{Z}[\mathbb{L}] \subset K_0(\widehat{\mathbf{Var}(k)})$ is the subring of virtual {\em mixed Tate motives}\index{virtual!mixed Tate motive}. So a $k$-variety $X$ has a mixed Tate motive if $[X] \in \mathbb{Z}[\mathbb{L}]$. 

\medskip
\subsection{Mixed Tate motives and counting polynomials}

If $X$ is a $\Z$-variety, and $X$ corresponds to a mixed Tate motive (over $k$), it has class $[X \times_{\Spec(\Z)} \Spec(k)] \in \Z[\mathbb{L}]$ in $K_0(\widehat{\mathbf{Var}(k)})$.
Up to a finite number of primes (where bad reduction phenomena could occur), it follows that 
\begin{equation}
N_{p^m}(X) := \#(X \times_{\Spec(\Z)}\Spec(\F_{p^m}))
\end{equation}
(with $p$ a prime and $m$ a nonzero integer) is a polynomial in $p^m$ since the counting function $N_{p^m}(\cdot)$ is an additive
invariant and since $N_{p^m}(\A_1(\F_{p^m})) = p^m$. 

There is more: {\em polynomial countability at all but finitely many primes is conjecturally
equivalent to the motive of the variety being mixed Tate} | the Tate conjecture
predicts that determining $N_p(X)$ for almost all primes $p$ would determine the
motive of $X  \times_{\Z}\F_{p^m}$ \cite{Andre}! So: 

\begin{theorem}
Under the assumption of the Tate conjecture, the $\mathbb{Z}$-schemes
in $\mathbf{C}_S$ which were addressed in Conjecture \ref{Zeta2} are those which have mixed Tate motives.
\end{theorem}

Recalling Kurokawa's $\Fun$-zeta functions, and assuming the Tate conjecture, it follows that schemes of $\Fun$-type
(slightly modifying Kurokawa's definition by allowing bad reduction up to a finite number of primes) are {\em precisely} those which have a 
mixed Tate motive.

\medskip
\subsection{Artin-Tate motives}

A motive $M$ is called an {\em Artin-Tate motive}\index{Artin-Tate motive} if there exists a finite collection of motives $\{ M_i \vert i \in I\}$ such that each tensor power of 
$M$ is a direct sum of Tate twists of the $M_i$s. The category of Artin-Tate motives contains the Tate motives.

\medskip
\subsection{Zeta functions of motives}

In any tannakian category, endomorphisms of objects have characteristic polynomials. We define the {\em zeta function}\index{zeta!function!of motive} $Z(M,T)$ of a pure motive $M$ over the finite field $\F_p$ ($p$ a prime)  of weight $i$ as the characteristic polynomial of the Frobenius map of $M$ if $i$ is odd, or its reciprocal if $i$ is even. (This polynomial has rational coefficients.) For pure motives $M$ and $M'$ of the same weight we have
\begin{equation}
Z(M \oplus M',T) = Z(M,T)Z(M',T),
\end{equation}
which one extends naturally to all motives.

Let $X$ be an $\F_p$-variety, and $\mathbf{H}$ \'{e}tale $\ell$-adic cohomology  ($\ell \ne p$).
We know by the positive solution of the first Weil conjecture (see \cite{GroBour})   that 
\begin{equation}
Z(X,T) = \frac{P_1(T)\cdots P_{2n - 1}(T)}{P_0(T)\cdots P_{2n}(T)}
\end{equation}
with the $P_i(T)$ being the characteristic polynomials of the Frobenius map of $X$ acting on \'{e}tale $\ell$-adic cohomology $H_{\mathrm{et}}^i(X,\mathbb{Q}_{\ell})$.

Assume Conjecture \ref{stanD}.
As the functor $\omega_{\mathbf{H}}$ preserves characteristic polynomials, we have that 
\begin{equation}
Z(h^i(X),T) = P_i(X,T)^{(-1)^{i + 1}}
\end{equation}
so that

\begin{equation}
Z(X,T) = Z(h^0(X),T)\cdots Z(h^{2n}(X),T).
\end{equation}

\medskip
\subsection{Regularized determinants}

Let $c \in \mathbb{C}^\times$, and suppose $c = \vert c\vert e^{i\gamma}$. Put, for $z$ a complex number,
\begin{equation}
c^{-z} := e^{-(\log{\vert c\vert} + i\gamma)z}.
\end{equation}

Then
\begin{equation}
\left.\mathrm{exp}\left(-\frac{d}{dz}c^{-z}\right|_{z = 0}\right) = \mathrm{exp}(\log{\vert c\vert} + i\gamma) = c.
\end{equation}
(Note that the right hand side is independent of the choice of the argument up to an integer multiple of $2\pi i$.) 
For a finite set of nonzero complex numbers $c_\nu$ (and an arbitrary choice of their arguments), we then get 
\begin{equation}
\left.\mathrm{exp}\left(-\frac{d}{dz}\sum_{\nu}c_{\nu}^{-z}\right|_{z = 0}\right) = \prod_{\nu}c_{\nu}.
\end{equation}

Based on this observation, we want to introduce an {\em infinite product}\index{infinite product}, as follows.
Let $(c_\nu)$ be an infinite family of nonzero complex numbers, let for each $\nu$, $\gamma_\nu$ be a corresponding argument, and assume that 
\begin{itemize}
\item
$\sum_{\nu}c_{\nu}^{-z}$ converges for sufficiently large $\mathrm{Re}(z)$;
\item
the function $\sum_{\nu}c_{\nu}^{-z}$ admits a meromorphic continuation and $z$ is regular in a neighborhoud of $z = 0$. 
\end{itemize}

Then, by definition, 
\begin{equation}
\Rprod_{\nu}c_{\nu} := \Rprod_{\nu}(c_\nu,\gamma_\nu) = \left.\mathrm{exp}\left(-\frac{d}{dz}\sum_{\nu}c_{\nu}^{-z}\right|_{z = 0}\right).
\end{equation}
This product does not change if a finite number of arguments are chozen differently, but it may change otherwise.

Let $X$ be a vector space, and $\Phi$ an operator of $X$ with spectrum $\{ c_{\nu}\}$.
We define the {\em regularized determinant}\index{regularized determinant} of $\Phi \vert X$ as
\begin{equation}
\mbox{\textsc{Det}}(s\cdot\mathbf{1} - \Phi \vert X)\index{\textsc{Det}} = \Rprod_{\nu}(s - c_{\nu}).
\end{equation}

Note that for finite products, $\Rprod$ and $\prod$ coincide.

\medskip
\subsection{Deninger's formula}

Let $\mC$ be a nonsingular absolutely irreducible algebraic curve over the finite field $\F_q$. Fix an algebraic closure $\overline{\F_q}$ of $\F_q$ and let $m \ne 0$ be a positive integer; as we noted in the introduction of this chapter, we have the following Lefschetz formula for the number $\vert \mC(\F_{q^m}) \vert$ of rational points over $\F_{q^m}$:

\begin{equation}
\vert \mC(\F_{q^m}) \vert = \sum_{\omega = 0}^2(-1)^{\omega}\mathrm{Tr}(\mathrm{Fr}^m \Big| H^\omega(\mC)) = 1 - \sum_{j = 0}^{2g}\lambda_j^m + q^f,
\end{equation}
where $\mathrm{Fr}$ is the Frobenius endomorphism acting on the \'{e}tale $\ell$-adic cohomology of $\mC$, the $\lambda_j$s are the eigenvalues of this action, and $g$ is the genus of the curve. We then have a motivic {\em weight decomposition}\index{weight!decomposition}
\begin{equation}
\zeta_{\mC}(s) = \prod_{\omega = 0}^2\zeta_{h^{\omega}(\mC)}(s)^{(-1)^{\omega - 1}} 
		= \frac{\prod_{j = 1}^{2g}(1 - \lambda_jq^{-s})}{(1 - q^{-s})(1 - q^{1 - s})} \nonumber \\
\end{equation}

\begin{eqnarray} 
\label{compLefsch}
	 = \frac{\mbox{\textsc{Det}}\Bigl((s\cdot\id - q^{-s}\cdot\mathrm{Fr})\Bigl| H^1(\mC)\Bigr)}{\mbox{\textsc{Det}}\Bigl((s\cdot\id - q^{-1}\cdot\mathrm{Fr})\Bigl| H^0(\mC)\Bigr.\Bigr)\mbox{\textsc{Det}}\Bigl((s\cdot\id - q^{-s}\cdot\mathrm{Fr})\Bigl| H^2(\mC)\Bigr.\Bigr)}.
	\end{eqnarray}
Here the $\omega$-weight component is the zeta function of the pure weight $\omega$ motive $h^{\omega}(\mC)$ of $\mC$.\\

In the early nineties, Christopher Deninger published his studies (\cite{Deninger1991}, \cite{Deninger1992}, \cite{Deninger1994}) on motives and regularized determinants. In \cite{Deninger1992}, Deninger gave a description of conditions on a category of motives that would admit a translation of Weil's proof of the Riemann Hypothesis for function fields of projective curves over finite fields $\F_q$ to the hypothetical curve $\overline{\Spec(\Z)}$. In particular, he showed that the following analogous formula would hold:

\begin{equation}
\label{169}
\zeta_{\overline{\Spec(\Z)}}(s) = 2^{-1/2}\pi^{-s/2}\Gamma(\frac s2)\zeta(s)  
		= \Rprod_\rho\displaystyle\frac{\frac{s - \rho}{2\pi}}{\frac{s}{2\pi}\frac{s - 1}{2\pi}} \overset{?}{=} \nonumber \\
		\end{equation}
		
	\begin{eqnarray} 
	\label{eq2}
	 \frac{\mbox{\textsc{Det}}\Bigl(\frac 1{2\pi}(s\cdot\id - \Theta)\Bigl| H^1(\overline{\Spec(\Z)},*_{\mathrm{abs}})\Bigr.\Bigr)}{\mbox{\textsc{Det}}\Bigl(\frac 1{2\pi}(s\cdot\id - \Theta)\Bigl| H^0(\overline{\Spec(\Z)},*_{\mathrm{abs}})\Bigr.\Bigr)\mbox{\textsc{Det}}\Bigl(\frac 1{2\pi}(s\cdot\id - \Theta)\Bigl| H^2(\overline{\Spec(\Z)},*_{\mathrm{abs}})\Bigr.\Bigr)}, 
	\end{eqnarray}
	where $\Rprod$ is the regularized product,
$\mbox{\textsc{Det}}$ denotes the regularized determinant, $\Theta$ is an absolute Frobenius endomorphism that comes with the category of motives, and the $H^i(\overline{\Spec(\Z)},*_{\mathrm{abs}})$ are certain proposed cohomology groups. The $\rho$s run through the set of critical zeroes of the classical Riemann zeta.

Note that in the first equation of ($\ref{compLefsch}$) we can re-write the factors of type $(1 - q^{-s}\lambda)$ as 
\begin{equation}
1 - q^{-s}\lambda = \Rprod_{\alpha \vert q^{\alpha} = \lambda}\frac{\log{q}}{2\pi i}(s - \alpha).
\end{equation}

The expression (\ref{169}) combined with Kurokawa's work on multiple zeta functions (\cite{Kurokawa1992}) from 1992 subliminated to  the hope that there are motives $h^0$ (``the absolute point''), $h^1$ and $h^2$ (``the absolute Tate motive'') with zeta functions
	\begin{equation}
	\label{eqzeta}
		\zeta_{h^w}(s) \ = \ \mbox{\textsc{Det}}\Bigl(\frac 1{2\pi}(s\cdot\mathbf{1}-\Theta)\Bigl| H^w(\overline{\Spec(\Z)},*_{\mathrm{abs}})\Bigr.\Bigr) 
	\end{equation}
for $w=0,1,2$. Deninger computed that 
\begin{equation}
\zeta_{h^0}(s)=s/2\pi\ \ \mbox{and}\ \  \zeta_{h^2}(s)=(s-1)/2\pi. 
\end{equation}

Manin proposed in \cite{Manin} the interpretation of $h^0$ as $\Spec(\Fun)$ and the interpretation of $h^2$ as the affine line over $\Fun$. Since these observations were made, the search for a proof of the Riemann Hypothesis became a main motivation to look for a geometric theory over $\Fun$. \\

\medskip
\subsection{Absolute motives}

At present, we do not have a construction available of the object 
\begin{equation}
\Spec(\Z) \times_{\Fun} \Spec(\Z)
\end{equation}
which suits our needs, but what we {\em can} do at least is to image how its zeta function looks like. Here, following Manin \cite{Manin}, we 
will consider zeta functions of {\em absolute motives}\index{absolute!motive}.\\

\subsubsection{Structure}

Imagine that ``natural factors'' of zeta functions of $\Z$-schemes of finite type (think of the decompostions in the previous paragraphs for the intuition behind natural factors), correspond to {\em absolute motives}\index{absolute!motive} $\hM$ that can be reconstructed from the zeroes of $\zeta_{\hM}(s)$ in some way. 
We want to have addition and (sometimes) multiplication at hand, which satisfies the natural rules below. We denote the composition of absolute motives by 
``$\underset{\circ}{\times}$'', and the dual of an absolute motive $\hM$ by $\hM^T$. The rules are:

\begin{equation}
\zeta_{\hM + \hN}(s) = \zeta_{\hM}(s)\zeta_{\hN}(s);
\end{equation}

\begin{equation}
\zeta_{\hM \underset{\circ}{\times} \hN}(s) = \zeta_{\hM}(s)\otimes\zeta_{\hN}(s);
\end{equation}

\begin{equation}
\zeta_{\hM^T}(s) = \zeta_{\hM}(-s).
\end{equation}

\medskip
\begin{proposition}
Every motive must define an absolute motive with the same zeta function.
\end{proposition}

Denote the functor which associated to a motive its absolute motive, again, by $\underline{\mA}$\index{$\underline{\mA}$} (and its ``$i$-th factor'' by $\underline{\mA}^i$\index{$\underline{\mA}^i$}).

\medskip
\subsubsection{The  absolute Lefschetz motive}

Manin defines the zeta function of the {\em absolute Lefschetz motive}\index{absolute!Lefschetz motive} $\hL$ as
\begin{equation}
\zeta_{\hL}(s) := \frac{s - 1}{2\pi}.
\end{equation}

We then want that 
\begin{equation}
\zeta_{\hL^{\underset{\circ}\times n}}(s) = \frac{s - n}{2\pi}.
\end{equation}

\begin{remark}{\rm
Sometimes the factor $(2\pi)^{-1}$ is ignored.} \\
\end{remark}

It is imagined that $\hL$ is the motive of the affine line $\Spec(\Fun[X])$ over $\Fun$, and $\hL^0$ corresponds to the absolute point
$\Spec(\Fun)$. So the decomposition of ($\ref{eq2}$) corresponds to the following expression:
\begin{equation}
\underline{\mA}(\overline{\Spec(\Z)}) = \hL^0 \oplus \underline{\mA}^1(\overline{\Spec(\Z)}) \oplus \hL.
\end{equation}

\begin{remark}
{\rm
In \cite{Manin}, Manin calls $\hL$ the {\em absolute Tate motive}\index{absolute!Tate motive} and denotes it by $\hT$. We prefer
to name it ``Lefschetz'' instead of ``Tate'', in accordance with the usual nomenclature for motives. (The inverse of $\hL$ then is the absolute Tate motive.)
}
\end{remark}

\medskip
\subsection{Conjecture C, motivic decomposition and motives for incidence geometries}

Let $X$ be scheme of finite type over $\Z$. As $X$ is Noetherian, there exists a decreasing sequence of closed subschemes
\begin{equation}
X = X_0 \supset X_1 \supset \cdots \supset X_n = \emptyset
\end{equation}
such that each $X_i \setminus X_{i  + 1}$ is an affine scheme of finite type over $\Z$. If we define $Y$ to be  the disjoint union of these affine schemes $X_i \setminus X_{i + 1}$, then $Y$ is an affine scheme for which 
\begin{equation}
N_X(p^m) = N_Y(p^m)
\end{equation}
for any prime $p$, and any positive nonzero integer $m$. (Note that it follows that there exists a family of polynomials $\mP = \{ P_\nu(X_1,\ldots,X_u) \vert \nu \}$ in $\Z[X_1,\ldots,X_u]$ for some positive integer $u$, such that the number $N_{\mP}(p^m)$ of solutions of this family
for any such $p^m$ coincides with $N_X(p^m)$.)\\

\subsubsection{Flag varieties and motivic decomposition}

Recall the notion of ``maximal flag'' from the first chapter of this monograph, and
let $F = (\P_0,\ldots,\P_n)$ be a maximal flag of the $n$-dimensional projective space (scheme) $\P = \P^n(k)$ (over the field $k$). So
\begin{equation}
\emptyset \subset \P_0 \subset \cdots \subset \P_n = \P,
\end{equation}
where $\P^j$ is $j$-dimensional. Note that as a building, $\P^n$ can be defined through its set of maximal flags.
Then each $\P_i \setminus \P_{i - 1}$ is an affine scheme, and the above applies. 
This decomposition corresponds in a precise way to the motivic decomposition 
\begin{equation}
[\P^n(k)] = \mathbf{1} + \mathbb{L} + \mathbb{L}^2 + \cdots + \mathbb{L}^n
\end{equation}
we met earlier. 

The way that a projective space (or vector space) can be represented by its flags, is, for example, through its {\em flag variety}\index{flag variety}.
Let $V$ be a vector space, and let $\mF(V)$ be its set of maximal flags (defined similarly as for projective spaces). Let $\chi = (\chi_0,\chi_1,\ldots,\chi_n)$ be in   $\mF(V)$, where $V$ is an $n$-dimensional $k$-space ($k$ a field, $n$ a positive integer), and where each $\chi_j$ has dimension $j$. (Note that the extremities of the flag are not important for this discussion.) For each integer $d$ with $1 \leq d \leq n - 1$, the Grasmannian\index{Grassmannian} $\wis{Grass}(d,V)$ admits a natural embedding
\begin{equation}
\wis{Grass}(d,V) \hookrightarrow \wedge^d(V)
\end{equation}
which sends an element $W$ to a line of $\wedge^d(V)$. So each $\wis{Grass}(d,V)$ carries the structure of a projective variety. It also follows that $\mF(V)$ is a projective variety through its embedding in $\P(\wedge^1(V) \times \cdots \times \wedge^n(V))$. The flag $\chi$ becomes a projective point in this variety.\\

One is tempted to consider more general varieties which, in some way, can be defined through their set of maximal flags, and try to obtain Conjecture C in the same way as for projective spaces. 
Put more generally: \\

\quad\textsc{Motivic Decomposition.}\quad {\em One wants to deduce similar decompositions of motives for any variety $X$ which admits a filtration
 
\begin{equation}
\emptyset \subset {X}_0 \subset \cdots \subset {X}_n = X,
\end{equation}
of closed subvarieties such that each difference $U_i = X_i \setminus X_{i - 1}$ behaves ``sufficiently well''.}\\

\subsubsection{Abelian varieties}

For abelian varieties, Conjecture C has a neat solution, which we tersely describe here.

In \cite{Sherm} Shermenev gave a decomposition of the Chow motive of an arbitrary abelian variety $\mA$ over an algebraically closed field
\begin{equation}
h(\mA) = \bigoplus_{i = 0}^{2\mathrm{dim}(\mA)}h^i(\mA)
\end{equation}
such that the $\ell$-adic realization of each $h^i(\mA)$ is $H_{\mathrm{et}}^i(\mA,\mathbb{Q}_{\ell})$. 

It should be remarked that such a decomposition is by no means unique (and in this particular case, Shermenev's decomposition is not {\em canonical}). 
In a paper by Deninger and Murre \cite{DenMur}, a {\em canonical functorial} decomposition is given for the more general category of abelian schemes over a smooth quasi-projective scheme defined over a field. In this paragraph we will restrict ourselve to the case of abelian varieties for the sake of convenience.
The crucial observation of \cite{DenMur} is that $\widetilde{\mathrm{Corr}(X,X)} \otimes \mathbb{Q}$ has a decomposition 
into the eigenspaces of the endomorphism $(\id_{\mA} \times n)^*$ where $n$ is an integer. And for $\vert n \vert > 1$
the components of the diagonal with respect to this decomposition are pairwise orthogonal 
idempotents in the ring of correspondences inducing a motivic decomposition.\\

Let $k$ be a field, and let $\mathbf{Var}(k)$ be the category of nonsingular projective varieties $X \longrightarrow \Spec(k)$ over $k$. Let $\widetilde{\mathrm{Corr}(X,X)} =: \mathrm{CH}(X)$\index{$\mathrm{CH}(X)$} be the Chow ring of algebraic cycles of $X$ modulo rational equivalence, and let 
\begin{equation}
\mathrm{CH}(X,\mathbb{Q}) := \mathrm{CH}(X) \otimes \mathbb{Q}.
\end{equation}
\index{$\mathrm{CH}(X,\mathbb{Q})$}
Let $\mA$ be an abelian variety as above (of fibre dimension $g$), and let $T \in \mathbf{Var}(k)$. For $n \in \Z$ define the map $n_T: T \times_k \mA \longrightarrow T \times_k \mA$ to be $\id_T \times n$; for all $p$ we put
\begin{equation}
\mathrm{CH}^p_s(T \times_k \mA,\mathbb{Q}) := \{ \gamma \in \mathrm{CH}^p(T \times_k \mA,\mathbb{Q}) \vert  n_T^*(\gamma) = n^{2p - s}\gamma\ \forall\ n \in \Z\}.
\end{equation}

Then Deninger and Murre show that there is a decomposition 
\begin{equation}
\mathrm{CH}^p(T \times_k \mA,\mathbb{Q}) = \bigoplus_{s = l}^{l'}\mathrm{CH}^p_s(T \times_k \mA,\mathbb{Q}),
\end{equation}
with $l = \mathrm{max}(p - g,2p - 2g)$ and $l' = \mathrm{min}(2p,p + \mathrm{dim}(T))$.

Letting $T = \mA$ and $p = g$ we get that 
\begin{equation}
\mathrm{CH}^g(\mA \times_k \mA,\mathbb{Q}) = \bigoplus_{s = 0}^{2g}\mathrm{CH}^g_s(\mA \times_k \mA,\mathbb{Q}),
\end{equation}
and in particular we have
\begin{equation}
\Delta_\mA = \Delta = \sum_{i = 0}^{2g}\rho_i
\end{equation}
for elements $\rho_i \in \mathrm{CH}^g_i(\mA \times_k \mA,\mathbb{Q})$. Putting $\pi_j := \rho_{2g - j}$, we obtain the next theorem. 

\begin{theorem}[\cite{DenMur}]
There is a unique decomposition 
\begin{equation}
\Delta = \bigoplus_{i = 0}^{2g}\pi_i
\end{equation}
in $\mathrm{CH}^g(\mA \times_k \mA,\mathbb{Q})$
such that $(\id \times n)^*\pi_i = n^i\pi_i$ and $\pi_i \circ \pi_i = \pi_i$ for $i = 0,1,\ldots,2g$ and $n \in \Z$.
\end{theorem}

\subsubsection{Reductive groups}

This is indeed possible for those varieties that come with a ``BN-pair structure'' (which we met in some detail in the first chapter). 

For the convenience of the reader, we quickly recall some basic notions.\\

A group $G$ is said to have
a {\em BN-pair}\index{BN-pair} $(B,N)$, where $B, N$ are subgroups of $G$, if:

\begin{itemize}
\item[(BN1)] 
$\langle B,N \rangle = G$;
\item[(BN2)] 
$H = B \cap N
\lhd N$ and $N/H = W$ is a Coxeter group with distinct generator set of involutions
$S = \{ s_j \vert j \in J \}$;
\item[(BN3)] 
$BsBwB \subseteq BwB \cup
BswB$ whenever $w \in W$ and $s\in S$;
\item[(BN4)] 
$sBs \ne B$ for all $s\in S$.
\end{itemize}

The group $B$, respectively $W$, is a {\em Borel subgroup}\index{Borel subgroup},
respectively the {\em Weyl group}\index{Weyl group}, of $G$. The  quantity $\vert S\vert$
is called the \emph{rank}\index{rank!of BN-pair} of the BN-pair. 
If $W$ is a finite group, the BN-pair is {\em spherical}\index{spherical!BN-pair}. It is {\em irreducible}\index{irreducible!BN-pair} if the corresponding Coxeter system is.
Sometimes we call $(G,B,N)$ also a {\em Tits system}\index{Tits system}.

Note  that $S$ is uniquely determined as the set of elements in $W^\times$ for which
\begin{equation}
B \cup BsB
\end{equation}
is a group.\\

For $w \in W$, let $\ell(w)$ be the {\em length}\index{length} (obviously defined) of $w$ in $W$ with respect to the generating set $S$.\\

To each Tits system $(G,B,N)$ one can associate a building $\mB_{(G,B,N)}$ in a natural way, through a group coset construction. We introduce the standard {\em parabolic subgroups}\index{parabolic subgroup} | these are just the proper subgroups of $G$ which properly contain $B$. Let $I \subset J$, and define 
\begin{equation}
W_I := \langle s_i \vert i \in I \rangle \leq W.
\end{equation}
Then 
\begin{equation}
P_I := BW_IB
\end{equation}
is a subgroup of $G$ which obviously contains $B$, and vice versa it can be shown that any standard parabolic subgroup has this form.

Now define $\mB_{(G,B,N)}$ as follows.

\begin{itemize}
\item[(B1)]
\textsc{Elements}: are elements of the left coset spaces $G/P_I$, $\emptyset \ne I \subset J \ne I$.
\item[(B2)]
\textsc{Incidence}: $gP_I$ is incident with $hP_L$, $I \ne L$, if these cosets intersect nontrivially.
\end{itemize}

The {\em rank}\index{rank!of a building} of $\mB_{(G,B,N)}$ is the rank of the BN-pair.
The building $\mB_{(G,B,N)}$ is {\em spherical}\index{spherical!building} when the BN-pair $(B,N)$ is. It is {\em irreducible}\index{irreducible!building} when $(B,N)$ is irreducible. More details can be found in the first chapter of the author.\\

By work of Karpenko \cite{Karp} (relying on results of Rost \cite{Rost}), and of Chernousov, Gille and Merkurjev \cite{CherGilMer}, we have the following general decomposition result.

\begin{theorem}[\cite{Karp,Rost,CherGilMer}]
Let $X$ be a nonsingular projective $k$-variety ($k$ a field), and suppose that $X$ admits a filtration
\begin{equation}
\emptyset = X_{-1} \subset {X}_0 \subset \cdots \subset {X}_n = X,
\end{equation}
of closed subvarieties, coming with flat morphisms
\begin{equation}
f_i: X_i \setminus X_{i - 1} \longrightarrow Y_i
\end{equation}
of constant relative dimension $n$ for $i = 0,1,\ldots,n$, where the $Y_i$ are smooth projective $k$-varieties. Suppose furthermore that the fiber of every $f_i$ over any point $y \in Y_i$ is isomorphic to $\mathbb{A}^{a_i}_{\kappa(y)}$. Then we have an isomorphism
\begin{equation}
h(X) \cong \bigoplus_{i = 1}^nh(Y_i)(a_i).
\end{equation}
\end{theorem}
In the latter theorem, the $Y_i$ are anistropic spaces over some finite separable field extension of $k$.

\medskip
Split $k$-reductive algebraic groups $G$
 have a  natural  BN-pair structure: let $T$ be any maximal torus; $N$ is the normalizer $N_G(T)$ of $T$ in $G$,
 $W = N_G(T)/C_G(T)$ where $C_G(T)$ is the centralizer, and $B$ is any miminal parabolic containing $T$.

 We have the following application.

\begin{corollary}
Let $G$ be a split $k$-reductive algebraic group ($k$ a field), let $(B,N)$ be a standard BN-pair as above, and put $W = N/(B \cap N)$.
Let $I \subseteq S$ and $P = P_I$ (where $S$ is the distinguished set of generating involutions of $W$). Then we have an isomorphism
\begin{equation}
h(G/P) \cong \bigoplus_{w' \in W/W_I}\mathbb{L}^{\ell(w')}.
\end{equation}
(Where each $\ell(w')$ is the length of some representative.)\\
\end{corollary}

\begin{remark}[Krull-Schmidt theorem]
Let $\mathbf{A}$ be an additive category. An object $A$ of $\mathbf{A}$ is {\em indecomposable}\index{indecomposable} if it is not isomorphic to the direct sum
of two nonzero objects. If $\mathbf{A}$ is pseudo-abelian, $A$ is indecomposable if and only if the endomorphism ring of $A$ in $\mathbf{A}$ has no
nontrivial idempotents. Let again $\mathbf{A}$ be an additive category. We say that $\mathbf{A}$ satisfies the {\em Krull-Schmidt Theorem}\index{Krull-Schmidt Theorem} if any two direct sum decompositions of an object into indecomposable objects are isomorphic. 
\end{remark}

\medskip
\subsection{More general frameworks | back to $\Fun$ and absolute motives}

As any reductive group has the structure of a BN-pair, it is natural to wonder whether any variety which comes from a group with a BN-pair has a canonical
motivic decomposition in much the same way as for reductive groups. On the other hand, since to each Tits system $(G,B,N)$ we can associate a building, maybe there is a synthetic way to derive a canonical motivic decompostion of such a variety. One suggestion would be to extract the decomposition directly from 
the associated diagram.
Recall that groups with a spherical BN-pair of rank at least $3$ were classified by Tits in \cite{Titslect}.
Also, irreducible, spherical Coxeter diagrams were classified by H. S. M. Coxeter \cite{HSMC}, and the complete list is the following.\\

\bigskip
$\mathbf{A}_n$: 
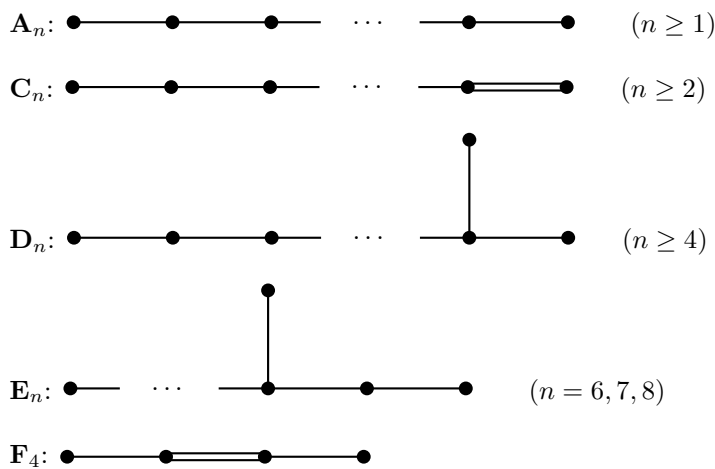
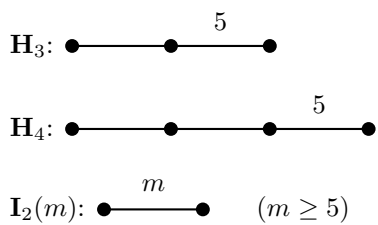
\begin{tikzpicture}[style=thick, scale=1.3]
\foreach \x in {1,2,3,5,6}{
\fill (\x,0) circle (2pt);}

\draw (1,0) -- (2,0);
\draw (2,0) -- (3,0);
\draw (3,0) -- (3.5,0);
\draw (4.5,0) -- (5,0);
\draw (5,0) -- (6,0);
\draw (4,0) node {$\dots$} ;

\end{tikzpicture} \hspace{0.5cm} ($n\geq 1$)\\

$\mathbf{C}_n$: \begin{tikzpicture}[style=thick, scale=1.3]
\foreach \x in {1,2,3,5,6}{
\fill (\x,0) circle (2pt);}

\draw (1,0) -- (2,0);
\draw (2,0) -- (3,0);
\draw (3,0) -- (3.5,0);
\draw (4.5,0) -- (5,0);
\draw (5,0.035) -- (6,0.035);
\draw (5,-0.035) -- (6,-0.035);
\draw (4,0) node {$\dots$} ;

\end{tikzpicture}\hspace{0.5cm} ($n\geq 2$)\\

$\mathbf{D}_n$: \begin{tikzpicture}[style=thick, scale=1.3]
\foreach \x in {1,2,3,5,6}{
\fill (\x,0) circle (2pt);}

\fill (5,1) circle (2pt);

\draw (1,0) -- (2,0);
\draw (2,0) -- (3,0);
\draw (3,0) -- (3.5,0);
\draw (4.5,0) -- (5,0);
\draw (5,0) -- (6,0);
\draw (5,0) -- (5,1);
\draw (4,0) node {$\dots$} ;

\end{tikzpicture}\hspace{0.5cm}  ($n\geq 4$)\\

$\mathbf{E}_n$: \begin{tikzpicture}[style=thick, scale=1.3]
\foreach \x in {0,2,3,4}{
\fill (\x,0) circle (2pt);}

\fill (2,1) circle (2pt);

\draw (0,0) -- (.5,0);
\draw (1.5,0) -- (2,0);
\draw (2,0) -- (3,0);
\draw (3,0) -- (4,0);
\draw (2,0) -- (2,1);
\draw (1,0) node {$\dots$} ;

\end{tikzpicture}
\hspace{0.5cm}  ($n=6,7,8$)\\

$\mathbf{F}_4$: \begin{tikzpicture}[style=thick, scale=1.3]
\foreach \x in {0,1,2,3}{
\fill (\x,0) circle (2pt);}

\draw (0,0) -- (1,0);
\draw (2,0) -- (3,0);
\draw (1,0.035) -- (2,0.035);
\draw (1,-0.035) -- (2,-0.035);

\end{tikzpicture}\\

$\mathbf{H}_3$: \begin{tikzpicture}[style=thick, scale=1.3]
\foreach \x in {0,1,2}{
\fill (\x,0) circle (2pt);}

\draw (0,0) -- (1,0);
\draw (1,0) -- (2,0);
\draw (1.5,.25) node {$5$} ;

\end{tikzpicture}\\

$\mathbf{H}_4$: \begin{tikzpicture}[style=thick, scale=1.3]
\foreach \x in {-1,0,1,2}{
\fill (\x,0) circle (2pt);}
\draw (-1,0) -- (0,0);
\draw (0,0) -- (1,0);
\draw (1,0) -- (2,0);
\draw (1.5,.25) node {$5$} ;

\end{tikzpicture}\\

$\mathbf{I}_2(m)$: \begin{tikzpicture}[style=thick, scale=1.3]
\foreach \x in {1,2}{
\fill (\x,0) circle (2pt);}

\draw (1,0) -- (2,0);
\draw (1.5,.25) node {$m$} ;

\end{tikzpicture}\hspace{0.5cm}  ($m\geq 5$)\\

\medskip
\begin{conjecture}
Any variety which comes as a homogeneous space from a Tits sytem $(G,B,N)$ as $V(G/P)$ with $P$ a parabolic, has 
a canonical decomposition which can be directly derived from its Coxeter diagram.\\
\end{conjecture}

If this conjecture were true, there would be a short cut to derive such a decomposition from below, since $\underline{\mA}(\mB_{(G,B,N)})$
has the same diagram as $\mB_{(G,B,N)}$, cf. the first chapter of this monograph:\\

\begin{conjecture}
Any variety which comes as a homogeneous space from a Tits sytem $(G,B,N)$ as $V(G/P)$ with $P$ a parabolic, has 
a canonical decomposition which can be directly derived from its Weyl geometry.\\
\end{conjecture}

Since groups with a BN-pair are not necessarily associated to varieties (there exist such groups which are not ``classical'' in any sense, by 
free constructions), it seems natural to ask for a theory of motives which are associated to (sufficiently structured) synthetic geometries, such as 
the $\Fun$-incidence geometries we considered in the first chapter (since all these objects come with a diagram), or the loose graphs of the present chapter which yield Deitmar schemes, and 
hence $\Upsilon$-schemes. This could lead to the theory of absolute motives alluded to several times before in the cases when varieties are associated to the incidence geometries in question.

More precisely, let $\Gamma$ be a loose graph, let $S_{\Gamma}$ be the corresponding Deitmar scheme, and let $\mathbf{C}_{S_{\Gamma}}$ be as before (so $(\mathbf{C}_{S_{\Gamma}},S_\Gamma,d)$ is an $\Upsilon$-scheme). \\

\quad$\natural$ \textsc{Question}.\quad 
{\em Can the motives of the objects in $\mathbf{C}_{S_{\Gamma}}$ be canonically decomposed just by the data of $S_{\Gamma}$, or $\Gamma$?
(As a motivating example, let $\Gamma$ be a complete graph, let $(\mathbf{C}_{S_{\Gamma}},S_\Gamma,d)$ be the corresponding projective space $\Upsilon$-scheme, etc.)}

\section{Back to buildings | The hyperring of ad\`{e}le classes, and Singer fields}

In \cite{Connes3}, Connes and Consani relate hyperstructures to the geometry of $\mathbb{F}_1$, giving rise to an interesting connection between hyperfield extensions of the so-called ``Krasner hyperfield'', and sharply transitive actions (on points) of automorphism groups of  combinatorial projective planes (see e.g. \cite{Hall,Pren,KTDZ} for strongly related discussions). (One is also referred to the paper \cite{Order} by the author for far more details on the Connes-Consani connection, sketched in the framework of  ``prime power conjectures''.) The {\em Krasner hyperfield}\index{Krasner hyperfield} $\mathbf{K}$ is the hyperfield $(\{0,1\},+,\cdot)$ with additive neutral element $0$, usual multiplication with identity $ \id$, and satisfying the ``hyperrule''

\begin{equation}
1 + 1 = \{0,1\}.
\end{equation}

In the category of hyperrings, $\mathbf{K}$ can be seen as the natural extension of the commutative pointed monoid $\mathbb{F}_1$, that is, $(\mathbf{K},\cdot) = \mathbb{F}_1$.

 \begin{equation}
  \hspace*{-1cm}
\begin{array}{ccc}
\Spec(\mathbf{K})&\leftarrow&\Spec(\mathbb{Z})\\
&&\\
\downarrow\scriptstyle{}&\swarrow&\\
&&\\
\Spec(\mathbb{F}_1)&&\\


\end{array} 
\end{equation}

It is precisely a number of foundational questions that arise in \cite{Connes3} in the context of classifying hyperfield extensions of $\mathbf{K}$, and that can be traced back to Hall \cite{Hall} in 1947, which we want to consider in this section. Formulated in a rather general form, one set of problems we want 
to study is:\\

\quad$\natural$ \textsc{Question(s)}.\quad
{\em For which fields $\K$  can classical planes $\mathbf{PG}(2,\mathbb{K})$ admit sharply transitive automorphism groups?}\\

The precise connection with the field of characteristic one is explained in the next section.
Very roughly, we will see that a group is in some sense ``defined over the field with one element'' if there exists a projective plane such that this group
acts sharply transitively, as an automorphism group, on its points.

And on the other hand, we also want to know\\

\quad$\natural$ \textsc{Question(s)}.\quad
{\em Which groups $G$ act as  sharply transitive automorphism groups of projective planes?}\\

Here, the sharply transitive action is taken on the points. Groups with this action are called {\em Singer groups}\index{Singer!plane} (of the planes) throughout.

The reader recalls a basic classical result for finite projective planes which states that if a finite projective plane $\Gamma$ of order $n$ admits the 
projective general linear group $\mathbf{PGL}_3(n)$ as an automorphism group (so $n$ is already assumed to be a prime power), then $\Gamma$
comes from a vector space over $\F_q$, i.e., is coordinatized over $\F_q$, and the group acts doubly transitively on both points and lines. In fact, similar results can be obtained when we ask that the group is abstract and acts doubly transitively on points or lines, or even by merely assuming it is big enough.
In other words, once an automorphism group is assumed to be big or classical enough or {\em acts} classically enough, this can only work for a classical plane, and the group contains  the information of only the classical plane (which can be reconstructed from the group) for this action. So we ask ourselves the 
following\\

\quad$\natural$ \textsc{Question}.\quad
{\em Can one group act as a Singer group on nonisomorphic projective planes?}\\

(Whereas one and the same classical plane could admit nonisomorphic Singer groups.)\footnote{It should be noted that for {\em affine} Singer actions, such as sharply transitive actions on the point set of an affine plane, many examples exist of nonisomorphic planes that admit the same Singer group. Even for buildings of higher rank such affine actions are known.} 

The sharply transitive action on points (or lines) is much more rigid than the ``linear actions'', but in the finite case no examples of nonclassical planes are known which admit sharply transitive automorphism groups. Even when the groups are assumed to be:
\begin{itemize}
\item
abelian, or even
\item
cyclic
\end{itemize}
we do not know that the planes are classical, or even have {\em a prime power order}! 
In fact, we will see that by a result of Karzel, ``abelian'' implies ``cyclic''. And as we will also see, this is not true at all in the infinite case: Karzel showed that infinite finitely generated abelian groups can {\em never} act sharply transitively on classical projective planes, and 
we shall prove that virtually {\em all} infinite abelian groups can act as Singer groups on {\em certain} projective planes, but those will never be classical.

Eventually, one of the things we want to understand is what the relation is between the structure of a Singer group of a classical plane, and the field over which the plane is  coordinatized. In particular, at the moment we do not know whether, if such a field is algebraically closed or real-closed, such groups exist (in all the other field cases, we will construct Singer groups explicitly). For the algebraic closures of finite fields, however, we will prove indeed that Singer groups cannot exist.

It seems that some general principle is present: the farther away a field $\F$ is from being algebraically closed, the easier it is to construct (isomorphism classes of) Singer groups for planes of type $\PG(2,\F)$. It would be extremely interesting to find a precise formulation for this principle.

Finally, we want to consider not just planes, but any (also infinite dimensional) projective space for this problem.

\bigskip

\subsection{The hyperring of ad\`{e}le classes}

In a recent paper \cite{Connes3}, the authors discovered unexpected connections between hyperrings and (axiomatic) projective geometry, foreseen with certain group actions.
Denoting the profinite completion of $\mathbb{Z}$ by $  \widehat{\mathbb{Z}}$ (and noting that it is isomorphic to the product $\prod_p\mathbb{Z}_p$ of all $p$-adic integer rings), the {\em integral ad\`{e}le ring}\index{integral!ad\`{e}le ring} is defined as 
\begin{equation}
\mathbb{A}_\mathbb{Z} = \mathbb{R} \times   \widehat{\mathbb{Z}},
 \end{equation}
 
 \noindent
  endowed with a suitable topology.
Let $\K$ be a global field (that is, a finite extension of $\mathbb{Q}$ or the function field of an algebraic curve over $\mathbb{F}_q$ | the latter is a finite field extension of the field of rational functions $\mathbb{F}_q(X)$). The {\em ad\`{e}le ring}\index{ad\`{e}le ring} of $\K$ is 
given by the expression
\begin{equation}
\mathbb{A}_{\K} = {\prod_{\nu}^{}}'\K_{\nu},
\end{equation}
which is the restricted product of local fields $\K_{\nu}$, labeled by the places of $\K$. 
If $\K$ is a number field, the ad\`{e}le ring of $\K$ can also be defined as the 
tensor product
\begin{equation}
\mathbb{A}_\K = \K \otimes_{\mathbb{Z}}\mathbb{A}_\mathbb{Z}.
\end{equation}

We need a few more definitions.

\subsection{Hyperrings and hyperfields}

Let $H$ be a set, and ``$+$'' be a ``hyperoperation'' on $H$, namely a map

\begin{equation}
+: H \times H \rightarrow (2^H)^\times,
\end{equation}

\noindent
where $(2^H)^\times = 2^H \setminus \{\emptyset\}$. For $U, V \subseteq H$, denote $\{ \cup(u + v)\vert u \in U, v \in V\}$ by $U + V$. (Here, we identify an element $h \in H$ with the singleton $\{h\} \subset H$.)
Then $(H,+)$ is an {\em abelian hypergroup}\index{abelian hypergroup} provided the following properties are satisfied:

\begin{itemize}
\item
$x + y = y + x$ for all $x, y \in H$;
\item
$(x + y) + z = x + (y + z)$ for all $x, y, z \in H$; 
\item
there is an element $0 \in H$ such that $x + 0 = 0 +  x$ for all $x \in H$;
\item
for all $x \in H$ there is a unique $y \in H$ ($=: -x$) such that $0 \in x + y$;
\item
$x \in y + z$ $\Longrightarrow$ $z \in x - y$.
\end{itemize}

\begin{proposition}[\cite{Connes3}]
Let $(G,\cdot)$ be an abelian group, and let $K \subseteq \Aut(G)$. Then the following operation defines a hypergroup structure on $H = \{g^K\vert g \in G\}$:

\begin{equation}
g_1^{K}\cdot g_2^K := (g_1^K\cdot g_2^K)/K.
\end{equation}
\end{proposition}

A {\em hyperring}\index{hyperring} $(\mathbf{R},+,\cdot)$ is a nonempty set $\mathbf{R}$ endowed with a hyperaddition $+$ and the usual multiplication $\cdot$ such that:

\begin{itemize}
\item
$(\mathbf{R},+)$ is an abelian hypergroup with neutral element $0$;
\item
$(\mathbf{R},\cdot)$ is a monoid with multiplicative identity $ \id$;
\item
for all $u, v, w \in \mathbf{R}$ we have that $u\cdot(v + w) = u\cdot v + u\cdot w$ and $(v + w)\cdot u = v\cdot u + w\cdot u$;
\item
for all $u \in \mathbf{R}$ we have that $u\cdot 0 = 0 = 0\cdot u$;
\item
$0 \ne  \id$.
\end{itemize}

A hyperring $(\mathbf{R},+,\cdot)$ is a {\em hyperfield}\index{hyperfield} if $(\mathbf{R} \setminus \{0\},\cdot)$ is a group.

\subsection{The Krasner hyperfield and its extensions}

The {\em Krasner hyperfield}\index{Krasner hyperfield} $\mathbf{K}$ is the hyperfield 
\begin{equation}
(\{0,1\},+,\cdot)
\end{equation}
with additive neutral element $0$, usual multiplication with identity $ \id$, and satisfying the hyperrule

\begin{equation}
1 + 1 = \{0,1\}.
\end{equation}

In the category of hyperrings, $\mathbf{K}$ can be seen as the natural extension of the commutative pointed monoid $\mathbb{F}_1$, that is, $(\mathbf{K},\cdot) = \mathbb{F}_1$. 
As remarked in \cite{Connes5}, the Krasner hyperfield {\em encodes the arithmetic of zero and nonzero numbers}, just as $\mathbb{F}_2$ does for even and odd numbers.
From this viewpoint, it is of no surprise that projective geometry will come into play.


Let $\mathbf{R}$ be a commutative ring, and let $G$ be a subgroup of its multiplicative group. The following operations define a hyperring on the set $\mathbf{R}/G$ of $G$-orbits
in $\mathbf{R}$ under multiplication.

\begin{itemize}
\item
\textsc{Hyperaddition}.\quad
$x + y := (xG + yG)/G = \{ xg + yh \vert g,h \in G \}/G$ for $x, y \in \mathbf{R}/G$.
\item
\textsc{Multiplication}.\quad
$xG\cdot yG := xyG$ for $x, y \in \mathbf{R}/G$.
\end{itemize}

\medskip
\quad\textsc{Important Example}.\quad
Let $\mathbf{R}$ be the finite field $\F_{q^m}$, where $q$ is a prime power and $m \in \mathbb{N}^\times$, and let $G$ be the multiplicative group
$\F_q^{\times} \leq \F_{q^m}^{\times}$. Then we can see $\mathbf{R}$ naturally as an $m$-dimensional $\F_q$-vector space, or better: as an $(m - 1)$-dimensional $\F_q$-projective space. In the latter case, projective points are the cosets $xG$ with $x \ne 0$. And lines, for instance, are of the form $(xG + yG)/G$.
Once one lets $m$ go to $1$, one naturally constructs the Krasner hyperfield $\mathbf{K}$.
These examples will be very important in what is to come. \\

\begin{proposition}[\cite{Connes3}]
Let $\mathbb{K}$ be a field with at least three elements. Then the hyperring $\mathbb{K}/\mathbb{K}^\times$ is isomorphic to the Krasner hyperfield.
If, in general, $\mathbf{R}$ is a commutative ring and $G \subset \mathbb{K}^\times$ is a proper subgroup of the group of units of $\mathbf{R}$, then the hyperring $\mathbf{R}/G$ defined as above contains $\mathbf{K}$ as a subhyperfield if and only if $\{0\} \cup G$ is a subfield of $\mathbf{R}$. 
\end{proposition}

\begin{remark}
{\rm Consider a global field $\mathbb{K}$. Its ad\`{e}le class space $\mathbb{H}_\mathbb{K} = \mathbb{A}_{\mathbb{K}}/\mathbb{K}^\times$ is the quotient of a commutative ring $ \mathbb{A}_{\mathbb{K}}$ by $G = \mathbb{K}^\times$, and $\{0\} \cup G = \mathbb{K}$, so it is a hyperring extension of $\mathbf{K}$.}\\
\end{remark}

A {\em $\mathbf{K}$-vector space}\index{$\mathbf{K}$-vector space} is a hypergroup $E$ provided with a compatible action of $\mathbf{K}$. As $0 \in \mathbf{K}$ acts by retraction (to 
$\{0\} \subset E$) and $ \mathrm{id} \in \mathbf{K}$ acts as the identity on $E$, the $\mathbf{K}$-vector space structure is completely determined by the hypergroup structure.
It follows that a hypergroup $E$ is a $\mathbf{K}$-vector space if and only if 

\begin{equation}
x + x = \{0,x\}\ \ \mathrm{for}\ \ x \ne 0.
\end{equation}

\medskip
Let $E$ be a $\mathbf{K}$-vector space, and define $\mP := E \setminus \{0\}$. For $x, y \ne x \in \mP$, define the {\em line} $L(x,y)$ as 

\begin{equation}
x + y \cup \{x,y\}.  
\end{equation}

\noindent
It can be easily shown | see \cite{Pren} | that $(\mP,\{L(x,y) \vert x, y\ne x \in \mathbf{P}\})$ is a projective space. Conversely, if $(\mP,\mL)$ is the point-line geometry of a projective space with at least $4$ points per line, then a hyperaddition on $E := \mP \cup \{0\}$ can be defined as follows:

\begin{equation}
x + y = xy \setminus \{x,y\}\ \ \mathrm{for}\ \ x \ne y,\ \ \mathrm{and}\ \ x + x =\{0,x\}.
\end{equation}

\medskip
Now let $\mathbb{H}$ be a hyperfield extension of $\mathbf{K}$, and let $(\mP,\mL)$ be the point-line geometry of the associated projective space; then Connes and Consani \cite{Connes3} show that $\mathbb{H}^\times$ induces a so-called ``two-sided incidence group'' (and conversely, starting from such a group $G$, there is a unique hyperfield extension $\mathbb{H}$ of $\mathbf{K}$ such that $\mathbb{H} = G \cup \{0\}$). By the Veblen-Young result (cf. the introductory chapter of this book), this connection is reflected by the next theorem for the finite case.

\begin{proposition}[\cite{Connes3}]
\label{CCH}
Let $\mathbb{H} \supset \mathbf{K}$ be a finite commutative hyperfield extension of $\mathbf{K}$. Then one of the following cases occurs:
\begin{itemize}
\item[{\rm (i)}]
$\mathbb{H} = \mathbf{K}[G]$ for a finite abelian group $G$.
\item[{\rm (ii)}]
There exists a finite field extension $\mathbb{F}_q \subseteq \mathbb{F}_{q^m}$ such that $\mathbb{H} = \mathbb{F}_{q^m}/\mathbb{F}_q^\times$.
\item[{\rm (iii)}]
There exists a finite nonDesarguesian projective plane admitting a sharply point-transitive automorphism group $G$, and $G$ is the abelian incidence group associated to $\mathbb{H}$. 
\end{itemize}
\end{proposition}


In case (i), there is only one line (otherwise we have to be in the other cases), so for all $x, y, x',  y' \in \mathbb{H} \setminus \{0\}$ with $x \ne y$ and $x' \ne y'$, we must have

\begin{equation}
L(x,y) =  (x + y) \cup \{x,y\} = (x' + y') \cup \{x',y'\} = L(x' + y') = \mathbb{H} \setminus \{0\}. 
\end{equation}

\medskip
\noindent
In other words, hyperaddition is completely determined:

\begin{equation}
\left\{\begin{array}{cc}
x + 0 = x  &\mathrm{for}\ \ x \in \mathbb{H}\\
x + x = \{0,x\}    &\mathrm{for}\ \ x \in \mathbb{H}^\times\\
x + y = \mathbb{H}  \setminus \{0,x,y\}   &\mathrm{for}\ \ x \ne y \in \mathbb{H}^\times\\
\end{array}\right.
\end{equation}

\medskip
\subsection{Recent developments}

There exist infinite hyperfield extensions $\mathbb{H} \supset \mathbf{K}$ for which $\mathbb{H}^\times \cong \mathbb{Z}$ and not coming from Desarguesian projective spaces in the above sense, see M. Hall \cite{Hall}, and the next section. This remark, together with the following general version of 
Theorem \ref{CCH} (see the remark before that theorem), is the starting point of the paper \cite{HyperSinger}.

\begin{proposition}[\cite{Connes3}]
Let $\mathbb{H} \supset \mathbf{K}$ be a hyperfield extension of the Krasner hyperfield $\mathbf{K}$. Then there exists a  projective space admitting a sharply point-transitive automorphism group $A$, and $A$ is the incidence group associated to $\mathbb{H}$. 
\end{proposition}

The space could be nonDesarguesian if its dimension is two.
If the dimension of the space is at least three, we know the space {\em is} coordinatized over a skew field by the Veblen-Young result, but when one does not assume the group to be commutative for these spaces, not much seems to be known about such actions. And in the planar case, we can have axiomatic projective planes which are not associated to vector $3$-spaces (over some skew field), and by Hall's result, such planes could admit extremely strange sharply transitive automorphism groups, such as the infinite cyclic group $\mathbb{Z}, +$.

\bigskip
\subsection{What is known}

Karzel proves the following (answering a more general version of a question of Hall \cite{Hall}):

\begin{theorem}[H. Karzel \cite{Karzel}]
\label{Karz}
Let $S$ be a finitely generated commutative Singer group of $\PG(m,\F)$, with $m \in \mathbb{N}^\times$ and different from $1$,
and $\F$ a field. Then $\F$ is finite, and $S$ is cyclic.
\end{theorem}

So Desarguesian projective spaces (different from projective lines) can only allow  commutative Singer groups which are 
infinitely generated. Later, we will construct commutative Singer groups with this property for spaces $\PG(m,\F)$ for many values of $(m,\F)$.

On the other hand, we will show in the next section that virtually any infinite  commutative group (those that do not have involutions) can 
act as a Singer group on an appropriate projective plane | so also the finitely generated examples | but the planes are not 
Desarguesian by Karzel's result. Our result is a corollary of a theorem which generalizes the next result of Hughes:

\begin{theorem}[D. R. Hughes \cite{Hughes}]
\label{Hugh}
Let $H$ be a countably infinite group, and
assume the following properties for $H$:
\begin{equation}
\left\{\begin{array}{ccccc}
(h_1) &h^2 &\ne & \id &\forall h \in H^\times \\
(h_2) &\vert h^H \vert &= &\infty &\forall h \in H \setminus Z(H) \\
(h_3) &\#\{ x \vert x^2 = h'\}&<&\infty &\forall h' \in H \\
\end{array}\right.
\end{equation}
Here, $Z(H)$ is the center of $H$. Then $H$ acts as a Singer group on some projective plane.
\end{theorem}

Hughes applied Theorem \ref{Hugh} to show that free groups with a finite number $n$ of generators ($n \geq 2$) act as Singer groups of 
certain projective planes. Later, we will obtain this result for any free group.\\

For vector spaces, one could also consider the related problem of studying sharply transitive automorphism groups of the nonzero vectors. 
We recall the following nice result.

\begin{theorem}[\cite{Cheretal}]
Let $\mathbb{K}$ be an algebraically closed field and $G$ a subgroup of $\mathbf{GL}_n(\mathbb{K})$ which acts sharply transitively on the set of nonzero vectors in $\mathbb{K}^n$, where $n \in \mathbb{N}^\times$. Then either $n = 1$ and $G = \mathbb{K}^\times$, or $n = 2$, and $G$ can be precisely described.
\end{theorem} 

The vector space problem relates to our problem as follows.
Let $V$ be an $\F$-vector space over the field $\F$, and let $\PG(V)$ be the corresponding projective space. If $K$ is a Singer group of $\PG(V)$, then if $\overline{K}$
is the corresponding automorphism group of $\PG(V)$ (that is, $\overline{K}$ contains the full scalar group $S$ and $\overline{K}/S = K$), the latter acts 
sharply transitively on the nonzero vectors. And if $K$ is linear, $\overline{K}$ is linear as well. (Of course, if $\F$ is algebraically closed, Singer groups cannot exist
due to the fact that all polynomials over $\F$ have roots in $\F$, cf. later sections.) Conversely, if $H$ acts sharply transitively on the nonzero vectors of 
the vector space $V$, and $H$ is abelian, then $H$ contains all scalar automorphisms $S$,\footnote{If an element of $H$ fixes some vector line, it must fix all vector lines as $H$ is abelian and transitive on vector lines, so due to the transitivity on nonzero vectors, $H$ contains all scalar automorphisms.} and $H/S$ induces an abelian Singer group of the associated projective space.

As we will construct Singer groups for ``most'' projective spaces, we will get the sharply transitive groups of the vector space for free.

\bigskip
\subsection{Construction of Singer groups}

In this section, we slightly generalize the result of Hughes on planar difference sets in not necessarily abelian groups,
by removing the assumption on countability. The results are taken from \cite{HyperSinger}.\\

\subsubsection{Partial difference sets}

If $G$ is a group (written multiplicatively)  and $C \subseteq G$, denote by $D(C)$ the set of ``differences'' $\{ c{c'}^{-1} \vert c, c' \in C\}$. Now assume that $K$ is a group, 
and $S$ a subset such that for any $k \in K^\times$, there is precisely one couple $(a,b) \in S \times S$ such that $k = ab^{-1}$ | in other words, the map
\begin{equation}
\phi: S \times S \setminus \mathrm{diagonal} \longrightarrow K^\times: (a,b) \longrightarrow ab^{-1}
\end{equation}
is a bijection (and as a consequence, $D(C) = K$). We call $S$ a {\em difference set}\index{difference set} in $G$. If the map $\phi$ merely is injective, we call $S$ a {\em partial difference set}\index{partial difference set}.
Then defining a ``point set'' $\mP = G$ and ``line set'' $\mB = G$, where a point $x$ is incident with a line $y$ (and we write $x \I y$) if and only if $xy^{-1} \in S$, we 
obtain a projective plane $\Gamma(G,S)$\index{$\Gamma(G,S)$} with the special property that  $G$ acts (by right translation) as a sharply point transitive automorphism group (= Singer group).\footnote{For $g \in G$, we have $a \I b$ if and only if $ab^{-1} \in S$ if and only if $(ag)(g^{-1}b^{-1}) \in S$ if and only if $ag \I bg$.} And
conversely, a projective plane admitting a Singer group $G$ can always be constructed in this way from a difference set $S \subset G$.

We will use the following easy lemma without reference.
\begin{lemma}
If ${(S_\omega)}_{\omega \in \Omega}$ is a chain of partial difference sets in a group $K$ (so $\Omega$ is well ordered and from 
$\nu < \mu$ follows that $S_{\nu} \subseteq S_{\mu}$), then $\cup_{\omega \in \Omega}S_{\omega}$ also is a partial difference set.\\ 
\end{lemma}

\medskip
\subsubsection{Ordinals}

Each {\em ordinal}\index{ordinal} is the well-ordered set of all smaller ordinals. The smallest ordinal is $0 = \emptyset$,  the next-smallest ordinal is
$1 = \{0\} = \{\emptyset\}$, followed by
$2 = \{0,1\} = \{\emptyset,\{\emptyset\}\}$, etc. After all finite ordinals have been constructed, we continue with
$\omega = \{0,1, 2,\ldots \}$, $\omega + 1 = \{0,1,2,\ldots\} \cup \{\omega\}$, and so on,
and eventually $2\omega = \{0,1,2,\ldots\} \cup \{\omega,\omega + 1,\omega + 2,\ldots \}$.
This is followed by $2\omega + 1,2\omega + 2, \ldots, \omega^2$. And so on. All the ordinals we have mentioned so far are countably infinite. After all countable ordinals have been defined, we
meet the first uncountable ordinal, denoted $\omega_1$; later we reach $\omega_2$, etc. Let $\gamma$ be an ordinal. 
Then the {\em successor}\index{successor} of $\gamma$ is 
\begin{equation}
\gamma + 1 =  \gamma \cup \{\gamma\}, 
\end{equation}
this being the smallest ordinal exceeding $\gamma$. Every ordinal is either a successor ordinal or a {\em limit ordinal}\index{limit ordinal}, but never both. 
A {\em limit ordinal}\index{limit ordinal} is an ordinal $\alpha$ such that 
\begin{equation}
\alpha = \bigcup_{\beta < \alpha}\beta.
\end{equation} 
Now an arbitrary set $S$ may be indexed as $S = \{ s_{a} \vert a \in A\}$,
where $A$ is an ordinal. Moreover we may assume $A$ to be {\em minimal} among all
ordinals of cardinality $\vert A \vert$ | otherwise we may simply re-index suitably. 

\medskip
\subsubsection{Construction}

Let $H$ be an infinite group, and let $\vert H\vert = A$ be the smallest ordinal of cardinality $\vert H\vert$;
write $H = \{ h_{\alpha} \vert \alpha \in A \}$ ($A$ is well ordened).

Define for each $\gamma \in A$ a set $S_{\gamma}$ such that
\begin{itemize}
\item[(i)]
$\vert S_{\gamma}\vert \leq \vert\gamma\vert < \vert A\vert$;
\item[(ii)]
$h_{\gamma} \in D(S_{\gamma})$ for $\gamma \in A$;
\item[(iii)]
$S_{\gamma}$ is a partial difference set of $H$;
\item[(iv)]
$S_{\gamma} \subseteq S_{\beta}$ for $\gamma < \beta$ and $\beta \in A$.
\end{itemize}

If $\alpha$ is a limit ordinal, define $S_{\alpha} = \cup_{\beta < \alpha}S_{\beta}$. (Note that $\vert S_{\alpha}\vert \leq \vert \alpha \vert \cdot \vert \alpha \vert$.) 
Now let $\alpha$ be a successor ordinal $\alpha = \beta + 1$; if $h_{\alpha} \in D(S_{\beta})$, put $S_{\alpha} = S_{\beta}$.
Otherwise, we construct $S_{\alpha}$ by adding two new elements to $S_{\beta}$ such that $h_{\alpha} \in D(S_{\alpha})$.

We seek properties for $H$ such that this particular step (and then the whole construction) can be carried out.
Let $S_{\beta} = \{ s_i \vert i \in I \}$ with $\vert I \vert < \vert A \vert$. (We suppose without loss of generality that $\vert I \vert \not\in \mathbb{N}$.) 
Suppose $d = h_{\alpha} \not\in D(S_{\beta})$.

Assume the following properties for $H$:
\begin{equation}
\left\{\begin{array}{ccccc}
(d_1) &h^2 &\ne & \id &\forall h \in H^\times \\
(d_2) &\vert h^H \vert &= &\vert H \vert &\forall h \in H \setminus Z(H) \\
(d_3) &\#\{ x \vert x^2 = h'\}&<&\vert H\vert &\forall h' \in H \\
\end{array}\right.
\end{equation}

Note that a planar  Singer group never can have involutions (if $\sigma$ would be such an involution and $L$ is a line of the plane,
$L \cap L^{\sigma}$ would be a fixed point), so the first property is necessary.\\

We have the following theorem:

\begin{theorem}[Construction, \cite{HyperSinger}]
If an infinite group $H$ satisfies the following properties, then $H$ acts as a Singer group on some projective plane.
\begin{equation}
\left\{\begin{array}{ccccc}
(d_1) &h^2 &\ne & \id &\forall h \in H^\times \\
(d_2) &\vert h^H \vert &= &\vert H \vert &\forall h \in H \setminus Z(H) \\
(d_3) &\#\{ x \vert x^2 = h'\}&<&\vert H\vert &\forall h' \in H \\
\end{array}\right.
\end{equation}
\end{theorem}

\medskip
\subsubsection{An example: general free groups}

In \cite{Hughes}, Hughes showed that free groups on a finite number of generators satisfy the properties (h$_1$)-(h$_2$)-(h$_3$) of Theorem \ref{Hugh}, 
so that they act on certain projective planes as Singer groups. Now let $\hF(\Omega)$\index{$\hF(\Omega)$} be a free group with generator set $\Omega$, where $\Omega$
is any infinite alphabet. For any reduced element $f \in \hF(\Omega)$, let $\pi(f)$ be the subset of $\Omega \cup \Omega^{-1}$ of letters used in $f$. 

First of all, note that $\hF(\Omega)$ cannot have involutions since if $x$ would be an involution, it also would be an involution in $\hF(\pi(x))$, which 
contradicts Hughes's result.

Next, let $h$ be any nontrivial element in $\hF(\Omega)$, and consider the equation
\begin{equation}
x^2 = h.
\end{equation}
Hughes shows in \cite{Hughes} that this equation has a unique solution in a free group $\hF(S)$ where $\pi(h) \subseteq S$ and $S$ is finite, so
it follows easily that it also has a unique solution in $\hF(\Omega)$.

Next we want to consider orbits $g^{\hF(\Omega)}$. Define $\Omega' := \Omega \setminus \pi(g)$, and define the set
\begin{equation}
\xi(g) := \{g^{\omega}:= \omega^{-1}g\omega \vert \omega \in \Omega'\} \subset g^{\hF(\Omega)}.
\end{equation}
It follows that
\begin{equation}
\vert \xi(g)\vert = \vert \Omega' \vert = \vert \Omega\vert = \vert \hF(\Omega)\vert.
\end{equation}

So indeed (d$_1$)-(d$_2$)-(d$_3$) are satisfied, and whence $\hF(\Omega)$ acts as a Singer group on some plane.

\bigskip
\subsection{Construction of difference sets | Abelian case}

By ($d_2$), one would expect that the previous section would not apply to the abelian case, but this is, in fact, not entirely true.

For suppose $H$ is abelian now, without involutions (cf. ($d_1$)).
If, as above, we want to add $x$ and $d^{-1}x$ to $S_{\beta}$ to obtain a partial difference set $S_{\beta} \cup \{ x,d^{-1}x\}$ for which $d$
is a difference, we need to find an $x$ for which $d^x \ne s_j^{-1}s_i$; but from $d^x = {s_j}^{-1}s_i$ we would have $d^x = d = s_j^{-1}s_i = s_is_j^{-1}$ since $H$ is abelian, contradiction
since $d \not\in D(S_{\beta})$. So ($d_2$) is not needed here. Secondly, suppose 
\begin{equation}
\#\{ x \not\in S_{\beta} \vert  s_j^{-1}ds_i = (s_j^{-1}x)^2\ \mbox{for some}\ s_i,s_j \in S_{\beta}\} > \vert S_{\beta} \vert.
\end{equation}
Then obviously we can find an $s_\ell \in S_{\beta}$ and different $z, z' \not\in S_{\beta}$ such that $(s_{\ell}^{-1}z)^2 = (s_{\ell}^{-1}{z'})^2$,
implying that $z{z'}^{-1}$ is an involution, contradiction. So for abelian groups, ($d_3$) need not be assumed since it follows (in the context that we need it) from
($d_1$).

So for the abelian case, we obtain the most general constructive result as possible:

\begin{theorem}[Abelian Singer groups | Characterization, \cite{HyperSinger}]
An infinite abelian group acts as a Singer group on some projective plane if and only if it contains no involutions.
\end{theorem}

\bigskip
\subsection{Singer groups for classical spaces}
\label{class}

Let $\F$ be any field,  and suppose $\F'/\F$ is a proper  field extension. Then $\F'$ can be naturally seen
as an $\F$-vector space $V(\F')$ as before. Now ${\F'}^{\times}$ acts by multiplication on $V(\F')$, and clearly this induces a subgroup of $\GL(V(\F'))$ which
acts sharply transitively on the nonzero vectors. The subgroup $\F^{\times} \leq {\F'}^{\times}$ acts as scalars, and ${\F'}^{\times}/\F^{\times}$ induces 
a sharply transitive group on the points of the projective space $\PG(V(\F'))$. 
If the degree of $\F'/\F$ is a natural nonzero number $m$ (which is at least $2$), then $\PG(V(\F')) = \PG(m - 1,\F)$. We have

\begin{theorem}[\cite{HyperSinger}]
If $\omega = [\F' : \F]$ is the not necessarily finite degree of the field extension $\F'/\F$, then $\PG(\omega - 1,\F)$ allows a linear Singer  group.
In particular, if $\omega = 3$, this applies to the Desarguesian plane $\PG(2,\F)$.
\end{theorem}

\begin{corollary}[Singer|Algebraic closure principle, \cite{HyperSinger}]
The farther away a field $\F$ is from its algebraic closure, the more Desarguesian projective spaces over $\F$ allow a (linear) Singer group in this 
construction. And the more isomorphism classes of such groups arise. 
\end{corollary}

\begin{corollary}[\cite{HyperSinger}]
If a field $\F$ is not real-closed or algebraically closed, then $\PG(2,\F)$ admits a (linear) Singer group.
\end{corollary}
\begin{proof}
If any polynomial of degree $3$ over $\F$ has a root in $\F$, then $\F$ is either real-closed or algebraically closed. 
\end{proof}

(In the next section, one can find more formal information about real-closed fields.)

\bigskip
\subsection{Structural theorems and nonexistence}

For some fields $\F$, it is rather easy to exclude the existence of Singer groups for $\PG(n - 1,\F)$, $n \in \mathbb{N} \setminus \{0,1,2\}$, $n$ odd. 

\begin{theorem}[\cite{HyperSinger}]
\label{struct}
Suppose $\overline{\F}$ is such that $[\overline{\F} : \F]$ is finite of degree $m \ne 1$. Suppose furthermore that 
\begin{equation}
\vert \Aut(\F) \vert < \vert \F\vert 
\end{equation}
(where $\vert \cdot \vert$ denotes cardinality and $\Aut(\cdot)$ the automorphism group).  Then $\PG(n - 1,\F)$ does not admit a Singer group, where  $n \in \mathbb{N} \setminus \{0,1,2\}$, $n$ odd.
\end{theorem}

(In the above, it makes no sense to allow the extension $[\overline{\F} : \F] = 1$, since $\vert \Aut(\overline{\F})\vert = \#2^{\overline{\F}}$.)
As $\vert \Aut(\mathbb{R}) \vert = 1$, we have the following corollary. 

\begin{corollary}[\cite{HyperSinger}]
If $\F$ is real-closed and the positive odd integer $n$ is at least $3$, $\PG(n - 1,\F)$ cannot admit Singer groups if $\vert \Aut(\F) \vert < \vert \F\vert$. In particular, 
$\PG(2,\mathbb{R})$ has no Singer groups.
\end{corollary}

Now let $\F$ be a real-closed field. Then for any $k \in \F$, either $k$ or $-k$ is in $\F^2$ (the set of squares). There is a unique 
total order $\leq$ on $\F$ defined by:
\begin{equation}
0 \leq y \ \mbox{if and only if}\ y \in \F^2.
\end{equation}
(The order is unique as squares must be positive with respect to a total order.) Let $\alpha \in \Aut(\F)$; as $\alpha(k^2) = \alpha(k)^2$ for any
$k \in \F$, $\alpha$ preserves the order (as $a < b$ if and only if there is a nonzero square $c^2$ such that $a + c^2 = b$).

The next result detects certain real-closed fields with trivial automorphism groups.

\begin{theorem}[\cite{HyperSinger}]
Let $\F$ be a real-closed field which is a subfield of $\mathbb{R}$. Then $\Aut(\F)$ is trivial.
\end{theorem}
\begin{proof}
Let $\beta \in \Aut(\F)$; then $\mathbb{Q} \leq \F$ is fixed elementwise, and $\beta$ preserves the unique total order $\leq$ on $\F$ (which is 
the one inherited by $\mathbb{R}$). For $\kappa \in \F$, define $\mathbb{Q}^+(\kappa) := \{ q \in \mathbb{Q} \vert q \geq \kappa \}$ and 
$\mathbb{Q}^-(\kappa) := \{ q \in \mathbb{Q} \vert q \leq \kappa \}$. Note that both $\mathbb{Q}^+(\kappa)$ and $\mathbb{Q}^-(\kappa)$ are uniquely defined
by $\kappa$; if $\kappa \ne \kappa'$ are elements of $\F$ and $\kappa < \kappa'$, then there is a rational number $q$ such that $\kappa < q < \kappa'$.
Whence $q \in \mathbb{Q}^+(\kappa) \cap \mathbb{Q}^-(\kappa')$. It follows that all $\kappa \in \F$ must be fixed by $\beta$.
\end{proof}

\begin{corollary}[\cite{HyperSinger}]
\label{exmp}
 For the positive odd integer $n$ which is at least $3$, $\PG(n - 1,\F)$ cannot admit Singer groups if $\F$ is either the field of real algebraic numbers, the field of computable numbers or the field of (real) definable numbers. In particular, 
$\PG(2,\F)$ has no Singer groups in these cases.
\end{corollary}
\begin{proof}
Each of these fields is real-closed and a subfield of the reals. By the previous theorem, $\Aut(\F)$ is always trivial. The statement then follows from Theorem \ref{struct}.
\end{proof}

\bigskip
\subsection{Possible strategy for classification?}
\label{strat}

One is tempted to study the following property (which we formulate for planes, but which is easily generalized to other spaces):\\

\quad(E)\quad {\em Let $\F$ be a field and $\mathbb{K}\vert\F$ a field extension. 
If $S$ is a Singer group of $\PG(2,\F)$, then $\PG(2,\mathbb{K})$ also allows a Singer group  $S'$ such that $S'_{\vert \F} = S$.}\\

In a category $\mathbf{E}$ of fields for which (E) is true (completed by the appropriate fields), we have:\\

\quad(AC)\quad {\em If $\F$ is an object in $\mathbf{E}$, then $\PG(2,\cup_{\F' \in \mathbf{E}, \F \leq \F'} \F')$ also allows a Singer group.}\\

(Consider an arbitrary filtration
\begin{equation}
\F = \F^{(0)} \leq \F^{(1)} \leq \cdots 
\end{equation}
such that $\cup_{i \in \mathbb{N}}\F^{(i)} = \cup_{\F \leq \F' \in \mathbf{E}}\F'$, and take a direct limit of the induced directed system of Singer groups.)

Let us for instance define a category $\mathbf{E}$ as having as objects a fixed finite field $\F_p$, $p$ a prime, and all finite extensions $\F_{p^i}$ with
$(i,3) = 1$. (Morphisms are natural.) Then by \S \ref{class}, we can construct a canonical Singer group $S(\PG(2,q))$ for each object $\F_q$ in $\mathbf{E}$.
As we will later see, the property $(3,i) = 1$ translates in the fact that if $m$ divides $n$, $m, n \in \mathbb{N} \setminus 3\mathbb{N}$, then 
$S(\PG(2,p^m)) \leq S(\PG(2,p^n))$. Taking the direct limit of the naturally defined directed system of groups, we obtain a Singer group of 
$\PG(2,\cup_{\F \in \mathbf{E}}\F)$. Of course, this specific Singer group can also be obtained directly by using \S \ref{class} (since this 
limit has field extensions of degree $3$). \\

If (E) would be true for a sufficiently large category of field extensions of a fixed field $K$, and $\overline{K}$ is an algebraically closed field for which $\PG(2,\overline{K})$ does not have a Singer group, then 
$\PG(2,K)$ also does not have a Singer group. Unfortunately, even for the category for finite fields (with completions)
(E) is not satisfied, although almost.  Still, all categories $\mathbf{E}$ of field extensions of some fixed field $\F$ that enjoy (E) also enjoy (AC), so that 
Singer groups for the ``$\mathbf{E}$-closures'' also exist.

\bigskip
\subsection{Algebraically closed fields}

In the case of algebraically closed fields, we can say the following.

\begin{theorem}[\cite{HyperSinger}]
\label{torsion}
Let $S$ be a Singer group of $\PG(m,\F)$, $m \in \mathbb{N}^{\times}$, $m + 1$ odd, $\F$ algebraically closed. Then 
$S$ is torsion-free if $\mathrm{char}(\F) \ne 0$. If $\mathrm{char}(\F) = 0$, then $S$ is torsion-free if $m = 2$.  
\end{theorem}

\begin{remark}
{\rm Note that if an involution $\sigma$ of some projective space in characteristic $0$ does not fix any point, it must fix a parallel class of lines elementwise.}
\end{remark}

\medskip
\subsubsection{Singer groups of $\PG(2,\overline{\mathbb{F}_p})$ do not exist}
\label{finac}

Put $\mathbb{N}^\times =: I$, and make the latter into
a directed set, by writing that $n \preceq m$ if $n \vert m$ and $(m/n,3) = 1$. We will first explain the motivation for this definition.  
Let $p$ be a prime, and consider $\F_i := \F_{p^i} \leq \F_j  := \F_{p_j}$ with $i \preceq j \ne i$. Let $\F_j' = \F_j[X]/(f(X))$ be an extension of degree $3$ of $\F_j$
($f(X)$ having degree $3$), and define $\F_i' := \F_i[X]/(f(X))$ | this extension is also of degree $3$, and is a subfield of $\F_j'$. Then $\F_j \cap \F_i' = \F_i$ as 
$(j/i,3) = 1$. We have that 

\begin{equation}
\label{eqcalc}
{\F_i'}^{\times}/\F_i^{\times} = {\F_i'}^{\times}/({\F_i'}^{\times} \cap \F_j^{\times}) \cong  {\F_i'}^{\times}{\F_j^{\times}}/\F_j^{\times} \leq {\F_j'}^{\times}/\F_j^{\times}.
\end{equation}

In other words, we have an inclusion of cyclic groups $C_{p^{2i} + p^i + 1} \leq C_{p^{2j} + p^j + 1}$.\footnote{Notice that this formula gives an easy, calculation-free, proof of the following property:
\begin{lemma}
If $j \equiv 0\mod{i}$, $i$ and $j$ being positive nonzero integers, and $(j/i,3) = 1$, then $p^{2i} + p^i + 1$ divides $p^{2j} + p^j + 1$ for any prime $p$.
\end{lemma}}

This is exactly what we need for having a good definition for the directed system above | for the infinite case, we will use the form of $(\ref{eqcalc})$.
Unfortunately, our system is not directed anymore due to the divisibility constraint: if $i, j \in I$ and $3^n \vert i$ but not $3^{n + 1}$, and $3^m \vert j$ but not 
$3^{m + 1}$ and $n \ne m$, then there is no $k \in I$ such that $i \preceq k$ and $j \preceq k$. (Similar obstructions arise when going to higher dimensions.)

The next theorem explains that we can not adapt the construction.

\begin{theorem}[Nonexistence for the fields $\overline{\F_p}$, \cite{HyperSinger}]
For any prime $p$ and any positive integer $m \geq 1$ with $m + 1$ odd, the space $\PG(m,\overline{\F_p})$ does not admit Singer groups. In particular, this result 
applies to the planar case.
\end{theorem}
\begin{proof}
Let $p, m$ and $\PG(m,\overline{\F_p})$ be as in the statement, and suppose that $S$ is a Singer group. We represent an element $\gamma$ of $S$  by a couple $(A,\sigma)$, where $A \in \GL_{m + 1}(\overline{\F_p})$ and $\sigma \in \Aut(\overline{\F_p})$. If $\sigma = \id$, we know that $\gamma$ has fixed points, so $\sigma \ne \id$ and by Theorem \ref{torsion} we have that $\langle \gamma \rangle \cong \mathbb{Z}, +$. Write $\overline{\F_p}$ as $\cup_{i = 1}^{\infty}\F_{p^i}$. As $A$ has a finite number of entries, and 
as each element of $\overline{\F_p}$ is contained in some finite subfield, there is a finite subfield $\F_q$ such that  $A \in \GL_{m + 1}(q)$. For each $i \in \mathbb{N}^\times$, $\overline{\F_p}$ contains a unique subfield of size $\F_{p^i}$, so $\sigma$ stabilizes all these subfields | in particular, it stabilizes $\F_q$. So $\gamma$ fixes $\PG(m,q) \subseteq
\PG(m,\overline{\F_p})$. As $\langle \gamma \rangle$ is not finite, some power of $\gamma$ fixes points (and even all points) of $\PG(m,q)$, contradiction. (Another way of finishing the proof, knowing that $A \in \GL_{m + 1}(\F_q)$, is to use the remark after Theorem \ref{torsion}.)
\end{proof}

\subsubsection{More on structure}

Recalling the result of Karzel stated in the beginning of this section, we deduce the following now.

\begin{theorem}[\cite{HyperSinger}]
\label{strfree}
Let $\F$ be a field which is not finite.
Let $S$ be a Singer group of $\PG(r,\F)$, $r \in \mathbb{N} \setminus \{ 0,1\}$,
and suppose $X \subseteq S$ is an independent set of generators  of $S$ of minimal size $\omega$.
Then either $\omega = \vert \F\vert$, or $\F$ is countably infinite. If in the latter case $w \ne \vert \F\vert$ (so if $\omega$ is finite), $\F$ is not isomorphic to $\overline{\mathbb{Q}}$.
\end{theorem}

The proof also leads to the following general formulation.

\begin{theorem}[\cite{HyperSinger}]
Let $S$ be a Singer group of $\PG(r,\F)$, $r \in \mathbb{N} \setminus \{ 0,1\}$,
and suppose $X \subseteq S$ is an arbitrary subset. Suppose $X = \{(A_i,\sigma_i) \vert i \in I\}$ ($A_i \in \GL_{r + 1}(\F)$, $\sigma_i \in \Aut(\F)$).
Let $\Omega$ and $\Sigma$ be as above, and let $\F(X)$ be the subfield of $\F$ generated by $\mathbb{P}(\Omega)^\Sigma$, where $\mathbb{P}$ is the prime field 
of $\F$. Then the subspace $\PG(r,\F(X))$ is stabilized by $\langle X \rangle$.
\end{theorem}

\medskip
\subsubsection{Infinite dimension}

When $\omega$ is any infinite cardinal number, then for $\PG(\omega,\F)$, with $\F$ either real-closed or algebraically closed, there {\em are} Singer groups.
For,  let $\chi$ be a set of indeterminates such that $\F(\chi)/\F$ has degree $\omega$.
Then as in \S \ref{class}, $\F(\chi)^\times$ acts naturally, linearly and sharply transitively on the vector space $\F^\omega$. Passing to the 
corresponding projective spave yields the construction.\\

\bigskip
\subsection{Virtual Singer groups and virtual fields}

Let $\Gamma$ be a projective plane, and $Y \leq \Aut(\Gamma)$. We say that $Y$ is {\em virtually Singer} (or $Y$ is a {\em virtual Singer group}\index{virtual!Singer group}) if $Y$ acts 
freely on the points of the plane, and if the number of $Y$-orbits on the point set is finite. (One could also relax this condition by asking that 
the cardinality of $Y$ and of the point set are the same.) We say that $\Gamma$ is {\em virtually Singer}\index{virtual!Singer plane} if $\Gamma$ admits a virtual Singer group.
We also say that a field $\K$ is {\em (virtually) Singer}\index{virtual!Singer field}\index{Singer!field} in {\em degree}\index{degree} $m$ ($m \in \mathbb{N}^\times$) if the space $\PG(m,\K)$ is (virtually) Singer. If we do not 
mention the degree, we mean the planar case.\\

\begin{theorem}[Examples, \cite{HyperSinger}]
We have the following for virtual Singer groups.
\begin{itemize}
\item
All Singer groups are virtually Singer.
\item
All groups acting freely on finite planes are virtually Singer. 
\item
All planes $\PG(2,\F)$ with $\F$ not real-closed or algebraically closed are virtually Singer.
\item
No plane of the statement of Corollary \ref{exmp} can be virtually Singer.
\end{itemize}
\end{theorem}

Two questions which arise are: (1) {\em do there exist real-closed or algebraically closed fields which are virtually Singer ?}; (2) {\em do there exist real-closed or algebraically closed fields which are virtually Singer but not Singer ?}\\

\bigskip
\subsection{Singer groups of $\mathbb{F}_1^m$-spaces}

Let $\mathbf{P}$ be an $m$-dimensional projective space over $\mathbb{F}_{1^n}$, where $(n,m) \in \mathbb{N}^\times \times \mathbb{N}^\times$. It is a set 
of $m + 1$ sets $X_i$ of size $n$, each endowed with a free action of the multiplicative group $\mu_n^i \cong C_n$, together with the induced subspace structure.
The linear automorphism group of $\mathbf{P}$  is $\mathbf{S}_{m + 1} \wr (C_n)^{m + 1}$ (elements consist of $(m + 1)\times(m + 1)$-matrices with only one nonzero
entry per row and column, and each such entry is an element of $C_n$). 

\subsubsection{First construction}

It is clear that once we fix a sharply transitive subgroup $S$ of $\mathbf{S}_{m + 1}$, 
we can construct a Singer group $S(n)$ of $\P$ by taking the direct product with a diagonal group $\langle (\nu_1,\nu_2,\ldots,\nu_{m + 1}) \rangle \cong C_n = \langle \nu \rangle$, 
where each $\nu_i$ is a copy of $\nu$ (acting on $X_i$), all $X_j$ being fixed. Now note that 

\begin{equation}
i\ \ \vert\ \ j \ \ \mbox{implies}\ \ S(i) \leq S(j). 
\end{equation}
(For each $S(j)$, we use the same group $S$.) So we obtain a directed system of Singer groups of projective $m$-spaces over all finite extensions of $\mathbb{F}_1$,
and after passing to the limit we obtain a Singer group $S(\infty)$ of $\PG(m,\overline{\Fun})$.

As we fix $S$ in the construction above, passing to the direct limit amounts to constructing $\overline{\Fun}$ by taking unions of all the cyclic groups $C_n$ in the 
appropriate way. As we have seen, there is no way this approach can be pulled to other finite fields (due to the fact that at the $\Fun$-level, we cannot see
the extra needed divisibility condition ``$(j/i,3) = 1$'').

\subsubsection{General construction}

There is a more generic construction, which in fact captures all possible Singer groups of the spaces $\PG(m,\Fun^d)$.
Let $\mathbf{P}$ be an abstract set (say suggestively of $m + 1$ elements with $m \in \mathbb{N}^\times$), and let $S \leq \mathrm{Sym}(\mathbf{P})$ be a transitive group.
We require that
\begin{itemize}
\item[(CY)] For some element $x \in \mathbf{P}$ (and then all elements), $S_x$ is cyclic. 
\end{itemize}
Let $S_x \cong C_n$; then clearly $S$ has a natural action as a Singer group on $\PG(m,\Fun^n)$, and all Singer groups of this space arise in this way.

\newpage

\newpage
\printindex
\newpage\mbox{}\newpage
%
%
%





\end{document}